%% file: CFO.tex
\documentclass[11pt,a4paper]{amsart}

\usepackage{amssymb,amsmath,amscd,amsthm,enumerate,verbatim,xcolor,url}
\usepackage[utf8]{inputenc}
\usepackage{mathrsfs}
\usepackage{cite} 
\usepackage{bbm}
\usepackage{stmaryrd}
\usepackage[margin=1.1in]{geometry}
\usepackage[T1]{fontenc}
\usepackage[
	colorlinks,
	linkcolor={black},
	citecolor={black},
	urlcolor={black}
]{hyperref}
\usepackage{enumitem}
\usepackage{bm}	
\usepackage{tikz}\usetikzlibrary{positioning,patterns,shapes,arrows,calc}

\usepackage{array}

\newcommand{\widecheck}{\check}

\DeclareMathOperator\tr{Tr}
\newcommand{\aubrydual}{{\sharp}}
\newcommand{\realified}{{\mathscr{R}}}
\newcommand{\costwopi}{{\mathrm{c}}}
\newcommand{\sintwopi}{{\mathrm{s}}}

\definecolor{purple}{rgb}{.5,0,1}
\definecolor{orange}{rgb}{1,.5,0}
\definecolor{pink}{rgb}{1,0,.5}
\definecolor{green}{rgb}{0,.5,0}
\definecolor{gold}{rgb}{1,.623,0}
\def\red#1{\textcolor{red}{#1}}
\def\blue#1{\textcolor{blue}{#1}}
\def\green#1{\textcolor{green}{#1}}

\newcommand{\bbC}{{\mathbb{C}}}
\newcommand{\bbD}{{\mathbb{D}}}

\newcommand{\bbG}{{\mathbb{G}}}
\newcommand{\bbL}{{\mathbb{L}}}
\newcommand{\bbN}{{\mathbb{N}}}
\newcommand{\bbO}{{\mathbb{O}}}
\newcommand{\bbQ}{{\mathbb{Q}}}
\newcommand{\bbR}{{\mathbb{R}}}
\newcommand{\bbS}{{\mathbb{S}}}
\newcommand{\bbT}{{\mathbb{T}}}
\newcommand{\bbU}{{\mathbb{U}}}
\newcommand{\bbZ}{{\mathbb{Z}}}
\newcommand{\calA}{{\mathcal{A}}}

\newcommand{\calE}{{\mathcal{E}}}

\newcommand{\calL}{{\mathcal{L}}}
\newcommand{\calM}{{\mathcal{M}}}
\newcommand{\calX}{{\mathcal{X}}}
\newcommand{\scrB}{{\mathscr{B}}}
\newcommand{\scrH}{{\mathscr{H}}}
\newcommand{\scrZ}{{\mathscr{Z}}}
\newcommand{\vecn}{{\underline{n}}}

\renewcommand{\Im}{{\mathrm{Im}}}
\renewcommand{\Re}{{\mathrm{Re}}}

\newcommand{\GL}{{\bbG\bbL}}
\newcommand{\SL}{{\bbS\bbL}}
\newcommand{\SO}{{\bbS\bbO}}
\newcommand{\SU}{{\bbS\bbU}}

\newtheorem{theorem}{Theorem}[section]
\newtheorem{lemma}[theorem]{Lemma}

\newtheorem{prop}[theorem]{Proposition}
\newtheorem{coro}[theorem]{Corollary}

\theoremstyle{definition}
\newtheorem{remark}[theorem]{Remark}
\newtheorem{definition}[theorem]{Definition}

\newcommand{\set}[1]{{\left\{ #1 \right\}}}

\newcommand{\spr}{{\mathrm{spr}}}

\makeatletter
\def\subsection{\@startsection{subsection}{2}%
	\z@{.5\linespacing\@plus.7\linespacing}{.5\linespacing}%
	{\normalfont\scshape\centering}}
\makeatother

\numberwithin{equation}{section}
\hypersetup{
	pdftitle = {Almost Everything About the Unitary Almost Mathieu Operator},
	pdfauthor = {Christopher Cedzich,
		Jake Fillman,
		Darren Ong},
	pdfsubject = {Quantum walks},
	pdfkeywords = {quasi-periodic cocycle, quantum walk, CMV matrix, GECMV matrix, almost-Mathieu operator, unitary dynamics, Aubry duality, global theory, spectral theory, spectral type, spectral characterization, dynamical systems, localization-delocalization transition, Anderson localization, absolutely continuous spectrum, singular continuous spectrum, pure point spectrum, Cantor spectrum
	}
}

\title[Unitary Almost Mathieu Operator]{Almost Everything About the \\ Unitary Almost Mathieu Operator}

\author[C.\ Cedzich]{Christopher Cedzich}
\email{\href{mailto:cedzich@hhu.de}{cedzich@hhu.de}}
\address{Quantum Technology Group, Heinrich Heine Universit\"at D\"usseldorf, Universit\"atsstr. 1, 40225 D\"usseldorf, Germany}

\author[J.\ Fillman]{Jake Fillman}
\email{\href{mailto:fillman@txstate.edu}{fillman@txstate.edu}}
\address{Department of Mathematics, Texas State University, San Marcos, TX 78666, USA}

\author[D.\ C.\ Ong]{Darren C.\ Ong}
\email{\href{mailto:darrenong@xmu.edu.my}{darrenong@xmu.edu.my}}
\address{Department of Mathematics and Applied Mathematics, Xiamen University Malaysia, Jalan Sunsuria, Bandar Sunsuria, 43900 Selangor, Malaysia and School of Mathematical Sciences, Xiamen University, 361005 Xiamen, Fujian, China}

\begin{document}

\begin{abstract}
We introduce the unitary almost-Mathieu operator, which is obtained from a two-dimensional quantum walk in a uniform magnetic field. We exhibit a version of Aubry--Andr\'{e} duality for this model, which partitions the parameter space into three regions: a supercritical region and a subcritical region that are dual to one another, and a critical regime that is self-dual. In each parameter region, we characterize the cocycle dynamics of the transfer matrix cocycle generated by the associated generalized eigenvalue equation. In particular, we show that supercritical, critical, and subcritical behavior all occur in this model. Using Avila's global theory of one-frequency cocycles, we exactly compute the Lyapunov exponent on the spectrum in terms of the given parameters. We also characterize the spectral type for each value of the coupling constant, almost every frequency, and almost every phase. Namely, we show that for almost every frequency and every phase the spectral type is purely absolutely continuous in the subcritical region, pure point in the supercritical region, and purely singular continuous in the critical region. In some parameter regions, we refine the almost-sure results. In the critical case for instance, we show that the spectrum is a Cantor set of zero Lebesgue measure for arbitrary irrational frequency and that the spectrum is purely singular continuous for all but countably many phases.

\end{abstract}

\maketitle

\setcounter{tocdepth}{1}
\tableofcontents

\hypersetup{
	linkcolor={black!30!blue},
	citecolor={black!30!blue},
	urlcolor={black!30!blue}
}

\input{01-intro.tex}

\input{02-modelanddefs.tex}

\input{03-motivation.tex}

\input{04-cocycle.tex}

\input{05-aubryduality.tex}

\input{06-contspec.tex}

\input{07-zeromeas.tex}

\input{08-localization.tex}

\bibliographystyle{abbrvArXiv}

\bibliography{non-iso-bib}

\end{document}

%% file: 01-intro.tex

\section{Introduction}
Discrete-time quantum walks have been extensively studied over the years from different points of view. They have intrinsic mathematical interest and are important in physics, where they serve as models for the quantum evolution of single particles with internal degrees of freedom and bounded hopping length on a lattice or graph. In this context, quantum walks were shown to exhibit many single and few particle effects, such as ballistic motion \cite{ambainis2001one, grimmett2004weak, AVWW2011JMP}, decoherence \cite{AVWW2011JMP, SpacetimeRandom}, dynamical localization \cite{joye_d_dim_loc, ASW2011JMP, Joye_Merkli}, and the formation of bound states \cite{molecules}. Recently, quantum walks were shown to provide a testbed for symmetry protected topological order \cite{KitaExploring} and bulk-boundary correspondence with a complete set of indices that are stable under compact as well as norm-continuous perturbations \cite{TopClass, Ti, WeAreSchur, F2W}. Complementary to this, quantum walks can be viewed as quantum mechanical analogues of classical random walks \cite{AhaDavZaq1993PRA}. Their increased ballistic spreading compared to classical diffusion makes them interesting for applications in quantum computing such as the element distinctiveness problem \cite{Ambainis}, universal quantum computation \cite{Lovett:2010ff}, and search algorithms \cite{SKW,portugal2013quantum}. 

From a mathematical point of view, a quantum walk is given by a unitary operator $W$ acting in a suitable Hilbert space with an underlying lattice or graph structure. As such, the long-time dynamical characteristics of the system are given by powers $W^t$ with $t \in \bbZ$, and one is naturally led to the study of the spectral problem of the unitary operator $W$. See \cite{BGVW, CGMV2010CPAM, DEFHV, DFO2016JMPA, AlainUnitaryBandMats, ewalks, morioka2019detection, locQuasiPer, KKMS2021} and references therein for a partial list of works that have analyzed quantum walks from the spectral perspective. 

The present paper aims to study the spectral properties of quantum walks in a homogeneous external magnetic field. We are motivated by work on electrons moving in a two-dimensional lattice under the influence of a uniform magnetic field, a physical system that has been intensively studied in recent decades; compare, e.g., \cite{hof76, Laughlin, TKNN}. The almost-Mathieu operator (and more generally the extended Harper's model) is one example of a model of an electron in a magnetic field, which inspired intense study in mathematics; see 
\cite{AJ2009Ann, AviJitoMarx2017Invent, AviYouZho2017DMJ,  BelSim1982JFA,  Han2018IMRN, Jitomirskaya1999Annals, JitoLiu2018Annals, MarxJito2017ETDS, Puig2004CMP} and references therein for a partial list. Despite the ostensible simplicity of the setup, the properties of this system are rich and deep, requiring extensive analysis to understand fully.

Motivated by this, we discuss in this paper two-dimensional quantum walks in homogeneous magnetic fields, which are constructed via a discrete analogue of minimal coupling \cite{CGWW2019JMP}. After a suitable choice of gauge, this leads us to the investigation of one-dimensional quantum walks with quasiperiodic coin distribution. Using intuition from the two-dimensional setting, we introduce a pair of coupling constants $\lambda_1,\lambda_2 \in [0,1]$ for the one-dimensional walk in a novel manner (cf.\ \eqref{eq:moddefs:slambdaDef} and \eqref{eq:moddefs:coinDef}). We then show that this partitions the parameter space into three regions. When $\lambda_1>\lambda_2$, the behavior of the shift dominates and the resultant operator has absolutely continuous spectrum. When $\lambda_1<\lambda_2$, the coin operator dominates, and one observes localization for typical frequencies. When $\lambda_1=\lambda_2$, one observes singular continuous spectrum for all but an explicit countable set of phases. We describe the main accomplishments broadly here and more precisely in Section~\ref{sec:modelanddefs}.

First, we give a complete account of the spectrum and the spectral type in a measure-theoretical sense (that is for a.e.\ frequency and a.e.\ phase), hence the title, which is naturally inspired by the titles of \cite{Damanik2009surv, Jitomirskaya1995ICMP, Last1995ICMP}.

Next, we answer a question posed by and to the authors of \cite{FOZ2017CMP}: how can one suitably incorporate a version of a ``coupling constant'' into the model of \cite{FOZ2017CMP} in such a manner that one sees a phase transition? Their model is given as a one-dimensional quantum walk, and each such walk is known to be equivalent to a CMV matrix \cite{CGMV2010CPAM,Simon2005OPUC1, Simon2005OPUC2} which in turn is specified by a sequence $\{\alpha_n\}_{n=-\infty}^\infty$ of \emph{Verblunsky coefficients} in $\bbD$. In the CMV setting, it seems natural to introduce a coupling constant by multiplying the Verblunsky coefficients by a scalar $\lambda$; however, this does not lead to fruitful results in the present setting. See also the discussion in \cite[Appendix~B]{Simon2005OPUC2}, especially item~14 in Section~B.1. Indeed, Simon says therein that the lack of suitable coupling constants ``impacted my ability to discuss almost periodic Verblunsky coefficients'', which suggests that our novel method of incorporating a coupling constant will have applications in other models.

Let us also point out: as a result of this connection, every result in our paper can be translated into a result about a suitable CMV matrix. For the reader who is unfamiliar with OPUC and CMV matrices, we have included a detailed description of the connection between CMV matrices and one-dimensional quantum walks. We have chosen to formulate our work in terms of quantum walks, since that makes the origin of the model more transparent.

The first main ingredient in the proofs of our results is the classification of cocycle dynamics: in the nomenclature of Avila, we classify the matrix cocycle induced by the stationary equation of the walk (cf.\ \eqref{eq:AlambdaCocycleDef}) and show that it is subcritical on the spectrum when $\lambda_1>\lambda_2$, supercritical on the spectrum when $\lambda_1<\lambda_2$, and critical on the spectrum when $\lambda_1=\lambda_2$. This discussion sweeps a few technicalities under the rug; the transfer matrix cocycle is not always well-defined for some values of the parameters and hence one works with a suitable regularization whose dynamics one can then characterize in the global theory scheme.

The second key ingredient is a suitable version of duality via a Fourier-type transform (cf.\ Theorems~\ref{t:aubryViaSolutions} and \ref{t:aubryOperator}); in analogy with the setting of Harper's model, we call this ``Aubry duality'' \cite{AubryAndre1980}. Up to a transpose, the duality transformation exchanges the coupling constants $\lambda_1$ and $\lambda_2$, which enables one to translate suitable reducibility or localization statements from one parameter region to another. In particular, the operators described herein are quasiperiodic unitary operators that obey a version of Aubry duality, and thus this work addresses a query from Li--Damanik--Zhou  \cite{LiDamZhou2021Preprint}.  The form of duality enables one to show that all three spectral types (absolutely continuous, singular continuous, and pure point) occur within the non self-dual regimes. 

As a byproduct of our model and the detailed cocycle analysis, we produce a family of quasiperiodic unitaries with positive Lyapunov exponents in certain parameter regions. As noted by Zhang \cite{Zhang2012Nonlin}, this is often difficult to accomplish. See also \cite{locQuasiPer, ewalks}.

\begin{figure}[h]
	\begin{center}
		\begin{tabular}{>{\raggedright\arraybackslash}p{.45\linewidth}>{\raggedleft\arraybackslash}p{.45\linewidth}}
			\includegraphics[width=.9\linewidth]{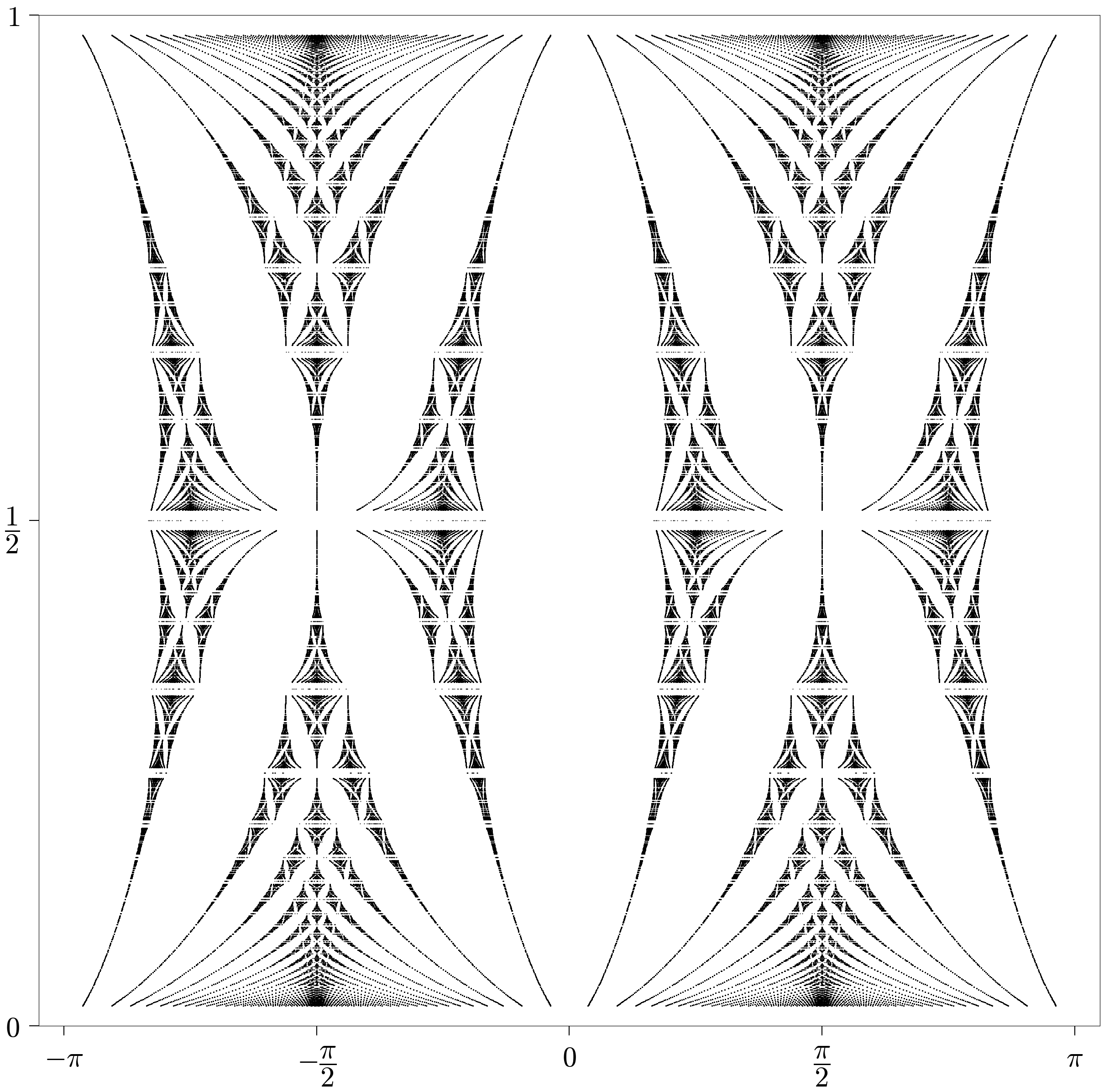}
			&
			\includegraphics[width=.9\linewidth]{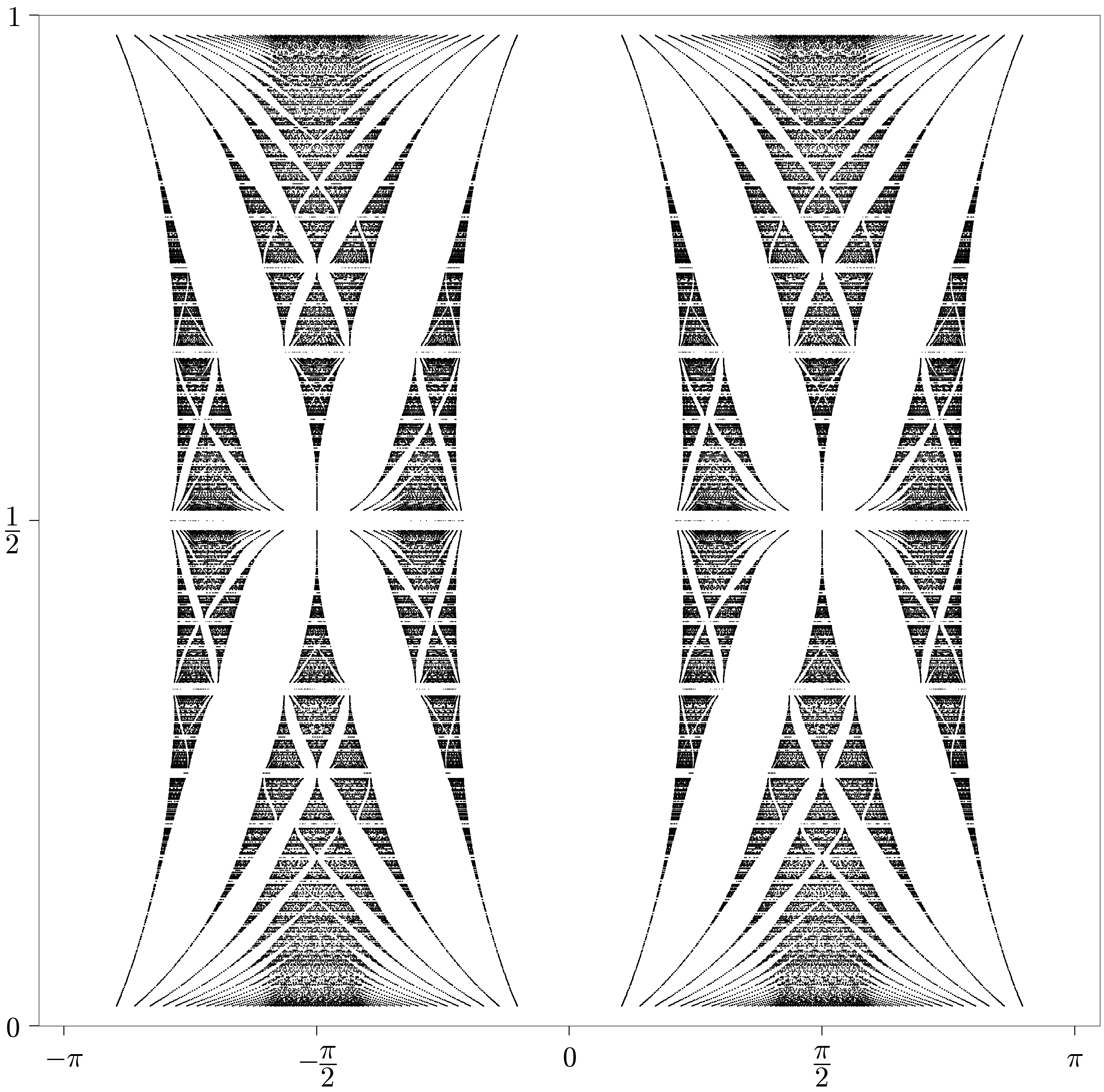}
		\end{tabular}
	\end{center}
	\caption{Spectrum of the UAMO for different (rational) $\Phi$ on the vertical axis and argument of the spectral parameter $z$ on the horizontal axis. Left: Critical case with $(\lambda_1,\lambda_2)=(1/\sqrt2,1/\sqrt2)$. Clearly visible is the fractal structure, which is reminiscent of (a quenched variant of) the Hofstadter butterfly \cite{hof76,papillon} (see \cite{CFGW2020LMP} for the spectrum for $(\lambda_1,\lambda_2)=(1,1)$). Right: Sub-critical case with $(\lambda_1,\lambda_2)=(1/\sqrt4,1/\sqrt2)$. The spectrum of the super-critical case $(\lambda_1,\lambda_2)=(1/\sqrt2,1/\sqrt4)$ is the same (see Corollary \ref{cor:aubryspecra}), but the spectral type is different. }
\end{figure}

It is worth emphasizing that there are two separate categories of challenges that had to be overcome for this work to be written. 
The first and most substantial challenge is to choose the correct model.
At the time when \cite{FOZ2017CMP} was written, it was not clear how to correctly incorporate coupling constants into the coins, and moreover, the naive guesses turned out to be incorrect.
After \cite{CFGW2020LMP}, one then was able to disentangle the correct choices from the physical model.

The second challenge is to suitably recontextualize techniques and ideas from the self-adjoint setting to the unitary setting.
This is sometimes easy, sometimes challenging, and sometimes impossible.
In the present case, the techniques fell into the first two classes. Our main goal was to present a complete picture in the sense that one observes all three types of cocycle dynamics, precise formulations of duality, and all three spectral types in suitable parameter regions. In order to achieve this while still keeping the length of the paper in check, we have aimed to spend less time on techniques and tools that are obvious generalizations from the self-adjoint setting in order to supply full details for problems that present genuine challenges and required more novel solutions.

Let us describe a few of the challenges. First, due to singularities that arise at complex phases in the non-critical cases, the transfer matrix cocycle is singular and hence one cannot immediately use the ``asymptotic cocycle'' method as in previous cases to compute the Lyapunov exponent. We use a topological argument to circumvent this difficulty: we show that the range of the determinant of the cocycle map has trivial winding around the origin in the two ``perturbative'' regions of large and small imaginary parts and use suitable continuity arguments to establish this for all imaginary parts away from the singularities. This is inspired by related works of Jitmirskaya--Marx for the (self-adjoint) extended Harper's model \cite{JitoMarx2012CMP}. 

 The calculations needed to establish the form of Aubry duality are quite delicate and require several non-obvious changes of coordinates. More generally, because of the more complicated structure of the operator and its transfer matrices, most calculations are more involved than the corresponding self-adjoint counterparts. Furthermore, as mentioned before, it is also not obvious what these analogous calculations should be.
 
 In the proof of continuous spectrum for the case of critical coupling, we need to deal with the lack of a symplectic symmetry of the cocycle. Concretely, we had to find a novel symmetry of the cocycle which was then different from the symplectic symmetry that is typically used in studying reducibility questions for Schr\"odinger cocycles. In the Schr\"odinger setting, one can simply exploit the fact that the transfer matrix is real and so invariant under complex conjugation, whereas the corresponding symmetry we had to use was more complicated. 

Moreover, in several of the arguments, it is needed to relate the transfer matrix cocycle to one in $\SL(2,\bbR)$ in order to apply the machinery of such cocycles. However, even after normalization, the cocycle for this model clearly does not belong to $\SL(2,\bbR)$; nevertheless, we found a novel conjugacy from the cocycle into $\SL(2,\bbR)$, which again is not at all obvious from the form of the cocycle.

Finally, our operators are not standard CMV matrices but rather \emph{generalized CMV matrices}, which allow for an additional complex phase in some parameters. This obstacle is the easiest to overcome: any generalized CMV matrix is unitarily equivalent to a genuine CMV matrix by a simple diagonal unitary operator. This allows one to translate results from CMV matrices to generalized CMV matrices with relative ease. For the reader's convenience, we describe this equivalence in Section~\ref{sec:CMV}. This will be discussed more thoroughly together with some additional applications in the forthcoming work \cite{CFLOZ}.

We wish to emphasize that, although there are often corresponding results on either the unitary or self-adjoint side, such results are often about \emph{classes} of operators (periodic, random, quasiperiodic), and not \emph{specific} models. That is to say, there is to our knowledge no dictionary that explicitly connects properties of Schr\"odinger operators  with CMV matrices or quantum walks in an explicit fashion.

The remainder of the paper is structured as follows. We defined the model and state our results in Section~\ref{sec:modelanddefs}. Section~\ref{sec:origins} discusses the two-dimensional quantum walks from which the main model of the paper is derived. In Section~\ref{sec:cocycle}, we discuss the associated transfer matrix cocycle in detail, characterizing the cocycle dynamics according to Avila's global theory in each parameter region. Section~\ref{sec:aubry} discusses and proves several manifestations of Aubry duality for the model, which is then put to use in Sections~\ref{sec:contspec}--\ref{sec:localization} to prove the remainder of our spectral results.

\bigskip

\subsection*{Acknowledgements} D.~C.~O.\ was supported in part by three grants from the Fundamental Research Grant Scheme from the Malaysian Ministry of Education (grant numbers  FRGS/1/2022/TK07/XMU/01/1, FRGS/1/2018/STG06/XMU/02/1 and FRGS/1/2020/STG06/XMU/02/1), a grant from the National Natural Science Foundation of China (grant number 12201524), and two Xiamen University Malaysia Research Funds (grant numbers XMUMRF/2023C11/
IMAT/0024 and XMUMRF/2020-C5/IMAT/0011). C.~C.\ was supported in part by the Deutsche Forschungsgemeinschaft (DFG, German Research Foundation) under the grant number 441423094. J.~F.\ was supported in part by National Science Foundation grant DMS 2213196 and Simons Foundation Collaboration Grant \#711663. The authors are grateful to Chris Marx for helpful conversations and to Luis Vel\'{a}zquez for patiently answering questions related to CMV matrices. The authors are grateful to Svetlana Jitomirskaya and Qi Zhou for helpful comments on an earlier draft and to the anonymous reviewers for helpful comments on this version.
\subsection*{Data Availability Statement}  Data sharing is not applicable to this article as no datasets were generated or analysed during the current study.

%% file: 02-modelanddefs.tex

\section{Model and Results} \label{sec:modelanddefs}

\subsection{Setting}

We will consider a quantum walk $W:\ell^2(\bbZ) \otimes \bbC^2 \to \ell^2(\bbZ) \otimes \bbC^2 =: \scrH_1$, defined by a quasiperiodic sequence of coins. 
In the present section, we will define precisely the model, which arises via a one-dimensional representation of a two-dimensional quantum walk in a magnetic field. 
However, we emphasize that the one-dimensional model can be defined and studied without reference to the two-dimensional model.  We explain the connection to two-dimensional magnetic quantum walks in detail in Section~\ref{sec:origins}.

Let us write the standard basis of $\scrH_1$ as
\begin{equation}
 \delta_n^s = \delta_n \otimes e_s, \quad n \in \bbZ, \, s \in \{+,-\},
 \end{equation}
where $\{\delta_n:n \in \bbZ\}$ denotes the standard basis of $\ell^2(\bbZ)$ and $e_+ = [1,0]^\top$, $e_-=[0,1]^\top$ denotes the standard basis of $\bbC^2$. For an element $\psi \in \scrH_1$, we write its coordinates as $\psi_n^s = \langle \delta_n^s, \psi \rangle$ so that 
\[\psi = \sum_{n\in\bbZ}\left( \psi_n^+ \delta_n^+ + \psi_n^- \delta_n^- \right),\]
where $\psi^+,\psi^- \in \ell^2(\bbZ)$. Viewing $\scrH_1 \cong \ell^2(\bbZ,\bbC^2)$, it is natural to also put $\psi_n = [\psi_n^+, \psi_n^-]^\top$. Recall that a quantum walk $W$ is specified by the iteration of a suitable unitary operator. Choosing a sequence of unitary coins
\begin{equation}
Q_n = \begin{bmatrix} q_{n}^{11} & q_n^{12} \\ q_{n}^{21} & q_n^{22} \end{bmatrix} \in \bbU(2,\bbC),
\end{equation}
and a parameter $\lambda \in [0,1]$, the walks we consider are given by 
 \begin{equation} \label{eq:moddefs:walkDef}
 W = S_{\lambda}	Q,
 \end{equation}
where $Q$ acts coordinatewise via $Q_n$ (i.e., $[Q\psi]_n =Q_n\psi_n$) and $S_\lambda$ is given by
\begin{equation}\label{eq:moddefs:slambdaDef}
S_\lambda\delta_n^\pm = \lambda \delta_{n\pm 1}^\pm \pm \lambda' \delta_n^\mp, \quad  \lambda' = \sqrt{1-\lambda^2}.\end{equation}
The reader may readily verify that $S_\lambda$, $Q$, and $W$ are all unitary.
We call a walk of the form \eqref{eq:moddefs:walkDef} a \emph{split-step walk} with \emph{coupling constant} $\lambda$.

We are now ready to define our main model. Given coupling constants $\lambda_1,\lambda_2 \in [0,1]$, a frequency $\Phi \in \bbT:= \bbR/\bbZ$, and a phase $\theta \in \bbT$, we consider the walk 
\begin{equation}
	W_{\lambda_1,\lambda_2,\Phi,\theta} = S_{\lambda_1}Q_{\lambda_2,\Phi,\theta}
\end{equation}
where the coins $Q_n = Q_{\lambda_2,\Phi,\theta,n}$ are given by
\begin{align} \label{eq:moddefs:coinDef}
 Q_n = Q_{\lambda_2,\Phi,\theta,n} = 
 \begin{bmatrix}
\lambda_2 \cos(2\pi(n\Phi+\theta)) + i\lambda_2'
& -\lambda_2 \sin(2\pi(n\Phi+\theta))\\ 
 \lambda_2\sin(2\pi(n\Phi+\theta)) 
 & \lambda_2 \cos(2\pi(n\Phi+\theta)) - i\lambda_2'
\end{bmatrix}.
\end{align}
By an argument using minimality of $\theta \mapsto \theta+\Phi$ and strong operator convergence, there is for each $\lambda_1,\lambda_2$ and irrational $\Phi$ a fixed set $\Sigma_{\lambda_1,\lambda_2,\Phi}\subseteq \partial \bbD$ with 
\[\sigma(W_{\lambda_1,\lambda_2,\Phi,\theta}) = \Sigma_{\lambda_1,\lambda_2, \Phi} \text{ for every } \theta \in \bbT.\]
Working out the relevant product, the reader can check that the action of $W = W_{\lambda_1,\lambda_2,\Phi,\theta}$ in coordinates is given as
	\begin{align*} 
		[W\psi]_n^+ & = \lambda_1 \left((\lambda_2 \cos(2\pi((n-1)\Phi+\theta)) + i\lambda_2')\psi_{n-1}^+ -\lambda_2 \sin(2\pi((n-1)\Phi+\theta))\psi_{n-1}^-\right)  \\
		& \qquad   		-\lambda_1' \left( \lambda_2\sin(2\pi(n\Phi+\theta)) \psi_{n}^+ + (\lambda_2 \cos(2\pi(n\Phi+\theta)) - i\lambda_2')\psi_{n}^-\right),\\
		[W\psi]_n^- & = \lambda_1 \left( \lambda_2\sin(2\pi((n+1)\Phi+\theta)) \psi_{n+1}^+ + (\lambda_2 \cos(2\pi((n+1)\Phi+\theta)) - i\lambda_2')\psi_{n+1}^-\right) \\
		& \qquad + \lambda_1' \left((\lambda_2 \cos(2\pi(n\Phi+\theta)) + i\lambda_2')\psi_{n}^+  -\lambda_2 \sin(2\pi(n\Phi+\theta))\psi_{n}^-\right).
	\end{align*}
	See also Lemma~\ref{lem:w}. 
	
	For later use, notice that
\begin{equation} \label{eq:detQn=1}
\det(Q_{\lambda_2,\Phi, \theta, n}) = 1 \quad \forall \ \lambda_2 \in [0,1], \ \Phi,\theta \in \bbT, \ n \in \bbZ.
\end{equation}

 Because of the close parallels between this model and the almost-Mathieu operator (AMO), we propose calling $W_{\lambda_1,\lambda_2,\Phi,\theta}$ (an instance of) the \emph{unitary almost-Mathieu operator} (UAMO).
 The case $\lambda_1 = \lambda_2 = 1$ was studied in \cite{FOZ2017CMP,CFGW2020LMP,Shikano:2010id,Linden2009}.
 
\begin{remark}
Let us make a few remarks about the parameters.
\begin{enumerate}[label={\rm (\alph*)}]
\item Since $\sqrt{1-(-\lambda)^2} = \sqrt{1-\lambda^2}$, one has
\begin{equation*}
Q_{-\lambda_2,\Phi,\theta} = Q_{\lambda_2,\Phi,\theta+\frac{1}{2}}.
\end{equation*}
 Moreover, if $\mathcal{U}$ denotes the unitary transformation $\delta_n^\pm \mapsto \pm \delta_{-n}^\mp$, one can check that
\begin{equation*}
\mathcal{U} S_{\lambda_1}   \mathcal{U} = -S_{-\lambda_1} \quad 
\mathcal{U} Q_{\lambda_2,\Phi,\theta} \,  \mathcal{U} = -\overline{Q_{\lambda_2,-\Phi,\theta}}.
\end{equation*}
Thus, we take $\lambda_1,\lambda_2 \in [0,1]$ and do not consider $\lambda_j \in [-1,0)$.
\smallskip

\item  If $\Phi$ is rational, then $\{Q_n\}$ is periodic and hence its spectral \cite{Simon2005OPUC2} and dynamical \cite{AVWW2011JMP, DFO2016JMPA} properties are well-understood.\smallskip

\item  If $\lambda_1 = 0$, then $S_0\delta^\pm_n = \pm \delta_n^\mp$, so the spectrum of $W$ is pure point and all walkers are localized for trivial reasons. 
On the other hand, if $\lambda_2=0$ the coin sequence is constant. In analogy with the self-adjoint setting, in which one considers $\Delta+\lambda V$ with $\lambda$ a real coupling constant, the quantity $\lambda_2/\lambda_1$ appears to play a role similar to $\lambda$. In fact, we will see later that the most appropriate analogue of the coupling constant $\lambda$ appears to be
\begin{equation}\label{eq:model:lambda0Def} 
\lambda_0 = \lambda_0(\lambda_1,\lambda_2): = \frac{\lambda_2(1+\sqrt{1-\lambda_1^2})}{\lambda_1(1+\sqrt{1-\lambda_2^2})} =  \frac{\lambda_2(1+\lambda_1')}{\lambda_1(1+\lambda_2')}, 
\end{equation}
since we will show (Corollary~\ref{coro:lyapExact}) that the Lyapunov exponent is given by $\max\{0,\log\lambda_0\}$ on the spectrum.
 In particular, $\lambda_1=0$ corresponds to $\lambda_0 = \infty$, while $\lambda_2 = 0$ corresponds to setting $\lambda_0 = 0$.\smallskip

\item  In fact, the connection to the coupling constant of the almost-Mathieu operator is even stronger than the previous remark suggests. To wit: we will see that the form of Aubry duality for the present walk (modulo some niceties that we will discuss more precisely later in the paper) acts by exchanging $\lambda_1$ and $\lambda_2$, similar to the duality transformation $\lambda \mapsto 1/\lambda$ in the case of the AMO (compare \cite{AubryAndre1980, MarxJito2017ETDS}). Thus, we have three interesting regimes to study:

\definecolor{myblue}{RGB}{100,100,220}
\definecolor{myred}{RGB}{255,100,100}
\definecolor{mygreen}{RGB}{119,221,119}
\definecolor{myyellow}{RGB}{253,253,150}
\definecolor{myorange}{RGB}{255,200,100}
\def\legendwidth{1.75}
\def\legendhight{0.32}
\def\legendposx{1.7}
\def\legendposy{.5}
\def\circlerad{.8ex}
\def\circleshift{-.7ex}
\newcommand\Square[1]{+(-#1,-#1) rectangle +(#1,#1)}
\def\colorsquare#1{\tikz[baseline=\circleshift]\draw[#1,fill=#1,rotate=45] (0,0) \Square{\circlerad};}

\smallskip
\begin{minipage}{.6\textwidth}
\begin{description}
\item[\colorsquare{myblue}\ \ Subcritical] $\lambda_1>\lambda_2$ (equivalently, $\lambda_0 < 1$)\smallskip
\item[\colorsquare{myred}\ \ Critical] $\lambda_1=\lambda_2$ (equivalently, $\lambda_0=1$)\smallskip
\item[\colorsquare{myorange}\ \ Supercritical] $\lambda_1 < \lambda_2$ (equivalently, $\lambda_0 > 1$)
\end{description}
\end{minipage}
\begin{minipage}{.2\textwidth}
	\begin{tikzpicture}
		[
		scale=1.5,
		font=\footnotesize
		]
		
		\definecolor{myblue}{RGB}{100,100,220}
		\definecolor{myred}{RGB}{255,100,100}
		\definecolor{mygreen}{RGB}{119,221,119}
		\definecolor{myyellow}{RGB}{253,253,150}
		\definecolor{myorange}{RGB}{255,200,100}
		
		\draw[thick,black] (-.05,-.05) rectangle +(1.1,1.1);
		
		\foreach \i in {0,1}{
			\draw[align=left] (-.02,{\i}) -- (-.08,{\i});
			\draw[align=left] ({\i},-.02) -- ({\i},-.08);
		}
		
		\draw (-.05,0) node[left, align=left] {$0$};
		\draw (-.05,1) node[left, align=left] {$1$};
		\draw (0,-.05) node[below, align=center] {$0$};
		\draw (1,-.05) node[below, align=center] {$1$};
		
		\draw (0,.5) node[left,align=right] {$\lambda_2$};
		\draw (.5,0) node[below,align=center] {$\lambda_1$};
		
		\path[preaction={fill=white}, pattern=north west lines, pattern color=myblue] (0,0) -- (1,1) -- (1,0) -- cycle;
		\path[preaction={fill=white}, pattern=north east lines, pattern color=myorange] (0,0) -- (1,1) -- (0,1) -- cycle;
		
		\draw[very thick, myred] (0,0) -- (1,1);
		
		\def\legendwidth{1.75}
		\def\legendhight{0.32}
		\def\legendposx{1.7}
		\def\legendposy{.5}
		\def\circlerad{1ex}
		\def\circleshift{-0.5ex}
		\def\colorsquare#1{\tikz[baseline=\circleshift]\draw[#1,fill=#1,rotate=45] (0,0) \Square{\circlerad};}
	\end{tikzpicture}
\end{minipage}

Later in the paper, we will show that the nomenclature is consistent with the characterization of cocycle dynamics according to Avila's global theory (cf.\ Theorem~\ref{t:cocycleREG}).
\smallskip

\item	Naturally, one is interested in the \emph{dynamical} behavior of the walk as well, that is, the spreading (or lack thereof) associated with the unitary group $\{W^t\}_{t \in \bbZ}$. Here we note that (at least for Diophantine $\Phi$) one appears to observe a transition from ballistic motion in the subcritical regime to diffusive motion in the critical case to localization in the supercritical region. For the frequency $\Phi = (\sqrt{5}-1)/2$ and phase $\theta = 0$, we plot the standard deviation of the position operator as a function of time in Figure~\ref{fig:unitaryDynamics}.
		\begin{figure}[h]
			\begin{center}
				\includegraphics[width=.5\textwidth]{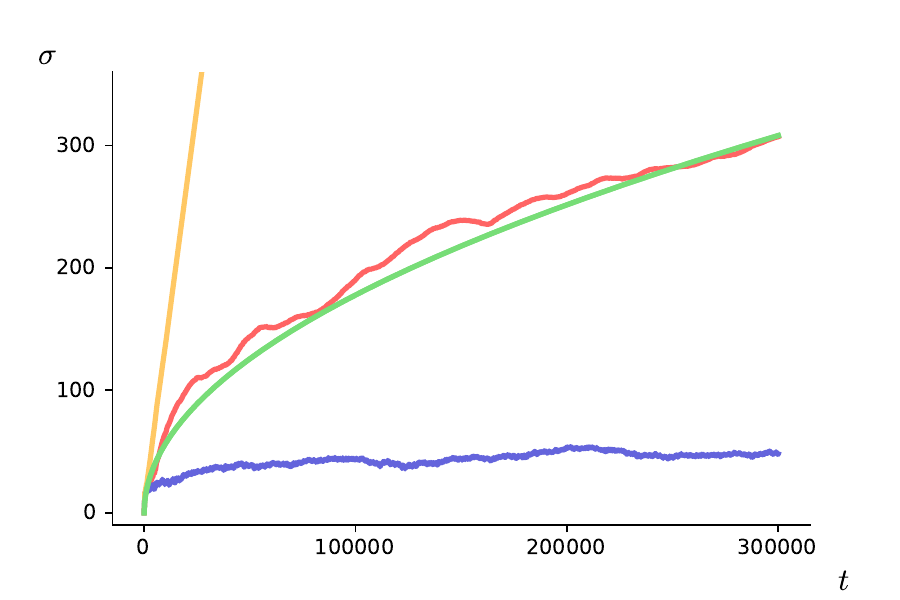}
			\end{center}
			\caption{Different spectral types have drastically different dynamical behaviour, even for very small changes in the coupling constants as can be seen from the square of the variance of the position operator: fixing $\Phi=(1-\sqrt{5})/2$ and $\theta=0$, the orange line corresponds to the subcritical case with $(\lambda_1,\lambda_2)=(1,1-10^{-4})$ where $\sigma\sim t$, the red curve is critical with $(\lambda_1,\lambda_2)=(1,1)$ and the localized blue curve with $\sigma\sim\text{const}$ is supercritical with $(\lambda_1,\lambda_2)=(1-10^{-4},1)$. The green curve with $\sigma\sim \sqrt{t}$ is plotted for reference. \label{fig:unitaryDynamics}}
		\end{figure}
				
\end{enumerate}
\end{remark}

\subsection{Results}

Let us now describe our main results. First, we discuss the spectral type in each parameter region for a.e.\ frequency and phase.
\begin{theorem} \label{t:maintype}  Assume $\lambda_1,\lambda_2 \in [0,1]$.
\begin{enumerate}[label={\rm (\alph*)}]
\item \label{t:maintype:sub} If $\lambda_1 > \lambda_2$, the spectral type of $W_{\lambda_1,\lambda_2,\Phi,\theta}$ is purely absolutely continuous for a.e.\ $\Phi$ and $\theta$.\smallskip

\item \label{t:maintype:super} If $\lambda_1 < \lambda_2$, the spectral type of $W_{\lambda_1,\lambda_2,\Phi,\theta}$ is pure point for a.e.\ $\Phi$ and $\theta$.\smallskip

\item \label{t:maintype:gordon} If $\lambda_1 < \lambda_2$, the spectral type of $W_{\lambda_1,\lambda_2,\Phi,\theta}$ is purely singular continuous for generic $\Phi$ and all $\theta$.\smallskip

\item \label{t:maintype:criticalCantor} If $0 < \lambda_1 = \lambda_2 \leq 1$, then $\Sigma_{\lambda_1,\lambda_2,\Phi}$ is a Cantor set of zero Lebesgue measure for every irrational $\Phi$.   \smallskip

\item \label{t:maintype:critical} If $0 < \lambda_1 = \lambda_2 < 1$, the spectral type of $W_{\lambda_1,\lambda_2,\Phi,\theta}$ is purely singular continuous for all irrational $\Phi$ and all but countably many $\theta$.   \smallskip

\item \label{t:maintype:criticalFOZ} If $\lambda_1 = \lambda_2 = 1$, the spectral type of $W_{\lambda_1,\lambda_2,\Phi,\theta}$ is purely singular continuous for all irrational $\Phi$ and a.e.\ $\theta$.
\smallskip

\item \label{t:maintype:lambda1=0} For any $0\leq\lambda_2\leq1$,  $\Phi \in \bbT$, and $\theta \in \bbT$, the spectral type of $W_{0,\lambda_2,\Phi,\theta}$ is pure point. For every irrational $\Phi$, $\Sigma_{0,\lambda_2,\Phi} = \{z \in \partial \bbD : |\Re(z)| \leq \lambda_2 \}$.
\smallskip

\item \label{t:maintype:lambda2=0} For any $0<\lambda_1\leq1$,  $\Phi \in \bbT$, and $\theta \in \bbT$, the spectral type of $W_{\lambda_1,0 ,\Phi,\theta}$ is purely absolutely continuous. For every irrational $\Phi$, $\Sigma_{\lambda_1,0,\Phi} = \{z \in \partial \bbD : |\Re(z)| \leq \lambda_1 \}$. 
\end{enumerate}
\end{theorem}

\begin{remark}
Let us make some comments about Theorem~\ref{t:maintype}.
\begin{enumerate}[label={\rm (\alph*)}]
\item The result in Theorem~\ref{t:maintype}.\ref{t:maintype:sub} is obtained from the result in Theorem~\ref{t:maintype}.\ref{t:maintype:super} by duality, so the sets of $\Phi$ are the same in each part.\smallskip
\item The generic set in Theorem~\ref{t:maintype}.\ref{t:maintype:gordon} consists of those irrational $\Phi$ satisfying a suitable Liouville condition, namely,
\[ \limsup_{k\to\infty} \frac{\log q_{k+1}}{q_k} = \infty, \]
where $q_k$ denotes the $k$th continued fraction denominator associated with $\Phi$.\smallskip
\item Theorem~\ref{t:maintype}.\ref{t:maintype:criticalFOZ} and the case $\lambda_1=\lambda_2=1$ from Theorem~\ref{t:maintype}.\ref{t:maintype:criticalCantor} were proved in \cite{FOZ2017CMP} and are listed here for the sake of completeness. We remark that the full-measure set of $\theta$ was written down explicitly in \cite{FOZ2017CMP}.\smallskip
\item 
	In the region $0<\lambda_1=\lambda_2\leq1$ numerical evidence suggests that the spectrum is purely singular continuous for all phases. In the self-adjoint AMO setting this was proved in \cite{Jitomirskaya2021Advances} using a particular choice of gauge in which the AMO becomes singular. However, we were unable to apply this gauge to obtain fruitful results in the current model.
	\smallskip

\item Parts~\ref{t:maintype:lambda1=0} and \ref{t:maintype:lambda2=0} of Theorem~\ref{t:maintype} are trivial and are listed for the sake of completeness.
Nevertheless, for the reader's convenience we have included the proofs to these statements at the end of Section \ref{sec:tms}.
\smallskip

\item If $\lambda_1>0$, $\lambda_2 < 1$, and $\Phi$ is rational, then $W$ has purely absolutely continuous spectrum for every phase. 
\smallskip

\item If $\lambda_1>0$, $\lambda_2 = 1$, and $\Phi = p/q$ is rational in lowest terms, then there are two possibilities. If $\theta \in \frac{1}{4}+\frac{1}{q}\bbZ+ \frac{1}{2}\bbZ$, then again $W$ decouples into an infinite direct sum and has pure point spectrum with eigenvalues of infinite multiplicity corresponding to compactly supported eigenfunctions. Otherwise, the spectrum of $W$ is purely absolutely continuous.\smallskip

\item In the region $\lambda_1>\lambda_2$, we expect that purely absolutely continuous spectrum holds for all $\theta$ and $\Phi$ (note that $\lambda_1>\lambda_2$ implies $\lambda_1>0$ and $\lambda_2<1$ so the examples with compactly supported eigenfunctions do not occur in this regime). The proof of this would take us outside the scope of the current paper; we plan to investigate it in future work. \end{enumerate}
\end{remark}

There are two crucial ingredients involved in the proofs of the main results: (1) a suitable version of Aubry duality for the operator family $\{W_{\lambda_1,\lambda_2,\Phi,\theta} : \theta \in \bbT\}$ and (2) a careful analysis of the associated transfer matrix cocycle.

The form of Aubry duality for this model demonstrates that dual of the family $\{W_{\lambda_1,\lambda_2,\Phi,\theta} : \theta \in \bbT\}$ is given by $\{ W_{\lambda_2,\lambda_1,\Phi,\theta}^\top : \theta \in \bbT\}$ in a suitable sense. 
Let us make this precise.
Define
\begin{equation}
W_{\lambda_1,\lambda_2, \Phi, \theta}^\aubrydual = W_{\lambda_2,\lambda_1,\Phi,\theta}^\top,
\end{equation}
where $\top$ denotes the transpose.

\begin{theorem}[Aubry duality via solutions] \label{t:aubryViaSolutions}
Let $\lambda_1, \lambda_2\in [0,1]$, $\Phi \in \bbT$ irrational, and $\theta \in \bbT$ be given, and suppose $\psi \in \scrH_1$ satisfies $W_{\lambda_1,\lambda_2,\Phi,\theta}\psi = z\psi$ for some $z \in \partial \bbD$.
If $\varphi^\xi$ is defined for $\xi \in \bbT$ by
\begin{equation} \label{eq:models:dualSolutionDef}
 \begin{bmatrix} \varphi^{\xi,+}_n \\ \varphi_n^{\xi,-} \end{bmatrix} 
= \frac{1}{\sqrt{2}} e^{2\pi i n \theta} \begin{bmatrix} 1 & i \\ i & 1 \end{bmatrix} \begin{bmatrix} \widecheck{\psi}^+(n\Phi+\xi) \\ \widecheck{\psi}^-(n\Phi+\xi) \end{bmatrix}, 
\end{equation}
where $\widecheck{\cdot}$ denotes the inverse Fourier transform,
then 
\begin{equation} \label{eq:tdualsolution:Wxi=phixi}
W_{\lambda_1,\lambda_2, \Phi, \xi}^\aubrydual \varphi^{\xi} = z\varphi^\xi
\end{equation} for a.e.\ $\xi$. Furthermore, if $\psi \in \ell^1(\bbZ,\bbC^2)$, then \eqref{eq:tdualsolution:Wxi=phixi} holds for all $\xi$.
\end{theorem}

Let us emphasize that $\varphi^\xi \in \bbC^\bbZ \otimes \bbC^2$ is not assumed to belong to the Hilbert space $\scrH_1 = \ell^2(\bbZ) \otimes \bbC^2$, so one should interpret $W_{\lambda_1, \lambda_2, \Phi, \xi}^\aubrydual$ as a finite difference operator in the previous theorem.
Indeed, when $\psi$ is an Anderson localized (i.e., exponentially decaying) eigenvector of $W_{\lambda_1,\lambda_2, \Phi, \theta}$, then $\varphi^\xi$ is in general an extended state: bounded but not decaying.

One can also express duality of \emph{generalized} eigenfunctions via suitable direct integrals.  The following formulation is useful.

\begin{theorem}[Aubry duality via direct integrals] \label{t:aubryOperator}
For any $\lambda_1$, $\lambda_2$, and $\Phi$, 
\[\int_\bbT^\oplus W_{\lambda_1,\lambda_2,\Phi,\theta}\, d\theta\cong \int_\bbT^\oplus W_{\lambda_1,\lambda_2,\Phi,\theta}^\aubrydual \, d\theta,\]
where $\cong$ denotes unitary equivalence of operators on $\displaystyle \int_\bbT^\oplus \! \scrH_1 \, d\theta$.
\end{theorem}
Let us note the following elementary consequence of Theorem~\ref{t:aubryOperator}.
\begin{coro}\label{cor:aubryspecra}
For any $\lambda_1, \lambda_2 \in [0,1]$ and $\Phi \in \bbT$ irrational,
\begin{equation}
\Sigma_{\lambda_1,\lambda_2,\Phi}
 = \Sigma_{\lambda_2,\lambda_1,\Phi}.
\end{equation}
\end{coro}

\begin{proof}
By minimality and Theorem~\ref{t:aubryOperator}, we have
\[ \Sigma_{\lambda_1,\lambda_2,\Phi}
=\sigma\left(\int_\bbT^\oplus W_{\lambda_1,\lambda_2,\Phi,\theta}\, d\theta\right)
= \sigma\left( \int_\bbT^\oplus W_{\lambda_2,\lambda_1,\Phi,\theta}^\top \, d\theta \right)
 = \Sigma_{\lambda_2,\lambda_1,\Phi}, \]
 as desired.
\end{proof}

The second crucial ingredient in our work is the classification of cocycle behavior in the language of Avila's global theory \cite{Avila2015Acta}. To formulate these results, let us recall the general setting. Given $\Phi$ irrational and a continuous map $M:\bbT \to \bbC^{2\times 2}$, we may consider the skew product 
\begin{equation}
(\Phi,M):\bbT \times \bbC^2 \to \bbT \times \bbC^2, \quad (\theta,v)\mapsto (\theta+\Phi,M(\theta)v)
\end{equation}
and its iterates given by $(\Phi,M)^n = (n\Phi,M^n)$, where
\begin{align}\nonumber
M^n(\theta) = M^{n,\Phi}(\theta) & = M((n-1)\Phi+\theta) \cdots M(\Phi+\theta) M(\theta) \\
\label{eq:cociterates}
& = \prod_{j=n-1}^0 M(j\Phi+\theta)
\end{align}
for $n \in \bbN$. The map $(\Phi,M)$ is called the \emph{quasiperiodic cocycle associated with $\Phi$ and $M$.} The \emph{Lyapunov exponent} of $(\Phi,M)$ is given by
\begin{equation} \label{eq:leDefinition}
 L(\Phi,M) = \lim_{n\to\infty} \frac{1}{n} \int_\bbT \log\|M^n(\theta)\| \, d\theta.
 \end{equation}
If $M$ is analytic and enjoys an analytic extension to a strip of the form $\{ \theta + i \varepsilon : \theta \in \bbT, \ |\varepsilon| < \delta \}$, then for  $|\varepsilon|< \delta$, we consider the complexified cocycle map
\begin{equation} \label{eq:complexifiedMdef} M(\cdot+i\varepsilon):\bbT \mapsto M(\theta+i\varepsilon).
\end{equation}
For each $|\varepsilon|<\delta$, one can consider the Lyapunov exponent of the complexified cocycle:
\begin{equation} \label{eq:complexifiedLEdef}
L(\Phi,M,\varepsilon)
:= L(\Phi,M(\cdot+i\varepsilon)).
\end{equation}
If, in addition, $M:\bbT \to \SL(2,\bbR)$, one says that $(\Phi,M)$ is 
\begin{itemize}
\item \textbf{uniformly hyperbolic} if $\|M^n(\theta)\| \geq ce^{\lambda|n|}$ for all $n \in \bbZ$ and constants $c,\lambda>0$
\item \textbf{supercritical} if $L(\Phi,M)>0$ and $(\Phi,M)$ is not uniformly hyerbolic
\item \textbf{subcritical} if $L(\Phi,M,\varepsilon)=0$ for all $\varepsilon$ in a strip containing $\varepsilon=0$
\item \textbf{critical} if $L(\Phi,M)=0$ but $(\Phi,M)$ is not subcritical. 
\end{itemize}

Let us now explain how quasiperiodic cocycles arise in the spectral analysis of the operators $W_{\lambda_1,\lambda_2,\Phi,\theta}$. To study the spectral problem associated with a walk $W = S_\lambda Q$, one naturally studies the generalized eigenvalue equation $W\psi = z \psi$ with $z \in \bbC$, which leads to the transfer matrices defined by 
\begin{equation}\label{eq:model:transmatDef}
	T_{z}(n)=\frac1{q^{22}_n}\begin{bmatrix}\lambda^{-1}z^{-1}\det Q_n+\lambda'\lambda^{-1}(q_n^{21}-q_n^{12})+z{\lambda'}^2\lambda^{-1}	&	q_n^{12}-\lambda'z	\\	-q_n^{21}-\lambda'z	&	\lambda z	\end{bmatrix}.
\end{equation}
We will show that 
\begin{equation*}
	\begin{bmatrix}\psi_{n+1}^+\\\psi_{n}^-\end{bmatrix}
	= T_{z}(n)\begin{bmatrix}\psi_{n}^+\\\psi_{n-1}^-\end{bmatrix}
\end{equation*}
whenever $z \in \bbC \setminus \{0\}$ and $\psi$ solves $W\psi = z \psi$ (see Section~\ref{sec:tms} for details).

In the primary quasiperiodic setting of this paper with coins $Q_n=Q_{\lambda_2,\Phi,\theta,n}$ given in \eqref{eq:moddefs:coinDef}, this leads naturally to a quasiperiodic cocycle.
To describe this more precisely, given $\lambda_1$, $\lambda_2$, $\Phi$, $\theta$, and $z$, define $A_z(\theta)
 = A_{\lambda_1,\lambda_2,z}(\theta)$ to be $T_z(0)$ as in \eqref{eq:model:transmatDef} with $Q_n$ as in \eqref{eq:moddefs:coinDef}, that is,
\begin{equation} \label{eq:AlambdaCocycleDef}
A_z(\theta)
= \frac{1}{\lambda_2 \costwopi(\theta) - i\lambda_2'}
\begin{bmatrix}\lambda_1^{-1}z^{-1}+2\lambda_1'\lambda_1^{-1}\lambda_2 \sintwopi(\theta) + z{\lambda_1'}^2\lambda_1^{-1}	&	-\lambda_2 \sintwopi(\theta)-\lambda_1'z	\\	- \lambda_2 \sintwopi(\theta) -\lambda_1'z	&	\lambda_1z	\end{bmatrix}, \end{equation}
where we adopt the abbreviations 
\begin{align} \label{eq:costwopiDef}
\costwopi(\theta) = \cos(2\pi\theta), \qquad
\sintwopi(\theta) = \sin(2\pi \theta).
\end{align}
This defines a cocycle map, whose iterates and Lyapunov exponent may be considered as in \eqref{eq:cociterates} and \eqref{eq:leDefinition}.

We define
\[ L(z) 
= L_{\lambda_1,\lambda_2,\Phi}(z) 
= L(\Phi,A_{\lambda_1,\lambda_2,z})
= \lim_{n\to\infty} \frac{1}{n} \int_{\bbT} \log\| A_{\lambda_1,\lambda_2,z}^{n,\Phi}(\theta) \| \, d\theta \]
which by irrationality of $\Phi$ and Kingman's subadditive ergodic theorem satisfies for each $z$ and a.e. $\theta$:
\[ L_{\lambda_1,\lambda_2,\Phi}(z) 
= \lim_{n\to\infty} \frac{1}{n}\log\|A_{\lambda_1,\lambda_2,z}^{n,\Phi}(\theta) \| . \]
Later, we will see $|\det(A_{\lambda_1,\lambda_2}(\theta,z))| \equiv 1$ (see \eqref{eq:AlambdaDet}), which naturally implies $L(z) \geq 0$ for all $z \in \bbC\setminus\{0\}$.

We will see later that the transfer matrix cocycle map associated with the Aubry dual walk $W^\aubrydual$ is given by
\begin{equation} \label{eq:models:Asharpdef} A^\aubrydual_{\lambda_1, \lambda_2, z}(\theta) = \overline{A_{\lambda_2, \lambda_1, 1/\bar z}(\theta)}.\end{equation}
In particular, for $z \in \partial \bbD$ one has $A^\aubrydual_{\lambda_1,\lambda_2,z} = \overline{A_{\lambda_2,\lambda_1, z}}$, so statements transfer into the dual setting in a straightforward fashion.

One must be careful in the present setting for two reasons. First, the transfer matrices for the UAMO are not in $\SL(2,\bbR)$: indeed they are not even unimodular; this can be solved with a suitable conjugacy that moves the normalized cocycle $A/\sqrt{\det A}$ into $\SL(2,\bbR)$ (we will give a more detailed description later). Second, in the case $\lambda_2 = 1$, the associated cocycle obviously is not analytic and hence the apparatus of global theory cannot be applied directly, even after shifting to the real cocycle. To deal with that, we work with a suitable regularization; in fact, because of singularities that arise off the real axis, we use this regularization even when $\lambda_2 \neq 1$. To that end, we consider $B = B_{\lambda_1,\lambda_2,z}$ given by
	\begin{align}
	\nonumber
B_{\lambda_1,\lambda_2,z}(\theta)
& = \left[\frac{2(\lambda_2 \costwopi(\theta)-i\lambda_2')}{1+\lambda_2'} \right] A_{\lambda_1,\lambda_2,z}(\theta) \\
& =  \label{eq:Blambdaearlydef}
\frac{2}{\lambda_1(1+\lambda_2')}
\begin{bmatrix} z^{-1}+2\lambda_1'\lambda_2 \sintwopi(\theta) + z{\lambda_1'}^2
&	-\lambda_1(\lambda_2 \sintwopi(\theta) + \lambda_1'z)	\\	
-\lambda_1(\lambda_2 \sintwopi(\theta)+\lambda_1'z	)
&	\lambda_1^2 z	\end{bmatrix}.
	\end{align}
The additional prefactor $2/(1+\lambda_2')$ is chosen so that the Lyapunov exponents of $A$ and $B$ are precisely the same (which we will demonstrate later). The regularized cocycle $B$ is analytic, even for $\lambda_2=1$ and one can classify cocycle dynamics throughout all parameter regions.

As mentioned before, the case $\lambda_2 = 1$ presents some complications with the general nomenclature: the cocycle $A$ is not analytic and the regularized cocycle $B$ has points of singularity at which $\det B = 0$, so neither can be pushed to $\SL(2,\bbR)$ in a simple way. Let us address this case first.

\begin{theorem} \label{t:cocycleLambda2=1}
Suppose $\Phi$ is irrational and $\lambda_2=1$.
\begin{enumerate}[label={\rm (\alph*)}]
\item \label{t:cocyclelambda2=1lambda1=1} If $\lambda_1=\lambda_2=1$, then $L(\Phi,B_{1,1,z},\varepsilon) = 2\pi|\varepsilon|$ for all $z \in \Sigma_{1,1,\Phi}$ and all $\varepsilon$.
\item \label{t:cocyclelambda2=1lambda1<1} If $0<\lambda_1 < \lambda_2=1$, then $L_{\lambda_1,1,\Phi}(z)>0$ for all $z \in \Sigma_{\lambda_1,1,\Phi}$. Indeed,
\begin{equation}
L_{\lambda_1,1,\Phi}(z) 
\geq \log \left[ \frac{1+\lambda_1'}{\lambda_1}\right],
\end{equation}
with equality if and only if $z$ belongs to the spectrum.
\end{enumerate}
\end{theorem}

One could consider calling the behavior in Theorem~\ref{t:cocycleLambda2=1}.\ref{t:cocyclelambda2=1lambda1=1} ``critical'' and the behavior in Theorem~\ref{t:cocycleLambda2=1}.\ref{t:cocyclelambda2=1lambda1<1} ``supercritical'', but it should be understood that this is a slight abuse of terminology since those words are generally understood in the literature to refer to $\SL(2,\bbR)$ cocycles.

Now let us consider the cases with $\lambda_2<1$. In order to apply the relevant results from the global theory, we must work with an $\SL(2,\bbR)$ cocycle. There are thus two impediments: $\det A_{\lambda_1,\lambda_2,z} \neq1$ and $A_{\lambda_1,\lambda_2,z}/\sqrt{\det A_{\lambda_1,\lambda_2,z}} \notin \SL(2,\bbR)$ in general.

Let $Y(\lambda)$ denote the unitary matrix
\begin{equation} \label{eq:selfdual:YlamDef} 
Y(\lambda) = \frac{1}{2} \begin{bmatrix} 
\sqrt{1+\lambda}- i \sqrt{1-\lambda} 
& - \sqrt{1-\lambda}+i\sqrt{1+\lambda} \\
\sqrt{1-\lambda} + i \sqrt{1+\lambda}
& \sqrt{1+\lambda} + i\sqrt{1-\lambda} \end{bmatrix}, \end{equation}
and put
\begin{equation} \label{eq:selfdual:realifiedAdef}
A_{\lambda_1,\lambda_2,z}^\realified(\theta)
:=  Y(\lambda_1)^* \frac{A_{\lambda_1,\lambda_2,z}(\theta)}{\sqrt{\det A_{\lambda_1,\lambda_2,z}(\theta)}} Y(\lambda_1).
\end{equation}

\begin{prop} \label{prop:conjtoSL2R}
For all $\lambda_1 \in (0,1]$, $\lambda_2 \in [0,1)$, and $z \in \partial \bbD$, $A^\realified_{\lambda_1,\lambda_2,z}
$ defined in \eqref{eq:selfdual:realifiedAdef} is an analytic map $\bbT \to \SL(2,\bbR)$ with analytic extension to a strip. Indeed,
\begin{equation} \label{eq:realifiedcocycleexplicit}
A_{\lambda_1,\lambda_2,z}^\realified(\theta)
= \frac{1}{\lambda_1 \sqrt{{\lambda_2'}^2 + \lambda_2^2\costwopi^2(\theta)}} 
\begin{bmatrix} 
\Re \, z + \lambda_1'\lambda_2 \sintwopi(\theta) 
& \lambda_1 \Im \, z - \lambda_1'\Re \, z - \lambda_2 \sintwopi(\theta)  \\
 -\lambda_1 \Im \, z - \lambda_1'\Re \, z - \lambda_2 \sintwopi(\theta) 
 & \Re \, z + \lambda_1'\lambda_2 \sintwopi(\theta) 
\end{bmatrix}
\end{equation}
for all $\theta \in \bbT$, where we recall $\sintwopi(\theta) = \sin(2\pi \theta)$ and $\costwopi(\theta) = \cos(2\pi \theta)$. 
\end{prop}

Having conjugated the cocycle into $\SL(2,\bbR)$, we can discuss the characterization of $(\Phi,A^\realified)$ in the language of global theory.

\begin{theorem} \label{t:cocycleREG}
Suppose $\Phi$ is irrational and $\lambda_2<1$.
\begin{enumerate}[label={\rm (\alph*)}]
\item  \label{t:cocycleSubcritical} 
If $0 \leq \lambda_2 < \lambda_1 \leq 1$, then the cocycle $(\Phi,A_{\lambda_1,\lambda_2,z}^\realified)$ is subcritical for all $z \in \Sigma_{\lambda_1,\lambda_2,\Phi}$. In particular, $L(z) = 0$ for all $z \in \Sigma_{\lambda_1,\lambda_2,\Phi}$.\smallskip
\item \label{t:cocycleCritical}
If $0 < \lambda_1 = \lambda_2 < 1$, then the cocycle $(\Phi,A_{\lambda_1,\lambda_2,z}^\realified)$ is critical for all $z \in \Sigma_{\lambda_1,\lambda_2,\Phi}$. In particular, $L(z) = 0$ for all $z \in \Sigma_{\lambda_1,\lambda_2,\Phi}$.
\smallskip
\item \label{t:cocycleSupercritical}
If $0 < \lambda_1 < \lambda_2 < 1$, then the cocycle $(\Phi,A_{\lambda_1,\lambda_2,z}^\realified)$ is supercritical for all $z \in \Sigma_{\lambda_1,\lambda_2,\Phi}$. Moreover, the Lyapunov exponent $L_{\lambda_1,\lambda_2,\Phi}$ satisfies
\begin{equation} 
L_{\lambda_1,\lambda_2,\Phi}(z) 
\geq \log \left[ \frac{\lambda_2(1+\lambda_1')}{\lambda_1(1+\lambda_2')}\right]
\end{equation}
with equality if and only if $z$ belongs to the spectrum.
\end{enumerate} 
\end{theorem}

By Theorems~\ref{t:cocycleLambda2=1} and \ref{t:cocycleREG}, one has the following corollary which shows that the Lyapunov exponent can be explicitly computed on the spectrum.

\begin{coro} \label{coro:lyapExact}
For any $0<\lambda_1 \leq 1$, $0\leq \lambda_2 \leq 1$, $\Phi \in \bbT$, and any $z \in \Sigma_{\lambda_1,\lambda_2,\Phi}$, one has
\[L_{\lambda_1,\lambda_2,\Phi}(z) = \max\{0,\log\lambda_0\},\]
where $\lambda_0$ is given by \eqref{eq:model:lambda0Def}.
\end{coro}

\begin{remark} \label{rem:subcritical}
The careful analysis of cocycle dynamics pays dividends for the determination of the spectral type. For instance, we use criticality in an essential manner to prove that the spectrum is a zero-measure Cantor set in the case $\lambda_1=\lambda_2>0$. Similarly, the subcriticality statement implies that the transfer matrix cocycle is almost reducible on the spectrum by Avila's almost-reducibility conjecture announced in \cite{Avila2015Acta} and proved in \cite{AvilaARAC1, AvilaARAC2}. One can then apply reducibility theory as in \cite{LiDamZhou2021Preprint} in order to deduce purely absolutely continuous spectrum for all phases. To execute this scheme would take us outside the scope of the current work; we plan to address it in future investigations.
\end{remark}

\subsection{Connection to CMV Matrices}
\label{sec:CMV}

Let us make explicit a connection between the current setup and Cantero--Moral--Vel\'{a}zquez 
(CMV) matrices. This generalizes the CGMV connection, which is named after the foundational papers \cite{CGMV2010CPAM, CGMV2012QIP}. The CMV matrix is a unitary operator on $\ell^2(\mathbb N)$ that has many applications in in spectral theory and orthogonal polynomials on the unit circle (OPUC). This connection is analogous to the relationship between the Jacobi matrix and orthogonal polynomials on the real line (OPRL). A standard reference for the many connections between the CMV matrix and OPUC is \cite{Simon2005OPUC1,Simon2005OPUC2}. In situations in which the coefficients are generated by an invertible dynamical system (such as an irrational circle rotation as in this work), is is natural to work with \emph{extended CMV matrices}, which are the natural analogues acting in $\ell^2(\bbZ)$.

In the present work, we need some additional flexibility afforded by complexifying certain parameters, which gives rise to objects we call \emph{generalized} CMV matrices. However, we will explain how to relate these objects to standard CMV matrices. Let us first define generalized CMV matrices. Let $\overline{\bbD} = \{ z \in \bbC : |z| \le 1\} $. For each pair 
\[ (\alpha,\rho) \in \bbS^3 = \set{(z_1,z_2) \in \overline{\bbD}^2 : |z_1|^2 + |z_2|^2=1},
\]
put
\[ \Theta(\alpha,\rho) 
= \begin{bmatrix} \overline{\alpha} & \rho \\ \overline{\rho} & -\alpha \end{bmatrix}. \]
Given sequences of elements $(\alpha_n,\rho_n) \in \mathbb{S}^3$, put
\begin{align*}
	\mathcal{L} = \mathcal{L}(\{(\alpha_{2n},\rho_{2n})\}_{n\in\bbZ}) & = \bigoplus \Theta(\alpha_{2n},\rho_{2n}), \\
	\mathcal{M} = \mathcal{M}(\{(\alpha_{2n+1},\rho_{2n+1})\}_{n\in\bbZ}) & = \bigoplus \Theta(\alpha_{2n+1},\rho_{2n+1}) 
\end{align*}
where in both cases $\Theta(\alpha_j,\rho_j)$ acts on $\ell^2(\{j, j+1\})$. One then takes $\calE = \calE_{\alpha,\rho} := \calL\calM$, which one can check has the matrix representation
\begin{equation} \label{def:extcmv}
	\calE
	=
	\begin{bmatrix}
		\ddots & \ddots & \ddots & \ddots &&&&  \\
		& \overline{\alpha_0\rho_{-1}} & \boxed{-\overline{\alpha_0}\alpha_{-1}} & \overline{\alpha_1}\rho_0 & \rho_1\rho_0 &&&  \\
		& \overline{\rho_0\rho_{-1}} & -\overline{\rho_0}\alpha_{-1} & {-\overline{\alpha_1}\alpha_0} & -\rho_1 \alpha_0 &&&  \\
		&&  & \overline{\alpha_2\rho_1} & -\overline{\alpha_2}\alpha_1 & \overline{\alpha_3} \rho_2 & \rho_3\rho_2 & \\
		&& & \overline{\rho_2\rho_1} & -\overline{\rho_2}\alpha_1 & -\overline{\alpha_3}\alpha_2 & -\rho_3\alpha_2 &    \\
		&& && \ddots & \ddots & \ddots & \ddots &
	\end{bmatrix},
\end{equation}
where all unspecified matrix entries are zero and we have placed a box around the entry corresponding to $\langle \delta_0, \calE \delta_0 \rangle$. 
If additionally one has $\rho_n  \in [0,1]$ for all $n$, one necessarily has $\rho_n = (1-|\alpha_n|^2)^{1/2}$, and we simply call this a \emph{standard CMV matrix}. In this case, we refer to $\{\alpha_n\}_{n \in \bbZ}$ as the sequence of \emph{Verblunsky coefficients} of $\calE$. If  $\rho_n=0$ for some $n \in \bbZ$, the reader can check that $\calE$ preserves $\ell^2(\bbZ \cap (-\infty,n])$ and $\ell^2(\bbZ \cap [n+1,\infty))$, so the operator $\calE$ decomposes as a direct sum of two half-line CMV operators on $(-\infty,n]$ and $[n+1,\infty)$.

In \cite{CGMV2010CPAM}, Cantero, Gr\"unbaum, Moral, and Vel\'azquez observed that one may connect CMV matrices and quantum walks. Since our quantum walk setting is slightly more general than theirs, let us briefly describe this connection. Order the basis of $\scrH_1$ as follows:
\begin{equation} \label{eq:orderedBasis} \ldots, \delta_{-1}^-, \delta_0^+, \delta_0^-, \delta_1^+, \delta_1^-, \delta_2^+, \ldots \end{equation}
Computing directly, one observes
\begin{align*}
	W\delta_n^+ 
	= S_\lambda(q_n^{21}\delta_n^- + q_n^{11}\delta_n^+)
	& = \lambda q_n^{21}\delta_{n-1}^- - \lambda' q_n^{21} \delta_n^+  + \lambda' q_n^{11} \delta_n^-  + \lambda q_n^{11} \delta_{n+1}^+  \\
	W\delta_n^- 
	= S_\lambda(q_n^{22}\delta_n^- + q_n^{12}\delta_n^+)
	& = \lambda q_n^{22}\delta_{n-1}^- - \lambda' q_n^{22} \delta_n^+  + \lambda' q_n^{12} \delta_n^-  + \lambda q_n^{12} \delta_{n+1}^+.
\end{align*}
Thus, writing $W$ in the ordered basis \eqref{eq:orderedBasis}, we get
\begin{equation}
	W 
	=  \begin{bmatrix}
		\ddots & \ddots & \ddots \\
		&  \lambda  q_0^{21} &  \lambda  q_0^{22} \\
		&  -\lambda' q_0^{21} & -\lambda' q_0^{22} \\
		&  \lambda' q_0^{11} & \boxed{\lambda' q_0^{12}} &  \lambda  q_1^{21} &  \lambda  q_1^{22} \\
		&  \lambda  q_0^{11} &  \lambda  q_0^{12} & -\lambda' q_1^{21} &  -\lambda' q_1^{22} \\
		&&&  \lambda' q_1^{11} &  \lambda' q_1^{12} \\
		&&&  \lambda  q_1^{11} &  \lambda  q_1^{12} \\
		&&& \ddots & \ddots & \ddots
	\end{bmatrix},
\end{equation}
where we have boxed the $\delta_0^-$-$\delta_0^-$ entry of $W$.
Thus, if we additionally assume that $\det Q_n = 1$ for all $n$, we may write
\[Q_n = \begin{bmatrix} \overline{\rho_{2n-1}} & -\alpha_{2n-1} \\ \overline{\alpha_{2n-1}} & \rho_{2n-1} \end{bmatrix} \]
for some suitable $(\alpha_{2n-1},\rho_{2n-1}) \in \mathbb{S}^3$. With this, $W$ can be identified with a generalized CMV matrix with those odd coefficients and such that $(\alpha_{2n},\rho_{2n}) = (\lambda',\lambda)$ for all $n \in \bbZ$. In particular, the walk $W_{\lambda_1,\lambda_2,\Phi,\theta}$ is equivalent to the generalized CMV matrix $\calE_{\lambda_1,\lambda_2,\Phi,\theta}$ given by
\begin{equation} \label{calElambda12phithetadef}
	\begin{split}
		\alpha_{2n-1} = \lambda_2 \sin(2\pi(n\Phi+\theta)), & \quad \rho_{2n-1} =  \lambda_2 \cos(2\pi(n\Phi+\theta)) - i \lambda_2' \\
		\alpha_{2n} = \lambda_1',& \quad \rho_{2n} = \lambda_1.
	\end{split}
\end{equation}

In view of this connection, each of our main results has an application to a suitable generalized CMV matrix.  We have chosen to begin with quantum walks rather than CMV matrices since it makes more clear the motivation behind choosing our parameters in the manner that we did. 

Moreover, every generalized CMV matrix is equivalent to a standard CMV matrix in a simple manner. Since we quote some results from the theory of standard CMV matrices, let us spell this connection out in more detail. The following result in this formulation may be found in \cite{CFLOZ}; for the reader's convenience, we give the proof. We say that two operators $A$ and $B$ in $\ell^2(\bbZ)$ are \emph{gauge equivalent} if there is a diagonal unitary operator $D$ on $\ell^2(\bbZ)$ such that $D A D^{*} = B$.

\begin{prop}\label{prop:realify}
	Every generalized extended CMV matrix is gauge-equivalent to a standard CMV matrix. Indeed, for any sequence $(\alpha,\rho) \in (\bbS^3)^\bbZ$, there is a diagonal unitary operator $D$ so that $D \calE_{\alpha,\rho}D^{*} = \calE_{\alpha,|\rho|}$
\end{prop}

\begin{proof}
	For $z \in \bbC$, denote $\mathrm{ang}(z) = z/|z|$ if $z \neq 0$ and $\mathrm{ang}(0)=1$. Define $d_0 = 1$, $d_n = \mathrm{ang}(\rho_0\rho_1\cdots\rho_{n-1}$) for $n >1$, and $d_n = (\mathrm{ang}(\rho_n \rho_{n+1} \cdots \rho_{-1}))^{-1}$ for $n<-1$. A direct calculation shows that the diagonal unitary $[D \psi]_n = d_n \psi_n$ satisfies $D \calE_{\alpha,\rho}D^* = \calE_{\alpha,|\rho|}$.
\end{proof}

\begin{coro} \label{coro:UAMOtoCMV}
	The generalized CMV matrix $\calE_{\lambda_1,\lambda_2,\Phi,\theta}$ defined by \eqref{calElambda12phithetadef} is gauge-equivalent to the standard CMV matrix $\calE_{\lambda_1,\lambda_2,\Phi,\theta}^\realified$ defined by
	\begin{equation}
		\alpha_{2n} = \lambda_1', \quad \alpha_{2n-1} = \lambda_2 \sin(2\pi(n\Phi+\theta)), \quad \rho_n = \sqrt{1-|\alpha_n|^2}.
	\end{equation}
\end{coro}

As a generalized CMV matrix, there are two other cocycles associated with $W$, namely, the \emph{Szeg\H{o} cocycle} (cf.\ \cite[Equation~(1.5.35)]{Simon2005OPUC1}) and the \emph{Gesztesy--Zinchenko {\rm(GZ)} cocycle} \cite{GZ2006JAT}. Denoting
\[X(\alpha,\rho,z) = \frac{1}{\rho} \begin{bmatrix} z & - \overline{\alpha} \\ -\alpha z & 1\end{bmatrix}, 
\quad P(\alpha,\rho,z) =  \frac{1}{\rho}\begin{bmatrix} - \overline{\alpha} & z \\ 1/z & -{\alpha} \end{bmatrix},
\quad Q(\alpha,\rho,z) = \frac{1}{\rho}\begin{bmatrix} - \alpha & 1 \\ 1 & -\overline{\alpha} \end{bmatrix}  \] the (two-step) Szeg\H{o} cocycle associated with $W_{\lambda_1,\lambda_2,\Phi,\theta}$ is given by 
\begin{equation} \label{eq:szegococycledef}
	S_{\lambda_1,\lambda_2,z}(\theta) = X(\lambda_1',\lambda_1,z)X(\lambda_2 \sin(2\pi\theta), \lambda_2\cos(2\pi\theta)-i\lambda_2',z)
\end{equation}
and the two-step GZ cocycle is given by
\begin{equation} \label{eq:GZcocycledef}
	G_{\lambda_1,\lambda_2,z}(\theta) =  P(\lambda_1',\lambda_1,z)Q(\lambda_2 \sin(2\pi\theta), \lambda_2\cos(2\pi\theta)-i\lambda_2',z)
\end{equation}
An equivalence between $G$ and $S$ is given in \cite[Eq.~(3.3)]{DFO2016JMPA}, and an equivalence between $G$ and $A$ is given in \cite[Eq.~(3.20)]{YangFPreprint}. We have
\begin{equation} \label{eq:szegoAzConjugacy}
	z^{-1}S_{z,\lambda_1,\lambda_2} = C_{\lambda_1} A_{\lambda_1,\lambda_2,z}C_{\lambda_1}^{-1}, \quad C_{\lambda_1} = \begin{bmatrix}
		-\lambda_1' & \lambda_1 \\ 1 & 0
	\end{bmatrix},
\end{equation}
and
\begin{equation} \label{eq:szegoToGZ}
	D_z^{-1} G_{\lambda_1, \lambda_2,z}D_z  =z^{-1} S_{\lambda_1,\lambda_2,z}, \quad  D_z = \begin{bmatrix} z & 0 \\ 0 & 1 \end{bmatrix}. 
\end{equation}

\begin{coro} \label{coro:UAMOtoSzego}
	For all irrational $\Phi$, $\lambda_1 \in (0,1]$, $\lambda_2 \in [0,1]$, and $z \in \partial \bbD$, \begin{equation} 
		L(\Phi,A_{\lambda_1,\lambda_2,z}) 
		= L(\Phi,S_{\lambda_1,\lambda_2,z})
		= L(\Phi,G_{\lambda_1,\lambda_2,z}).
	\end{equation}
\end{coro}
\begin{proof}
	This is immediate from \eqref{eq:szegoAzConjugacy} and \eqref{eq:szegoToGZ}.
\end{proof}

Together, Corollaries~\ref{coro:UAMOtoCMV} and \ref{coro:UAMOtoSzego} allow one to apply much of the theory of standard CMV matrices to generalized CMV matrices and in particular to the UAMO.
\bigskip

The remainder of the paper is organized as follows. In Section~\ref{sec:origins}, we describe the motivation behind the study of the model in question by relating it to a two-dimensional magnetic quantum walk in a uniform magnetic field. Section~\ref{sec:cocycle} analyzes the transfer matrix cocycle and in particular proves Theorem~\ref{t:cocycleREG} as well as Theorems~\ref{t:maintype}.\ref{t:maintype:lambda1=0} and \ref{t:maintype}.\ref{t:maintype:lambda2=0}. Section~\ref{sec:aubry} works out suitable versions of Aubry duality for the model and contains the proofs of Theorems~\ref{t:aubryViaSolutions} and \ref{t:aubryOperator}. Section~\ref{sec:contspec} discusses continuous spectrum (that is, the exclusion of eigenvalues) and contains the proof of Theorem~\ref{t:maintype}.\ref{t:maintype:gordon} and the ``purely continuous'' half of Theorem~\ref{t:maintype}.\ref{t:maintype:critical}. Section~\ref{sec:zeromeas} discusses the phenomenon of zero-measure Cantor spectrum in the critical case and in particular proves Theorems \ref{t:maintype}.\ref{t:maintype:criticalCantor} and~\ref{t:maintype}.\ref{t:maintype:critical}. Section~\ref{sec:localization} discusses localization in the supercritical region and spectral consequences, in particular proving Theorems~\ref{t:maintype}.\ref{t:maintype:super} and \ref{t:maintype}.\ref{t:maintype:sub}.

%% file: 03-motivation.tex

\section{Motivation: Two-Dimensional Magnetic Quantum Walks} \label{sec:origins}

Let us describe the motivation behind studying quantum walks with quasiperiodic coins as in \eqref{eq:moddefs:coinDef}. 
The reader who is not interested in the physical origins of the model could skip this section, but it explains why we study the one-dimensional quasiperiodic walks that we choose to study and it also provides some insight into why we could expect such walks to exhibit a suitable version of Aubry duality. 
Indeed, since the almost-Mathieu operator arises from the study of a two-dimensional tight-binding model of an electron subjected to an external magnetic field, the natural starting point of our model is a two-dimensional quantum walk subjected to an external magnetic field. We also direct the reader to \cite{MandelZhito1991CMP,Shubin} for additional insights about the relationship between self-adjoint one-dimensional quasiperiodic and two-dimensional magnetic operators.

As already mentioned, one can simply start from the model defined in \eqref{eq:moddefs:walkDef}, \eqref{eq:moddefs:slambdaDef},  and \eqref{eq:moddefs:coinDef}, but it is helpful to see how this arises physically, which also explains the choice of the model.

Let us recall the two-dimensional magnetic QW model of \cite{CFGW2020LMP}. 
The state space will be $\scrH_2 : = \ell^2(\bbZ^2) \otimes \bbC^2$ with orthonormal basis
\[\{\delta_{\underline n}^s := \delta_{\underline n} \otimes e_s: \underline n \in \bbZ^2, \ s \in \{\pm\} = \bbZ_2 \}.\]
As before, view $\bbC^2 = \ell^2(\bbZ_2)$ and write $e_+ = [1,0]^\top$ and $e_- = [0,1]^\top$. Fix a magnetic flux $\Phi$. In the symmetric gauge, the magnetic translations are given by \cite{CGWW2019JMP}
\begin{align}
 T_1 & = T_{1,\Phi} : \delta_{\underline n}^s \mapsto e^{-si  \pi \Phi n_2}\delta_{\underline n + s e_1}^s, \\
 T_2 & = T_{2,\Phi} : \delta_{\underline n}^s \mapsto e^{si  \pi \Phi n_1}\delta_{\underline n + s e_2}^s 
 \end{align}
Viewing $\scrH_2 = \ell^2(\bbZ^2) \oplus \ell^2(\bbZ^2)$ in the natural manner, we have
 \[ T_{j,\Phi} = \begin{bmatrix}
 U_{j,\Phi} & 0 \\ 0 & U_{j,\Phi}^*
\end{bmatrix}  \]
where $U_j=U_{j,\Phi}$ are the corresponding shifts on $\ell^2(\bbZ)$. Again in the symmetric gauge, one can take 
\begin{align*}
U_{1,\Phi}:\delta_{\vecn} \mapsto e^{-i\pi \Phi n_{2}} \delta_{\vecn + e_1}, \quad
U_{2,\Phi}:\delta_{\vecn} \mapsto e^{ i\pi \Phi n_{1}} \delta_{\vecn + e_2}.
\end{align*} 
 We now also introduce the coupling parameter $\lambda \in [0,1]$ to get the coupled operators $T_{j,\lambda,\Phi}$ given by
\begin{align*}
T_{j,\lambda,\Phi} = \begin{bmatrix} U_{j,\Phi} & 0 \\ 0 & 1 \end{bmatrix} 
\begin{bmatrix} \lambda & -\lambda' \\ \lambda' & \lambda \end{bmatrix} \begin{bmatrix}1 & 0 \\ 0 & U_{j,\Phi}^* \end{bmatrix}
= \begin{bmatrix}
\lambda U_j & -\lambda' \\ \lambda' & \lambda U_j^*
\end{bmatrix}
\end{align*}

Let us briefly explain the terminology. We call $\lambda$ a ``coupling constant'' to elucidate an analogy with the self-adjoint setting. Namely, $\lambda$ is a parameter that dictates how strongly neighboring sites interact with one another by mediating the strength of the shift relative to the coins. Thus, as $\lambda$ decreases from one to zero, neighboring sites interact less and less and hence are less coupled together. Moreover, the main model of the present paper experiences a phase transition in the coupling constants that is similar to the phase transition exhibited in the self-adjoint Harper's model and that is unitary for each choice of parameters, which in turn requires genuine care in the definition of the model.

 Denoting by 
\[ C_0 = \frac{1}{\sqrt{2}} \begin{bmatrix}
1 & i \\ -i & -1
\end{bmatrix}, \]
we then want to consider
\begin{equation*}
	W^{(2)} = W^{(2)}_{\lambda_1,\lambda_2,\Phi} = T_{1,\lambda_1,\Phi} C_0 T_{2,\lambda_2,\Phi} C_0,
\end{equation*}
see Figure \ref{fig:neigh}. The coin $C_0$ is chosen to be a conjugate of the Hadamard coin, which is a popular choice for an unbiased coin in the quantum walk setting. The particular conjugate is chosen so that our quantum walk has a relatively simple form, yet other choices are possible (see the recent experimental work \cite{WeidemannNature}).

Let us now pass to a universal setting -- the rotation algebra $\mathcal{A} = \mathcal{A}_\Phi$, which is the $C^*$-algebra generated by elements $u,v$ satisfying the commutation relation 
\begin{equation}
uv =e^{-2 \pi i \Phi} vu.
\end{equation}
One obtains $W^{(2)}$ from the matrix algebra $\mathcal{A}_\Phi \otimes \bbC^{2\times 2}$ via the following representation. Define $p_2:\mathcal{A}_\Phi \to \mathscr{B}(\ell^2(\bbZ^2))$ by $p_2(u) = U_{1,\Phi}$ and $p_2(v) = U_{2,\Phi}$. Here, the ``2'' refers to the dimension of the representation. Presently, we will define the representation $p_1$ carrying $\calA_\Phi$ into $\mathscr{B}(\scrH_1)$. Noting that one indeed has
\[U_{1,\Phi}U_{2,\Phi} = e^{-2\pi i \Phi}U_{2,\Phi}U_{1,\Phi}\]
we see that this is well-defined. One can then extend to the matrix algebra $\calA_\Phi\otimes \bbC^{2 \times 2}$ in the standard way: $p_2(\{u_{ij}\}) = \{p_2(u_{ij})\}$. One then obtains $W$ from the algebra via $W = p_2(w)$, where 
\begin{equation}\label{eq:abstract_walk}
w = 
\begin{bmatrix} \lambda_1 u & -\lambda_1' \\ \lambda_1' & \lambda_1 u^* \end{bmatrix} C_0 
\begin{bmatrix}\lambda_2 v & -\lambda_2' \\ \lambda_2' &  \lambda_2 v^* \end{bmatrix} C_0
\end{equation}

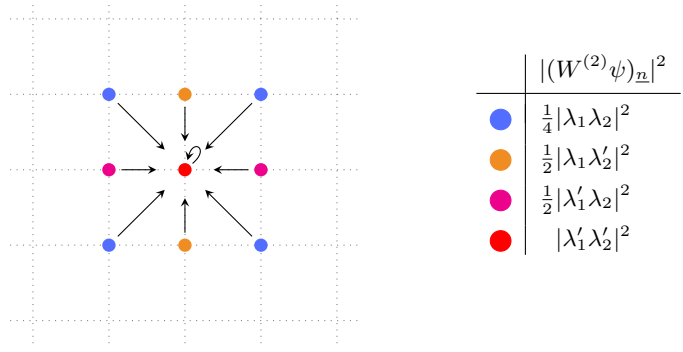
\begin{figure}[t]
	\begin{center}
		\input{plot_data/neighbourhood.tex}
	\end{center}
	\caption{\label{fig:neigh}Neighborhood structure and transition amplitudes for the walk $W^{(2)}$. }
\end{figure}

On the other hand, we can map $\mathcal{A}_\Phi$ into $\scrB(\ell^2(\bbZ))$ via $p_1(u) = S$ and $p_1(v) = M_{\Phi,\theta}$, where $S$ denotes the shift $\delta_n \mapsto \delta_{n+1}$ and $M = M_{\Phi,\theta}$ denotes the multiplication operator
\[M \delta_n  = e^{2\pi i (n \Phi+\theta)} \delta_n.\]
Note that $SM_\Phi  = e^{-2\pi i \Phi} M_\Phi S$, so this indeed defines a representation $\mathcal{A}_\Phi \to \scrB(\ell^2(\bbZ))$ and hence extends to a representation of the matrix algebras $p_1 : \mathcal{A}_\Phi \otimes \bbC^{2\times 2} \to \scrB(\ell^2(\bbZ) \otimes \bbC^2)$.
Thus, applying $p_1$ to $w$ yields the operator
\[ W = S_{\lambda_1} Q_{\lambda_2,\Phi,\theta} \]
where $Q_{\lambda_2,\Phi,\theta}$ is a coin operator with local coins
\begin{align*}  Q_n = Q_{\lambda_2,\Phi,\theta,n}
&  = \frac{1}{2} \begin{bmatrix} 1 & i \\ -i & -1 \end{bmatrix}
\begin{bmatrix}\lambda_2 e^{2\pi i (n\Phi + \theta)} & -\lambda_2' \\ \lambda_2' &  \lambda_2 e^{-2\pi i (n\Phi + \theta)} \end{bmatrix} 
\begin{bmatrix} 1 & i \\ -i & -1 \end{bmatrix} \\
& = \begin{bmatrix}
\lambda_2 \cos(2\pi(n\Phi+\theta)) + i\lambda_2'
& -\lambda_2 \sin(2\pi(n\Phi+\theta))\\ 
 \lambda_2 \sin(2\pi(n\Phi+\theta)) 
 & \lambda_2 \cos(2\pi(n\Phi+\theta)) - i\lambda_2'
\end{bmatrix}.
\end{align*}

%% file: plot_data/neighbourhood.tex
\def\dif{0.15}	

\definecolor{kardinal}{RGB}{168,0,128}
\definecolor{colup}{RGB}{239,140,34}
\definecolor{coldown}{RGB}{83,109,255}
\def\colup{colup}
\def\coldown{coldown}

\tikzstyle{dashing}=[dash pattern=on 1.5pt off 2pt]

\begin{tikzpicture}[font=\footnotesize,scale=1,auto]
	\tikzset{arr/.style = {->,>=stealth,shorten >=8pt,shorten  <=2pt}}

	\draw[gray,thin,dotted] (0,0)+(-2.3,-2.3) grid +(2.3,2.3);

	\node[shape=circle,fill=red,inner sep=0pt,minimum size=5pt] (00) at (0,0) {};

	\foreach \x in {-1,1}
	\foreach \y in {-1,1}
	{
		\node[shape=circle,fill=\coldown,inner sep=0pt,minimum size=5pt] (\x\y) at (\x,\y) {};
	}

	\foreach \x in {0}
	\foreach \y in {-1,1}
	{
		\node[shape=circle,fill=\colup,inner sep=0pt,minimum size=5pt] (\x\y) at (\x,\y) {};
	}

	\foreach \x in {-1,1}
	\foreach \y in {0}
	{
		\node[shape=circle,fill=magenta,inner sep=0pt,minimum size=5pt] (\x\y) at (\x,\y) {};
	}
	
	\foreach \x in {-1,0,1}
	\foreach \y in {-1,1}{
		\draw[arr] (\x\y) -- (00);
	}
	
	\foreach \x in {-1,1}{
		\draw[arr] (\x0) -- (00);
	}

	\draw[->, >=stealth, shorten >=1pt, shorten <=1pt] (00) to [out=40,in=75,loop,looseness=15] (00); 

	\def\legendposx{3.7}
	\def\legendposy{0.2}
	\def\circlerad{1ex}
	\def\circleshift{-0.5ex}
	\newcommand\Circle[1]{circle (#1)}
	\def\colorcircle#1{\tikz[baseline=\circleshift]\draw[#1,fill=#1] (0,0) \Circle{\circlerad};}

	\node[anchor=west] at (\legendposx,\legendposy) {
		\begin{tabular}{c|l}
			&	\multicolumn{1}{c}{$|(W^{(2)}\psi)_{\underline n}|^2$}		\\[.1cm]\hline\\[-.25cm]
			\colorcircle{\coldown}	&	$\frac14|\lambda_1\lambda_2|^2$		\\[.15cm]
			\colorcircle{\colup}	&	$\frac12|\lambda_1\lambda_2'|^2$	\\[.15cm]
			\colorcircle{magenta}	&	$\frac12|\lambda_1'\lambda_2|^2$	\\[.15cm]
			\colorcircle{red}	&	$\hphantom{\frac12}|\lambda_1'\lambda_2'|^2$
			\end{tabular}
	};

\end{tikzpicture}

%% file: 04-cocycle.tex

\section{Classification of Cocycle Behavior} \label{sec:cocycle}

In the following sections we aim at characterizations of the spectrum and spectral properties of $W_{\lambda_1,\lambda_2,\Phi,\theta}$.
This requires knowledge about the cocycle of its transfer matrices. 
We begin by describing the transfer matrix formalism for $W\psi = z\psi$ and the dual equation $W^\top \psi = z\psi$.
We then prove a lower bound on the Lyapunov exponent in the supercritical region via a Herman-type estimate. Afterwards, we completely classify cocycle behavior according to the relationship between the coupling constants.

\subsection{Transfer matrices} \label{sec:tms}

We begin by writing $W$ in coordinates.
\begin{lemma}\label{lem:w}
Suppose $\lambda \in (0,1]$ and $W = S_\lambda Q$ is a split-step walk with coupling constant $\lambda$ and coins $\{Q_n\}_{n\in\bbZ}$ as in \eqref{eq:moddefs:walkDef}.
For each $n \in \bbZ$, we have
	\begin{align} \label{eq:lemW+}
		[W\psi]_n^+ & = \lambda \left(q_{n-1}^{11}\psi_{n-1}^+ + q_{n-1}^{12}\psi_{n-1}^-\right)
		-\lambda' \left(q_{n}^{21}\psi_{n}^+ + q_{n}^{22}\psi_{n}^-\right),\\
		\label{eq:lemW-}
		[W\psi]_n^- & = \lambda \left(q_{n+1}^{21}\psi_{n+1}^+ + q_{n+1}^{22}\psi_{n+1}^-\right)
		+ \lambda' \left(q_{n}^{11}\psi_{n}^+ + q_{n}^{12}\psi_{n}^-\right).
	\end{align}
\end{lemma}

\begin{proof}
	Writing out the coordinates, we have
	\begin{align*}
		[W\psi]_n^+ & = [SQ\psi]_n^+ \\ 
		& =\lambda [Q\psi]_{n-1}^+ -\lambda' [Q\psi]_n^- \\
		& = \lambda \left(q_{n-1}^{11}\psi_{n-1}^+ + q_{n-1}^{12}\psi_{n-1}^-\right)
		-\lambda' \left(q_{n}^{21}\psi_{n}^+ + q_{n}^{22}\psi_{n}^-\right),
	\end{align*}
proving \eqref{eq:lemW+}. The proof of \eqref{eq:lemW-} is similar.
\end{proof}

Analogously, we can write $W^\top$ in coordinates.
\begin{lemma}\label{lem:wT}
Suppose $\lambda \in (0,1]$ and $W = S_\lambda Q$ is a split-step walk with coupling constant $\lambda$ and coins $\{Q_n\}_{n\in\bbZ}$ as in \eqref{eq:moddefs:walkDef}. For each $n \in \bbZ$, we have
	\begin{align}\label{eq:lemWT+}
		[W^\top\psi]_n^+ & = q_n^{11}(\lambda\psi_{n+1}^+ +\lambda'\psi_n^-)+q_n^{21}(-\lambda'\psi_n^++\lambda\psi_{n-1}^-)\\
		\label{eq:lemWT-}
		[W^\top\psi]_n^- & = q_n^{12}(\lambda\psi_{n+1}^++\lambda'\psi_n^-)+q_n^{22}(-\lambda'\psi_n^++\lambda\psi_{n-1}^-).
	\end{align}
\end{lemma}
\begin{proof}
This is almost identical to the proof of Lemma~\ref{lem:w}.
\end{proof}

\begin{prop} \label{prop:cocycle:tmMain}
Suppose $\lambda \in (0,1]$ and $W = S_\lambda Q$ is a split-step walk with coupling constant $\lambda$ and coins $\{Q_n\}_{n\in\bbZ}$ as in \eqref{eq:moddefs:walkDef}.
\begin{enumerate}[label={\rm (\alph*)}]
\item If $z \in \bbC \setminus \{0\}$, $W\psi = z\psi$, and $Q_n$ is not an off-diagonal matrix, then
\begin{equation} \label{eq:tm:main1}
	\begin{bmatrix}\psi_{n+1}^+\\\psi_{n}^-\end{bmatrix}
	= T_z(n)\begin{bmatrix}\psi_{n}^+\\\psi_{n-1}^-\end{bmatrix},
\end{equation}
where $T_z(n)$ is given by \eqref{eq:model:transmatDef}.
\smallskip

\item If $z \in \bbC \setminus \{0\}$, $W^\top \psi = z\psi$, and $Q_n$ is not an off-diagonal matrix, then
\begin{equation} \label{eq:tm:main1dual}
	\begin{bmatrix}\psi_{n+1}^+\\\psi_{n}^-\end{bmatrix}
	= \overline{T_{1/\bar{z}}(n)}
	\begin{bmatrix}\psi_{n}^+\\\psi_{n-1}^-\end{bmatrix}.
\end{equation}
\smallskip

\item If $z \in \bbC\setminus\{0\}$ and $Q_n$ is not off-diagonal,
\begin{equation}\label{eq:tm:main2}
\det T_z(n) = \frac{q_n^{11}}{q_n^{22}},
\end{equation}
which is unimodular.
\end{enumerate}
\end{prop}

\begin{proof}
\noindent (a) Beginning with Lemma~\ref{lem:w}, plug \eqref{eq:lemW+} into the eigenvector equation $W\psi=z\psi$ and shift indices $n\mapsto n+1$ to obtain
\begin{equation}\label{eq:wpsi+}
	z\psi_{n+1}^+ = \lambda \left(q_{n}^{11}\psi_{n}^+ + q_{n}^{12}\psi_{n}^-\right)
	-\lambda' \left(q_{n+1}^{21}\psi_{n+1}^+ + q_{n+1}^{22}\psi_{n+1}^-\right).
\end{equation}
Similarly, we obtain from \eqref{eq:lemW-}, $W\psi = z\psi$, and $n\mapsto n-1$
\begin{equation}\label{eq:wpsi-}
	z\psi_{n-1}^- = \lambda \left(q_{n}^{21}\psi_{n}^+ + q_{n}^{22}\psi_{n}^-\right) + \lambda'\left(q_{n-1}^{11}\psi_{n-1}^+ + q_{n-1}^{12}\psi_{n-1}^-\right).
\end{equation}
Then, $\eqref{eq:lemW-}\cdot\lambda'+\eqref{eq:wpsi+}\cdot\lambda$ and $\eqref{eq:lemW+}\cdot\lambda' - \eqref{eq:wpsi-}\cdot \lambda$ give 
\begin{equation} \label{eq:tm:eigeqDerived1}
\lambda' z\psi_n^- + \lambda z\psi_{n+1}^+ 
=\left(\lambda^2 +{\lambda'}^2\right)\left( q_{n}^{11}\psi_{n}^+ + q_{n}^{12}\psi_{n}^-\right)
= q_{n}^{11}\psi_{n}^+ + q_{n}^{12}\psi_{n}^-
\end{equation}
and
\begin{equation} \label{eq:tm:eigeqDerived2}
\lambda' z \psi_n^+ - \lambda z \psi_{n-1}^-
= -\left(\lambda^2 +{\lambda'}^2\right)\left(q_{n}^{21}\psi_{n}^+ + q_{n}^{22}\psi_{n}^-\right)
=- q_{n}^{21}\psi_{n}^+ - q_{n}^{22}\psi_{n}^-,
\end{equation}
respectively. 
Solving \eqref{eq:tm:eigeqDerived2} for $\psi_n^-$
yields
\begin{equation} \label{eq:tm:eigeqDerived3} \psi_n^- = \frac{1}{q_n^{22}}\left(\lambda z \psi_{n-1}^- -(q_n^{21} + \lambda' z) \psi_n^+\right),
\end{equation}
which is the bottom row of \eqref{eq:tm:main1}. 
Note that this step uses the assumption that $Q_n$ is not off-diagonal, that is, $|q_n^{11}| = |q_n^{22}| \neq 0$.
Solving \eqref{eq:tm:eigeqDerived1} for $\psi_{n+1}^+$ and inserting \eqref{eq:tm:eigeqDerived3} produces
\begin{align*}
\psi_{n+1}^+ 
& = \frac{1}{\lambda z} \left( q_n^{11} \psi_n^+ +(q_n^{12}-\lambda' z)\psi_n^- \right) \\
& = \frac{1}{\lambda z} \left( q_n^{11} \psi_n^+ +(q_n^{12}-\lambda' z) \left( \frac{1}{q_n^{22}}\left(\lambda z \psi_{n-1}^- -(q_n^{21} + \lambda' z) \psi_n^+\right) \right) \right) \\
& = \frac{1}{\lambda z q_n^{22}} \left(( q_n^{11}q_n^{22} - (q_n^{12}-\lambda' z)(q_n^{21} + \lambda' z) ) \psi_n^+ +(q_n^{12}-\lambda' z)\lambda z \psi_{n-1}^- \right) \\
& = \frac{1}{q_n^{22}} \left(\lambda^{-1}z^{-1} \det Q_n +\lambda'\lambda^{-1}(q_n^{21}-q_n^{12}) + z{\lambda'}^2\lambda^{-1} ) \psi_n^+ +(q_n^{12}-\lambda' z)  \psi_{n-1}^- \right),
\end{align*}
concluding the proof of \eqref{eq:tm:main1}. 
\smallskip

\noindent (b) Since $Q$ is unitary and $S_\lambda$ is real-symmetric, $Q^\top = \overline{Q}^{*}=\overline{Q}^{-1}$ and $S_\lambda^\top = S_\lambda^* = S_\lambda^{-1}$.
Thus, \eqref{eq:tm:main1dual} follows from \eqref{eq:tm:main1} and noting
\begin{align*}
W^\top \psi = z \psi
\iff Q^\top S_\lambda^\top \psi = z\psi
\iff z^{-1} \psi = S_\lambda \overline{Q} \psi.
\end{align*}
\smallskip

\noindent (c) For the determinant, we have
\begin{align*}
\det T_z(n)
& = \frac1{[q^{22}_n]^2} \det \begin{bmatrix}\lambda^{-1}z^{-1}\det Q_n+\lambda'\lambda^{-1}(q_n^{21}-q_n^{12})+z{\lambda'}^2\lambda^{-1}	&	q_n^{12}-\lambda'z	\\	-q_n^{21}-\lambda'z	&	\lambda z	\end{bmatrix} \\
& =\frac1{[q^{22}_n]^2} \left( \left(\det Q_n +\lambda'z(q_n^{21}-q_n^{12})+z^2{\lambda'}^2 \right) -	 (q_n^{12}-\lambda'z)(-q_n^{21}-\lambda'z) \right) \\
& =\frac1{[q^{22}_n]^2} (\det Q_n + q_n^{12}q_n^{21}) \\
& = \frac{q_n^{11}}{q_n^{22}},
\end{align*}
proving \eqref{eq:tm:main2}. Unimodularity of $\det T_z(n)$ then follows from unitarity of $Q_n$.
\end{proof}

As a consequence of this formalism, let us note the following result for the primary model of the manuscript.

\begin{prop} \label{prop:cocycle:pastur}
If $0< \lambda_1 \leq 1$, $0 \leq \lambda_2 \leq 1$, and $\Phi$ is irrational, then: 
\begin{enumerate}[label={\rm (\alph*)}]
\item \label{prop:cocycle:pastur:eig} For each $z \in \partial \bbD$,
\[\set{\theta \in \bbT : z \text{ is an eigenvalue of } W_{\lambda_1, \lambda_2, \Phi, \theta}}\]
has zero Lebesgue measure.
\smallskip

\item \label{prop:cocycle:pastur:iso} $\Sigma_{\lambda_1, \lambda_2, \Phi}$ has no isolated points.
\end{enumerate}
\end{prop}

\begin{proof}
This is a well-known argument using ergodicity; compare \cite{CFKS, Pastur1980CMP}. Since our setting is slightly different than the other settings in which this has been proved, we include the details for the reader's convenience.
Indeed, the spectral projector $\chi_{\{z\}}(W_{\lambda_1, \lambda_2, \Phi, \theta})$ is a weakly measurable function of $\theta \in \bbT$ that is covariant with respect to the shift $U:\delta_n^\pm \mapsto \delta_{n+1}^\pm$, on account of the identity:
\[U W_{\lambda_1, \lambda_2, \Phi, \Phi + \theta} U^* = W_{\lambda_1, \lambda_2, \Phi, \theta}. \] 
By ergodicity, the trace of $\chi_{\{z\}}(W_{\lambda_1, \lambda_2, \Phi, \theta})$ is almost-surely constant in $\theta$, and the almost-sure constant value must be $0$ or $\infty$.
Since $0 < \lambda_1$, Proposition~\ref{prop:cocycle:tmMain} implies the eigenspace corresponding to eigenvalue $z$ is always finite-dimensional, and hence we have $\chi_{\{z\}}(W_{\lambda_1, \lambda_2, \Phi, \theta}) = 0$ for a.e.\ $\theta$, proving part~\ref{prop:cocycle:pastur:eig}. Since an isolated point of the spectrum is necessarily an eigenvalue, part~\ref{prop:cocycle:pastur:iso} follows immediately.
\end{proof}

\begin{remark}
In fact, the reader can check that the conclusions \ref{prop:cocycle:pastur:eig} and \ref{prop:cocycle:pastur:iso} of Proposition~\ref{prop:cocycle:pastur} are true when $\lambda_1 = 0$ as long as $\lambda_2 > 0$. 
Of course, one can check that $\Sigma_{0,0,\Phi} = \{\pm i\}$ and $\pm i$ are eigevalues of infinite multiplicity for arbitrary phase, and hence the conclusions do not extend to $\lambda_1 = \lambda_2 = 0$.
\end{remark}

At this point, one can prove the trivial parts of Theorem~\ref{t:maintype}.

\begin{proof}[Proof of Theorems~\ref{t:maintype}.\ref{t:maintype:lambda1=0}, and \ref{t:maintype}.\ref{t:maintype:lambda2=0}]
Suppose first $\lambda_2 = 0$ and $\lambda_1>0$. 
The transfer matrix cocycle is constant and given by
\[ A_{\lambda_1, 0,z}(\theta) = i \begin{bmatrix} \lambda_1^{-1}z^{-1} + z{\lambda_1'}^2\lambda_1^{-1} & -\lambda_1'z \\ -\lambda_1'z & \lambda_1 z \end{bmatrix}
= i\begin{bmatrix} \lambda_1^{-1}z^{-1} + z(\lambda_1^{-1} - \lambda_1) & -\lambda_1'z \\ -\lambda_1'z & \lambda_1 z \end{bmatrix}. \]
In particular, the spectral type is purely absolutely continuous by Floquet theory and the spectrum may be computed from the discriminant.
Normalizing the determinant, we have
\begin{equation}
 \tr(-i A_{\lambda_1, 0 , z}(\theta)) = 2 \ \Re(\lambda_1^{-1} z), \quad z \in \partial\bbD, \end{equation}
and hence
\begin{align*}
\Sigma_{\lambda_1,0,\Phi} 
= \{z \in \partial \bbD : \tr(-i A_{\lambda_1, 0 , z}(\theta)) \in [-2,2]\} 
& = \{z \in \partial \bbD : \Re(\lambda_1^{-1}z) \in [-1,1] \} \\
& = \{z \in \partial \bbD : |\Re(z)| \leq \lambda_1 \},
\end{align*}
as desired.
Now, suppose $\lambda_1 = 0$ (and note that the the cocycle is no longer well-defined).
In this case, the shift is given  by $S_0\delta_n^\pm = \pm \delta_n^\mp$, so $W = S_0 Q_{\lambda_2,\Phi,\theta}$ is a direct sum of $2\times 2$ blocks of the form
\begin{equation}
Q_{\lambda_2,\Phi,\theta, n}' = 
\begin{bmatrix}
-\lambda_2 \sin(2\pi(n\Phi+\theta))
& - (\lambda_2 \cos(2\pi(n\Phi+\theta)) - i\lambda_2') \\
\lambda_2 \cos(2\pi(n\Phi+\theta)) + i\lambda_2'
& - \lambda_2 \sin(2\pi(n\Phi+\theta))
\end{bmatrix},
\end{equation}
and hence has pure point spectral type.

It is straightforward to check that the eigenvalues of $Q_{\lambda_2,\Phi,\theta, n}'$ are 
\[ -\lambda_2 \sin(2\pi(n\Phi+\theta)) \pm i \sqrt{{\lambda_2'}^2+\lambda_2^2 \cos^2(2\pi(n\Phi+\theta))}.\]
This shows the spectrum of $W_{0,\lambda_2,\Phi,\theta}$ is $\{ z \in \partial \bbD : |\Re(z)| \leq \lambda_2\}$ whenever $\Phi$ is irrational.
\end{proof}

\subsection{The Herman Estimate}

Let us show that the Lyapunov exponent is uniformly positive throughout the region $\lambda_1<\lambda_2$.

\begin{theorem} \label{thm:posle}
	Let $\Phi$ be irrational. For all $\lambda_1,\lambda_2,z$, we have
\begin{equation} \label{eq:lyap:hermanbound}
L_{\lambda_1,\lambda_2,\Phi}(z) \geq \log\left[ \frac{\lambda_2(1+\lambda_1')}{\lambda_1(1+\lambda_2')} \right].
\end{equation}
In particular, $L_{\lambda_1,\lambda_2,\Phi}(z) > 0$ for all $z \in \partial \bbD$ whenever $0< \lambda_1 < \lambda_2 \leq 1$ and
\begin{equation}
\lim_{\lambda_1\downarrow 0} L_{\lambda_1,\lambda_2,\Phi}(z)
= \infty.
\end{equation}
for all $z \in \partial \bbD$, $\lambda_2>0$, and irrational $\Phi$.
\end{theorem}

Of course, if $\lambda_2 \leq \lambda_1$, the right-hand side of \eqref{eq:lyap:hermanbound} is nonpositive and hence the content of the theorem is empty in that case.

The central argument in the proof of Theorem~\ref{thm:posle} is Herman's argument via complexification and subharmonicity \cite{herman}.
To get the sharpest possible inequality, we need to compute exactly a specific integral.
We only need the $\varepsilon=0$ case for the Herman estimate, but we will need the calculation for nonzero $\varepsilon$ for the eventual classification of cocycle behavior.

\begin{lemma} \label{lem:le:logIntegral}
For all $t \in [0,1]$, denote $\varepsilon_0=\varepsilon_0(t) = \frac{1}{2\pi} \mathrm{arcsinh}\sqrt{t^{-2}-1}$.
For $t \in [0,1]$ and $\varepsilon \in \bbR$, one has
\begin{align} \label{eq:specan:logIntegral}
\int_0^1 \log\left|t\cos(2\pi (\theta+i\varepsilon)) - i\sqrt{1-t^2} \right| \, d\theta 
& = \begin{cases}
 \log\left[\frac{1+\sqrt{1-t^2}}{2} \right] - 2\pi(\varepsilon + \varepsilon_0) & \varepsilon \leq -\varepsilon_0 \\[2mm]
 \log\left[\frac{1+\sqrt{1-t^2}}{2} \right] & - \varepsilon_0 \leq \varepsilon  \leq \varepsilon_0 \\[2mm]
  \log\left[\frac{1+\sqrt{1-t^2}}{2} \right] + 2\pi(\varepsilon - \varepsilon_0) & \varepsilon \geq \varepsilon_0.
\end{cases} \\[2mm]
& = \log\left[ \frac{1+\sqrt{1-t^2}}{2} \right] + 2\pi \max\{0,|\varepsilon|-\varepsilon_0\}.
\end{align}
\end{lemma}

\begin{proof}
When $t=0$, one has $\sqrt{1-t^2}=1$ and $\varepsilon_0(0)=\infty$, so that both sides of  \eqref{eq:specan:logIntegral} are zero for all $\varepsilon$. Thus, the $t=0$ case is trivial. For $0<t \leq 1$, define
\[ g(z) = \frac{t}{2} e^{-2\pi \varepsilon}z^2 - i  \sqrt{1-t^2} z + \frac{t}{2} e^{2\pi \varepsilon} \]
so that $|g(e^{2\pi i\theta})| = |e^{-2\pi i \theta}g(e^{2\pi i \theta})| =  |t \cos(2\pi(\theta+i\varepsilon)) - i  \sqrt{1-t^2}|$. The desired result follows from Jensen's formula applied to $g$.
\end{proof}

With the desired integral in hand, we now prove the theorem. Of course, the Herman estimate will be superseded by the global theory classification; since it is short and self-contained, we include the proof of positivity for the convenience of the reader who is unfamiliar with the Herman argument.

\begin{proof}[Proof of Theorem~\ref{thm:posle}]
Let $\lambda_1$, $\lambda_2$, $\Phi$, and $z$ be given. The transfer matrix cocycle of $W_{\lambda_1, \lambda_2, \Phi, \theta}$ is given by \eqref{eq:AlambdaCocycleDef}:
\begin{align*}
A&_{\lambda_1,\lambda_2,z}(\theta)	\\
&= \frac{1}{\lambda_2 \cos(2\pi \theta) - i\lambda_2'}
\begin{bmatrix}\lambda_1^{-1}z^{-1}+2\lambda_1'\lambda_1^{-1}\lambda_2 \sin(2\pi \theta) + z{\lambda_1'}^2\lambda_1^{-1}	&	-\lambda_2 \sin(2\pi \theta)-\lambda_1'z	\\	- \lambda_2 \sin(2\pi \theta) -\lambda_1'z	&	\lambda_1z	\end{bmatrix}.
\end{align*}
We view $\lambda_1$, $\lambda_2$, and $\Phi$ as fixed, so we suppress them from the subscripts throughout the argument, simply writing, for instance, $A_z(\theta)$ and $L(z)$ instead of $A_{\lambda_1,\lambda_2,z}(\theta)$ and $L_{\lambda_1,\lambda_2,\Phi}(z)$.
 Note that for fixed $z \in \partial \bbD$, this map is analytic (as a function of $\theta$) for $\lambda_2<1$ and meromorphic when $\lambda_2 = 1$.
		
Denote the regularized transfer matrices by $B_z(\theta)$:
	\begin{align}
	\label{eq:cocycle:BzRegularizedDef}
B_z(\theta)
& = \left[\frac{2(\lambda_2 \cos(2\pi \theta)-i\lambda_2')}{1+\lambda_2'} \right] A_z(\theta) \\
& = \frac{2}{\lambda_1(1+\lambda_2')}
\begin{bmatrix} z^{-1}+2\lambda_1'\lambda_2 \sin(2\pi\theta) + z{\lambda_1'}^2
&	-\lambda_1(\lambda_2 \sin(2\pi\theta) + \lambda_1'z)	\\	
-\lambda_1(\lambda_2 \sin(2\pi\theta)+\lambda_1'z	)
&	\lambda_1^2 z	\end{bmatrix}.
	\end{align}
	By Lemma~\ref{lem:le:logIntegral} (with $\varepsilon=0$), we note that
\begin{equation} \label{eq:lyap:LBztoLAzViaRes}
L(z) 
= \lim_{n\to\infty} \frac{1}{n} \int_0^1 \log\|A_z^n(\theta)\| \, d\theta
= \lim_{n\to\infty} \frac{1}{n} \int_0^1 \log\|B_z^n(\theta)\| \, d\theta
=: L(\Phi,B_z).
\end{equation}
Write $w = \exp(2\pi i \theta)$ and define $M_z(w)$ to mean $B_z$ recontextualized as a function of $w$, that is
\[
M_z(w) 
=  \frac{2}{\lambda_1(1+\lambda_2')} 
\begin{bmatrix} 
z^{-1}+2\lambda_1'\lambda_2 \left[\frac{w-w^{-1}}{2i}\right]+ z{\lambda_1'}^2
&	-\lambda_1(\lambda_2 \left[\frac{w-w^{-1}}{2i}\right] + \lambda_1'z)	\\	
-\lambda_1(\lambda_2 \left[\frac{w-w^{-1}}{2i}\right]+\lambda_1'z	)
&	\lambda_1^2 z	
\end{bmatrix}
\]
Finally, let $N_z(w) = wM_z(w)$, which one can check is an entire function $N_z:\bbC \to \GL(2,\bbC)$.
Notice that
\begin{equation} \label{eq:lyap:Nz0form}
N_z(0)
= \frac{i\lambda_2}{\lambda_1(1+\lambda_2')}
\begin{bmatrix} 2  \lambda_1'
&	- \lambda_1 	\\	
- \lambda_1 
&	0
\end{bmatrix}.
\end{equation}
Given $n \in \bbN$, we define the iterates of $M_z$ and $N_z$ in the usual manner (keeping in mind that we are writing the circle multiplicatively here):
\begin{equation*}
M_z^n(w)
:= \prod_{k=n-1}^0 M_z(e^{2\pi i k \Phi }w), \quad
N_z^n(w)
:= \prod_{k=n-1}^0 N_z(e^{2\pi i k \Phi }w).
\end{equation*}
In view of the identity $N^n_z(w) = w^n M_z^n(w)$ and the definion of $M_z$, one has
\begin{equation} \label{eq:lyap:normNktoNormBk}
\|N^n_z(e^{2\pi i \theta})\| 
= \|M^n_z(e^{2\pi i \theta})\| 
= \|B^n_z(\theta)\|
\end{equation}
for all $\theta \in \bbT$.
Since $N^n_z(\cdot)$ is analytic, it follows that $\log\|N^n_z(\cdot)\|$ is subharmonic.
	Using the definition of the Lyapunov exponent, \eqref{eq:lyap:normNktoNormBk}, subharmonicity of $\log\|N_z^n\|$, \eqref{eq:lyap:Nz0form}, and Gelfand's formula (in that order), we get
\begin{align}
\nonumber
L(\Phi,B_z) & = 
\lim_{n \to\infty} \frac{1}{n} \int_0^1 \log\|B_z^n(\theta)\| \, d\theta \\
\nonumber
& = \lim_{n \to\infty} \frac{1}{n} \int_0^1 \log\|N_z^n(e^{2\pi i\theta})\| \, d\theta\\
\nonumber
& \geq \limsup_{n \to\infty} \frac{1}{n} \log\|N_z^n(0)\|	\\[2mm]
\nonumber
& = \lim_{n \to\infty}\left[ \log \left[ \frac{\lambda_2}{\lambda_1(1+\lambda_2')}\right]+ \frac{1}{n} \log\left\|
\begin{bmatrix} 2\lambda_1'	&	-\lambda_1	\\	-\lambda_1	&	0	\end{bmatrix}^n\right\|	\right] \\[2mm]
\label{eq:lyap:LBtolog+spr}
& = \log \left[ \frac{\lambda_2}{\lambda_1(1+\lambda_2')}\right]+  \log \spr
\begin{bmatrix} 2\lambda_1'	&	-\lambda_1	\\	-\lambda_1	&	0	\end{bmatrix},
	\end{align}
	where $\mathrm{spr}$ denotes the spectral radius. One can check that the matrix in the last expression has eigenvalues $\lambda_1' \pm 1$, which yields 
\begin{equation} \label{eq:sprBlambda1}
\spr\begin{bmatrix} 2\lambda_1'	&	-\lambda_1	\\	-\lambda_1	&	0	\end{bmatrix} = 1 + \lambda_1'.
\end{equation} 
Combining \eqref{eq:sprBlambda1} with \eqref{eq:lyap:LBtolog+spr} and putting this together with \eqref{eq:lyap:LBztoLAzViaRes}, we have obtained
\begin{equation} 
L_{\lambda_1,\lambda_2,\Phi}(z) = L(\Phi,B_z) \geq
\log \left[ \frac{\lambda_2(1 + \lambda_1')}{\lambda_1(1 + \lambda_2')} \right],
\end{equation}
which concludes the proof.
\end{proof}

\subsection{Analytic One-Frequency Cocycles: A Brief Review}

Let $A:\bbT \to \bbC^{2\times 2}$ be continuous and $\Phi \in \bbT$ irrational.
Recall from \eqref{eq:cociterates} that the iterates of $A$ are given by
\begin{equation}
A^{n,\Phi}(\theta) = A((n-1)\Phi+\theta) \cdots A(\Phi+\theta)A(\theta), \quad n \in \bbN
\end{equation} 
and the Lyapunov exponents of the cocycle $(\Phi,A)$ are given by
\begin{align}
L_1(\Phi,A) & = \lim_{n\to\infty} \frac{1}{n} \int_{\bbT} \log \| A^{n,\Phi}(\theta)\|\, d\theta \\
\label{eq:cocycle:L2def}
L_2(\Phi,A) & = \lim_{n\to\infty} \frac{1}{n} \int_{\bbT} \log \| [A^{n,\Phi}(\theta)]^{-1} \|^{-1} \, d\theta.
\end{align}
In \eqref{eq:cocycle:L2def}, $\|\cdot^{-1}\|^{-1}$ simply computes the smaller singular value and hence one should interpret $\|A^{-1}\|^{-1} = 0$ if $A$ is not invertible. 
We will focus almost exclusively on the upper Lyapunov exponent and hence simply write $L(\Phi,A)$ instead of $L_1(\Phi,A)$.

\begin{definition}
We say that $(\Phi,A)$ enjoys a \emph{dominated splitting} if there is a continuous splitting $\bbC^2 = \Lambda_1(\theta) \oplus \Lambda_2(\theta)$ into one-dimensional subspaces with
\[ A(\theta)\Lambda_j(\theta) \subseteq \Lambda_j(\theta+\Phi) \]
and such that there exists $k \in \bbN$ such that
\[ \|A^{k,\Phi}(\theta) v_1 \| > \|A^{k,\Phi}(\theta) v_2 \| \]
for all $\theta \in \bbT$ and all unit vectors $v_j \in \Lambda_j(\theta)$
\end{definition}

If in addition $A$ is analytic with an analytic extension to a strip $\{ \theta + i \varepsilon : \theta \in \bbT, \ |\varepsilon| < \delta \}$, we may consider the complexified cocycle map and its Lyapunov exponent as in \eqref{eq:complexifiedMdef} and \eqref{eq:complexifiedLEdef}: $L(\Phi,A,\varepsilon) = L(\Phi,A(\cdot+i\varepsilon))$.

\begin{definition}
The \emph{acceleration} of $A$ is defined by
\[\omega(A) = \lim_{\varepsilon \downarrow 0} \frac{1}{2\pi \varepsilon} (L(\Phi,A,\varepsilon) - L(\Phi,A)).\]
More generally, for $|\eta|<\delta$ we also denote
\[\omega(A,\eta) = \omega(A(\cdot+i\eta))
= \lim_{\varepsilon\downarrow 0} \frac{1}{2\pi\varepsilon} (L(\Phi,A,\eta+\varepsilon) - L(\Phi,A,\eta)).\]
\end{definition}

Let us collect the main properties of this apparatus.
For details and proofs, we direct the reader to Avila \cite{Avila2015Acta} (who proved the result for unimodular cocycles) and Jitomirskaya--Marx \cite{JitoMarx2012CMP, JitoMarx2012CMPErr} (who extended the result to singular cocycles). See also Avila--Jitomirskaya--Sadel for the generalization to higher dimensions \cite{AJS2014JEMS}.

\begin{theorem}[Avila \cite{Avila2015Acta}, Jitomirskaya--Marx \cite{JitoMarx2012CMP, JitoMarx2012CMPErr}] \label{t:avilaGlobal}
Suppose $A:\bbT \to \bbC^{2\times 2}$ is analytic with an analytic extension to the strip $\{ \theta + i \varepsilon : \theta \in \bbT, \ |\varepsilon| < \delta \}$.
\begin{enumerate}[label={\rm (\alph*)}]
\item The function $\varepsilon \mapsto L(\Phi,A,\varepsilon)$ is continuous, convex, and piecewise affine.
\smallskip
 
\item $\omega(A,\eta) \in \frac{1}{2}\bbZ$ for all $|\eta|<\delta$.
\smallskip

\item If $A(\theta) \in \SL(2,\bbC)$ for all $\theta \in \bbT$, then $\omega(A,\eta) \in \bbZ$ for all $|\eta|<\delta$.
\smallskip

\item If $L_1(\Phi,A)>L_2(\Phi,A)$ and $\varepsilon \mapsto L(\Phi,A,\varepsilon)$ is affine on a neighborhood of $\varepsilon = 0$, then $(\Phi,A)$ enjoys a dominated splitting.
\end{enumerate}
\end{theorem}

As soon as one has an \emph{invertible} cocycle $A:\bbT \to \GL(2,\bbC)$, one can attempt to push into the unimodular setting by considering the normalized cocycle $A/\sqrt{\det A}$.
Thus, as discussed in \cite{JitoMarx2012CMPErr}, for \emph{invertible} (that is, $\GL(2,\bbC)$) cocycles, the central issue is whether $\det(A)$ enjoys  holomorphic square root, and this is precisely where the ``$1/2$'' in Theorem \ref{t:avilaGlobal} comes from. Namely, if $\det(A)$ has a holomorphic square root, then one can  apply Avila's result to $A/\sqrt{\det A}$ without further complication.
Otherwise, one can check that $\det(A)$ enjoys a holomorphic square root of period $2$ and hence one can apply Avila's work to a cocycle with doubled period.

Aside from the importance as an idea in dynamical systems, the notion of a dominated splitting plays a crucial role in the current paper by determining the complement of the spectrum. To spell this out in detail, we need to define the regularized GZ cocycle. Concretely, put
\[
 \widetilde P(\alpha,z) = \begin{bmatrix} - \overline{\alpha} & z \\ 1/z & -{\alpha} \end{bmatrix},
\quad \widetilde  Q(\alpha,z) =\begin{bmatrix} - \alpha & 1 \\ 1 & -\overline{\alpha} \end{bmatrix}, \]
 define
\begin{equation} \label{eq:tildeGlamdef}
\widetilde{G}_{\lambda_1,\lambda_2,z} = P(\lambda_1',z)Q(\lambda_2 \sin(2\pi\theta),z),
\end{equation}
and note that $\widetilde  G_{\lambda_1,\lambda_2,z}$ arises from $G_{\lambda_1,\lambda_2,z}$ in \eqref{eq:GZcocycledef} by clearing denominators.

As a consequence of a general result in \cite[Theorem~6.1]{FOZ2017CMP}, one has the following characterization of the almost-sure spectrum:
\begin{theorem} \label{thm:specviaDS}
If $\Phi$ is irrational, then
\begin{align}
\Sigma_{\lambda_1,\lambda_2,\Phi} 
& = \set{z \in \partial \bbD : (\Phi,\widetilde{G}_{\lambda_1,\lambda_2,z}) \text{ does not enjoy a dominated splitting}} \\
& = \set{z \in \partial \bbD : (\Phi,B_{\lambda_1,\lambda_2,z}) \text{ does not enjoy a dominated splitting}}.
\end{align}
\end{theorem}
\begin{proof}
The first line is an immediate consequence of \cite[Theorem~6.1]{FOZ2017CMP}, (and the fact that translation by $\Phi$ is a strictly ergodic isometry of $\bbT$) while the second follows from the definitions of and relationship between the various cocycles. Concretely, apply \eqref{eq:Blambdaearlydef}, \eqref{eq:szegoAzConjugacy}, \eqref{eq:szegoToGZ}, and \eqref{eq:tildeGlamdef} in that order to see that $(\Phi,B_{\lambda_1,\lambda_2,z})$ enjoys a dominated splitting if and only if $(\Phi,\widetilde G_{\lambda_1,\lambda_2,z})$ does.
\end{proof}

\subsection{Classification of Cocycles for the UAMO}

Recall that the cocycle map $A_z = A_{\lambda_1, \lambda_2,z}$ is given by \eqref{eq:AlambdaCocycleDef} and the regularized cocyle $B_z = B_{\lambda_1,\lambda_2,z}$ is given by \eqref{eq:cocycle:BzRegularizedDef}.
As before, to keep the notation cleaner in this section, we view $\lambda_1$ and $\lambda_2$ as fixed and suppress them from the notation, writing $A_z(\theta)$ and $B_z(\theta)$ for $A_{\lambda_1, \lambda_2,z}(\theta)$ and $B_{\lambda_1,\lambda_2,z}(\theta)$.

Let us point out the following: $B_z$ is always analytic, but (as we will see shortly), $A_z$ has singularities at some complex phases.
We will see concrete manifestations of this in the graph of $L(\Phi, A_z, \varepsilon)$, which we will see is \emph{not} convex. However, it is convex on subintervals of the $\varepsilon$-axis that avoid values of $\varepsilon$ for which $A_z(\cdot  + i\varepsilon)$ has singularities. We will  see that there are two such values when $0<\lambda_2<1$ and only one when $\lambda_2=1$, and moreover the exceptional values of $\varepsilon$ are explicit functions of the coupling constant attached to the coin sequence. Indeed, let us define $\varepsilon_0$ by
\begin{equation} \label{eq:cocycle:eps0Def}
\sinh(2\pi\varepsilon_0) = \lambda_2'/\lambda_2,
\end{equation}
 (formally allowing $\varepsilon_0=\infty$ when $\lambda_2=0$). Recalling Lemma~\ref{lem:le:logIntegral}, the Lyapunov exponents of the complexifications of the cocycles $A_z$ and $B_z$ are related via
\begin{equation} \label{eq:cocycle:LAtoLB}
L(\Phi,B_z,\varepsilon)
= L(\Phi,A_z,\varepsilon) + 2\pi \max\{|\varepsilon|-\varepsilon_0,0\}.
\end{equation}
In particular, $L(\Phi,A_z,\varepsilon) = L(\Phi,B_z,\varepsilon)$ whenever $|\varepsilon| \leq \varepsilon_0$.

We begin with our first key technical result: by Proposition~\ref{prop:cocycle:tmMain}
\begin{equation} \label{eq:AlambdaDet}
\det A_{\lambda_1,\lambda_2,z}(\theta) = \frac{\lambda_2 \cos(2\pi \theta) + i\lambda_2'}{\lambda_2 \cos(2\pi \theta) - i \lambda_2'},
\end{equation} 
so $A_z$ is not always a unimodular cocycle. 
Nevertheless, the acceleration of $A_z$ and the acceleration of all cocycles arising from $A_z$ via complexification of the phase are always integer-valued.

\begin{lemma}
Fix $0<\lambda_1 \leq 1$,  $0 \leq \lambda_2 \leq 1$, $\Phi \in \bbT\setminus \bbQ$, and $z \in \partial \bbD$, and let $A = A_{\lambda_1,\lambda_2,z}$ and $B_z = B_{\lambda_1,\lambda_2,z}$. We have
 \begin{equation} \label{eq:cocycle:accelAB}
 \omega(A,\varepsilon) \in \bbZ \ \forall \varepsilon \neq \pm \varepsilon_0 , \quad \text{ and } \quad \omega(B,\varepsilon) \in \bbZ \ \forall \varepsilon \in \bbR, \end{equation}
 where $\varepsilon_0$ is as in \eqref{eq:cocycle:eps0Def}.
\end{lemma}
\begin{proof}
\noindent \textbf{Case 1: \boldmath $\lambda_2=0$.} In this case, $A$ and $B$ are constant cocycles, so the result follows.
\medskip

\noindent \textbf{Case 2: \boldmath $0< \lambda_2 < 1$.} In view of \eqref{eq:cocycle:LAtoLB}, it suffices to prove $\omega(A,\varepsilon)\in \bbZ$ for all $\varepsilon \neq \pm \varepsilon_0$.
 Write $\xi = \theta+i\varepsilon$, recall \eqref{eq:AlambdaDet}, and apply the sum identity for cosine to get
\begin{align} \det(A(\xi))
& = \frac{\lambda_2\cos(2\pi\xi) + i\lambda_2'}{\lambda_2\cos(2\pi\xi) - i\lambda_2'} \\[2mm]
 \label{eq:cocycle:cxADet}
& = \frac{\lambda_2 \cos(2\pi\theta)\cosh(2\pi \varepsilon)+i(\lambda_2' - \lambda_2 \sin(2\pi\theta)\sinh(2\pi \varepsilon))}{\lambda_2 \cos(2\pi\theta)\cosh(2\pi \varepsilon) - i(\lambda_2' + \lambda_2 \sin(2\pi\theta)\sinh(2\pi \varepsilon))}. 
\end{align}
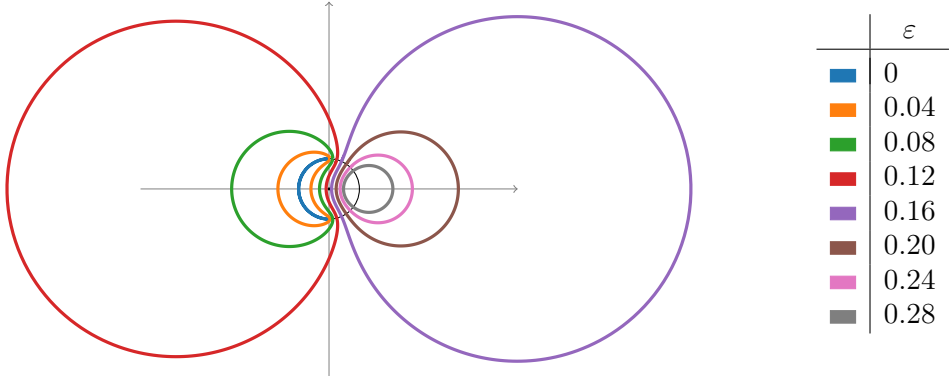
\begin{figure}
	\begin{center}
		\begin{tikzpicture}[scale=.4]
		
		\definecolor{col1}{RGB}{31, 119, 180}
		\definecolor{col2}{RGB}{255, 127, 14}
		\definecolor{col3}{RGB}{44, 160, 44}
		\definecolor{col4}{RGB}{214, 39, 40}
		\definecolor{col5}{RGB}{148, 103, 189}
		\definecolor{col6}{RGB}{140, 86, 75}
		\definecolor{col7}{RGB}{227, 119, 194}
		\definecolor{col8}{RGB}{127, 127, 127}
		\definecolor{col9}{RGB}{188, 189, 34}
		\definecolor{col10}{RGB}{23, 190, 207}
		
		\draw[help lines,->] (-6.2,0) -- (6.2,0);
		\draw[help lines,->] (0,-6.2) -- (0,6.2);
		
		\draw (0,0) circle (1);
		
		\fill (0,0)  circle (.05);
		
		\draw[col1, very thick]	plot [smooth cycle] file {plot_data/det_plots/det_plots_0.000.tab};
		\draw[col2, very thick]	plot [smooth cycle] file {plot_data/det_plots/det_plots_0.040.tab};
		\draw[col3, very thick]	plot [smooth cycle] file {plot_data/det_plots/det_plots_0.080.tab};
		\draw[col4, very thick]	plot [smooth cycle] file {plot_data/det_plots/det_plots_0.120.tab};
		\draw[col5, very thick]	plot [smooth cycle] file {plot_data/det_plots/det_plots_0.160.tab};
		\draw[col6, very thick]	plot [smooth cycle] file {plot_data/det_plots/det_plots_0.200.tab};
		\draw[col7, very thick]	plot [smooth cycle] file {plot_data/det_plots/det_plots_0.240.tab};
		\draw[col8, very thick]	plot [smooth cycle] file {plot_data/det_plots/det_plots_0.280.tab};
		
		\def\legendposx{15.7}
		\def\legendposy{.5}
		\def\linex{1ex}
		\def\liney{.5ex}
		\def\circleshift{-0.5ex}
		\newcommand\Line[2]{+(-#1,-#2) rectangle +(#1,#2)}
		\def\colorline#1{\tikz[baseline=\circleshift]\draw[#1,fill=#1] (0,0) \Line{\linex}{\liney};}
		
		\node[anchor=west] at (\legendposx,\legendposy) {\begin{tabular}{c|l}
			&	\multicolumn{1}{c}{$\varepsilon$}		\\\hline\\[-.35cm] 
			\colorline{col1}	&	0		\\
			\colorline{col2}	&	0.04	\\
			\colorline{col3}	&	0.08	\\
			\colorline{col4}	&	0.12	\\
			\colorline{col5}	&	0.16	\\
			\colorline{col6}	&	0.20	\\
			\colorline{col7}	&	0.24	\\
			\colorline{col8}	&	0.28	\\
		\end{tabular}};
				
		\end{tikzpicture}
		\caption{Range of $\theta \mapsto \det(A(\theta+i\varepsilon))$ for several values of $\varepsilon \neq \varepsilon_0$. Here, $\lambda_2=1/\sqrt{2}$ such that $\varepsilon_0=\sinh^{-1}(\lambda_2'/\lambda_2)/(2\pi)\approx.1403$. For this value of $\varepsilon$ the range of $\theta \mapsto \det(A(\theta+i\varepsilon))$ is a circle with infinite radius along the imaginary axis.
		}
	\end{center}
\end{figure}

Recall that $\varepsilon_0$ is given by \eqref{eq:cocycle:eps0Def}.
Note that the denominator of \eqref{eq:cocycle:cxADet} vanishes if and only if $\theta = 1/4 \ \mathrm{mod} \ \bbZ$ and $\varepsilon = -\varepsilon_0$ or $\theta = 3/4 \ \mathrm{mod} \ \bbZ$ and $\varepsilon = \varepsilon_0$.
In particular, the cocycle $A(\cdot + i\varepsilon)$ has singularities when $|\varepsilon|=\varepsilon_0$ and is analytic otherwise.
Similarly, the numerator of \eqref{eq:cocycle:cxADet} has a root if and only if $\varepsilon = \pm \varepsilon_0$.

Since $\lambda_2< 1$, a direct calculation (using Proposition~\ref{prop:cocycle:tmMain}) yields
\begin{equation}\label{eq:cocycle:detPM} 
\det(A(\theta)) \in \set{e^{is} :0< 2\arctan(\lambda_2'/\lambda_2)\leq s \leq 2(\pi - 2\arctan(\lambda_2'/\lambda_2)) < 2\pi}. \end{equation}
In particular, the image of $\theta \mapsto  \det(A(\theta))$ lies in a simply connected subset of $\bbC\setminus\{0\}$ and hence $\theta \mapsto \det(A(\theta))$ has trivial winding around the origin. 
Since $\det(A(\xi)) \neq 0$ for all $\xi$ with $|\Im(\xi)| <  \varepsilon_0$, 
\begin{equation} \label{eq:cocycle:winding}
\theta \mapsto \det A(\theta+i\varepsilon) \text{ has trivial winding aroung the origin}
\end{equation}
for all $|\varepsilon|<\varepsilon_0$.
Similarly, one can check that if $|\varepsilon|$ is sufficiently large,
\[\left|1 - \det(A(\theta+i\varepsilon)) \right| < \frac{1}{2} \quad  \text{ for all } \theta \in \bbT. \]
Consequently, $\theta\mapsto \det(A(\theta+i\varepsilon))$ has trivial winding around zero for such large $|\varepsilon|$, and hence \eqref{eq:cocycle:winding} also holds true for all $|\varepsilon|>\varepsilon_0$. 
Thus, for all $\varepsilon \neq \pm\varepsilon_0$, $\det(A(\theta+i\varepsilon,z))$ enjoys a holomorphic square root, so \eqref{eq:cocycle:accelAB} for $A$ and $\varepsilon\neq\pm \varepsilon_0$ follows by applying Theorem~\ref{t:avilaGlobal} to $A(\cdot,z)/\sqrt{\det A(\cdot,z)}$.
As discussed above, the claim for $B$ follows by applying \eqref{eq:cocycle:LAtoLB}; the extension to $\varepsilon = \pm \varepsilon_0$ follows from continuity and piecewise affinity of $\varepsilon\mapsto L(\Phi, B,\varepsilon)$.
\medskip

\noindent \textbf{\boldmath Case 3: $\lambda_2 =1$.} Notice that $A(\cdot + i\varepsilon)$ is an analytic $\SL(2,\bbC)$ cocycle for all $\varepsilon \neq 0$.
Thus, the result for $A$ at nonzero $\varepsilon$ follows immediately from Theorem~\ref{t:avilaGlobal}. The extension to $B$ uses \eqref{eq:cocycle:LAtoLB} and properties of $L(\Phi,B,\varepsilon)$ as before.
\end{proof}

The central idea is now to analyze the behavior as $\varepsilon \to \infty$, then use quantization of acceleration, continuity, and convexity to bring this information to $\varepsilon$ near zero.
The key technical result is the following asymptotic calculation for the regularized cocycle.
\begin{prop} \label{prop:cocycle:LBeAsymptotic}
Fix $0<\lambda_1 \leq 1$,  $0 \leq \lambda_2 \leq 1$, $\Phi \in \bbT\setminus \bbQ$, and $z \in \partial \bbD$, and let $B_z = B_{\lambda_1,\lambda_2,z}$.
For all $|\varepsilon|$ sufficiently large,
\begin{equation} \label{eq:cocycle:LBepsAsymp}
L(\Phi,B_z,\varepsilon) = 2\pi|\varepsilon|+ \log \left[ \frac{\lambda_2(1+\lambda_1')}{\lambda_1(1+\lambda_2')}\right].
\end{equation}
Moreover, $L(\Phi,B,\varepsilon)$ is an even function of $\varepsilon \in \bbR$.
\end{prop}

\begin{proof}
Define $\widetilde{B}_{z,\varepsilon}^\pm$ by
\begin{align*} 
\widetilde{B}_{z,\varepsilon}^\pm(\theta)
& = e^{\mp 2\pi \varepsilon} B_z(\theta + i\varepsilon) \\
& = 
\frac{2e^{\mp 2\pi \varepsilon}}{1+\lambda_2'}
\begin{bmatrix}\lambda_1^{-1}z^{-1}+2\lambda_1'\lambda_1^{-1}\lambda_2 \sin(2\pi (\theta + i\varepsilon)) + z{\lambda_1'}^2\lambda_1^{-1}	&	-\lambda_2 \sin(2\pi (\theta + i\varepsilon))-\lambda_1'z	\\	- \lambda_2 \sin(2\pi (\theta + i\varepsilon)) -\lambda_1'z	&	\lambda_1z	\end{bmatrix}. \end{align*}
Naturally, one has 
\[L(\Phi,\widetilde{B}_{z,\varepsilon}^\pm) 
= L(\Phi,B_z(\cdot + i\varepsilon))  \mp 2\pi \varepsilon\]
for any $\varepsilon$. Define
\[B_\infty(\theta) = \lim_{\varepsilon \to \pm \infty} \widetilde{B}_{z,\varepsilon}^\pm(\theta)
= \frac{2ie^{-2\pi i \theta}}{1+\lambda_2'}
 \begin{bmatrix}  \lambda_1'\lambda_1^{-1} \lambda_2   
 & -\lambda_2 /2 \\ 
 -\lambda_2 /2
 & 0
 \end{bmatrix}
 = \frac{i\lambda_2e^{-2\pi i \theta}}{\lambda_1(1+\lambda_2')} 
 \begin{bmatrix}  2\lambda_1'  
 & -\lambda_1  \\ 
 -\lambda_1 
 & 0
 \end{bmatrix}.\]

Notice that the convergence is uniform in $\theta$.
 Using \eqref{eq:sprBlambda1}, we can directly calculate the Lyapunov exponent of the cocycle $B_\infty$ via
 \begin{equation} \label{eq:cocycle:BinftLE}
 L(\Phi,B_\infty) = \log \mathrm{spr}(B_\infty(0))
 = \log\left[ \frac{\lambda_2(1+\lambda_1')}{\lambda_1(1+\lambda_2')} \right] .
 \end{equation}
By \eqref{eq:cocycle:BinftLE}, quantization of acceleration, convexity, and continuity of the Lyapunov exponent \cite{Avila2015Acta,JitoMarx2012CMP} we get
\begin{equation} \label{eq:cocyc:LBepsLargeAbsVal}
L(\Phi,B,\varepsilon)
 = \log\left[ \frac{\lambda_2(1+\lambda_1')}{\lambda_1(1+\lambda_2')} \right] + 2\pi |\varepsilon|
 \text{ for all } |\varepsilon| \text{ sufficiently large.}
 \end{equation}
By \eqref{eq:cocyc:LBepsLargeAbsVal}, convexity, and quantization of acceleration, $L(\Phi,B_z,\varepsilon)$ is even.
 \end{proof}

\newcolumntype{V}{>{\centering\arraybackslash} m{.4\linewidth} }

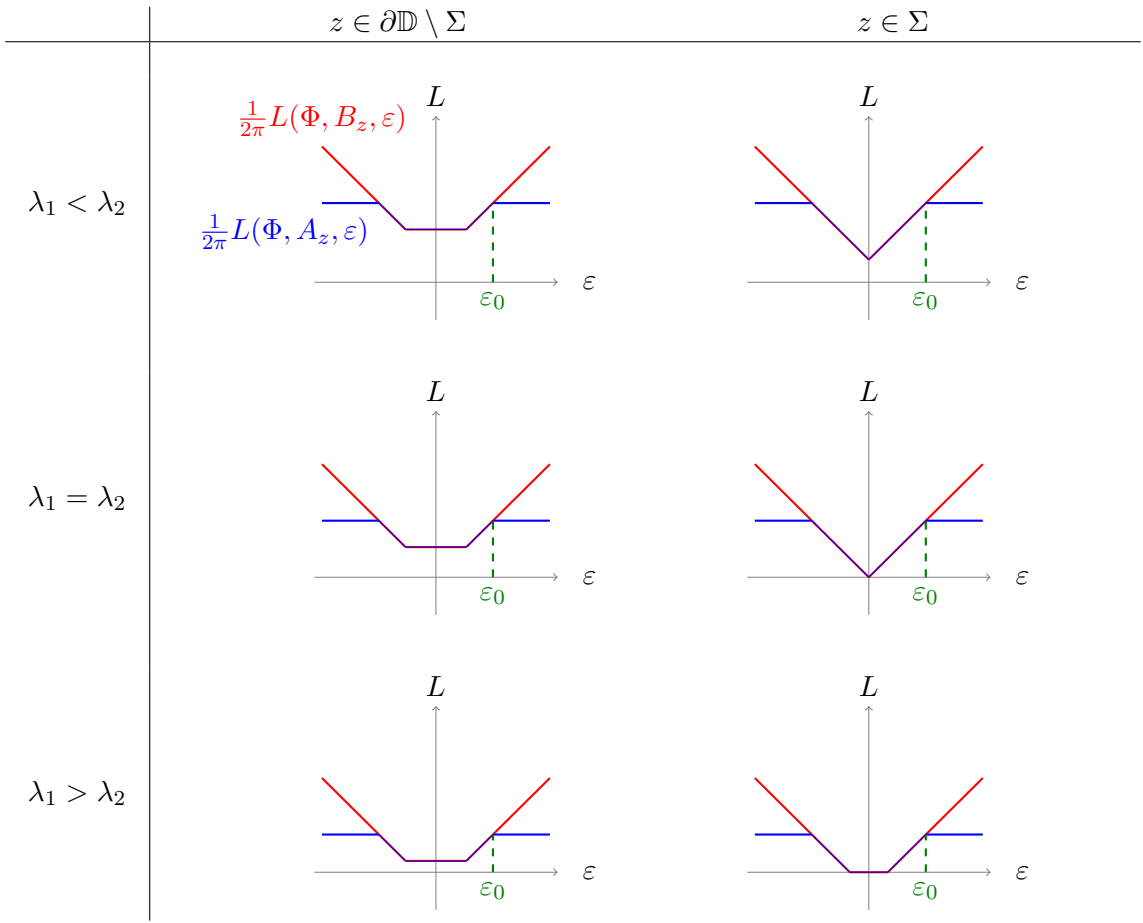
\begin{figure}
	\begin{center}
		\begin{tabular}{>{\centering\arraybackslash}m{.1\linewidth}|VV}
			& $z\in \partial\bbD\setminus \Sigma$ & $z\in\Sigma$ \\\hline
			\\
			$\lambda_1<\lambda_2$ & \begin{tikzpicture}[scale=1]
			\draw [help lines,->] (-1.6,0) -- (1.6,0);
			\draw [help lines,->] (0,-0.5) -- (0,2.2);
			\draw [thick,red,domain=-1.5:-0.4] plot (\x, {-1*\x+0.3});
			\draw [thick,red,domain=0.4:1.5] plot (\x, {1*\x+0.3});
			\draw [thick,red,domain=-0.4:0.4] plot (\x, {0.4+0*\x+0.3});
			\draw [thick,blue,domain=-1.5:-0.75] plot (\x, {0.75+0.3});
			\draw [blue,domain=-0.75:-0.4] plot (\x, {-\x+0.3});
			\draw [blue,domain=-0.4:0.4] plot (\x, {0.4+0.3});
			\draw [blue,domain=0.4:0.75] plot (\x, {\x+0.3});
			\draw [thick,blue,domain=0.75:1.5] plot (\x, {0.75+0.3});
			\draw [thick,dashed,green,domain=-0.3:0.75] plot ({0.75},{\x+0.3});
			\node [above] at (0,2.2) {$L$};
			\node [above] at (-1.5,1.8) {\red{$\frac1{2\pi} L(\Phi,B_z,\varepsilon)$}};
			\node [above] at (-2,0.3) {\blue{$\frac1{2\pi} L(\Phi,A_z,\varepsilon)$}};
			\node [right] at (1.8,0) {$\varepsilon$};
			\node [below] at (0.75,0) {\green{$\varepsilon_0$}};
			\end{tikzpicture} & 
			\begin{tikzpicture}[scale=1]
			\draw [help lines,->] (-1.6,0) -- (1.6,0);
			\draw [help lines,->] (0,-0.5) -- (0,2.2);
			\draw [thick,red,domain=-1.5:0] plot (\x, {-1*\x+0.3});
			\draw [thick,red,domain=0:1.5] plot (\x, {1*\x+0.3});
			\draw [thick,blue,domain=-1.5:-0.75] plot (\x, {0.75+0.3});
			\draw [thick,blue,domain=0.75:1.5] plot (\x, {0.75+0.3});
			\draw [blue,domain=-0.75:0] plot (\x, {-1*\x+0.3});
			\draw [blue,domain=0:0.75] plot (\x, {1*\x+0.3});
			\draw [thick,dashed,green,domain=0:1.05] plot ({0.75},\x);
			\node [above] at (0,2.2) {$L$};
			\node [below] at (0.75,0) {\green{$\varepsilon_0$}};
			\node [right] at (1.8,0) {$\varepsilon$};
			\end{tikzpicture}
			\\[2cm]
			$\lambda_1=\lambda_2$ & \begin{tikzpicture}[scale=1]
			\draw [help lines,->] (-1.6,0) -- (1.6,0);
			\draw [help lines,->] (0,-0.5) -- (0,2.2);
			\draw [thick,red,domain=-1.5:-0.4] plot (\x, {-1*\x});
			\draw [thick,red,domain=0.4:1.5] plot (\x, {1*\x});
			\draw [thick,red,domain=-0.4:0.4] plot (\x, {0.4+0*\x});
			\draw [thick,blue,domain=-1.5:-0.75] plot (\x, {0.75});
			\draw [blue,domain=-0.75:-0.4] plot (\x, {-\x});
			\draw [blue,domain=-0.4:0.4] plot (\x, {0.4});
			\draw [blue,domain=0.4:0.75] plot (\x, {\x});
			\draw [thick,blue,domain=0.75:1.5] plot (\x, {0.75});
			\draw [thick,dashed,green,domain=0:0.75] plot ({0.75},\x);
			\node [above] at (0,2.2) {$L$};
			\node [above] at (-2,0) {\hphantom{\blue{$\frac1{2\pi} L(\Phi,A_z,\varepsilon)$}}};
			\node [right] at (1.8,0) {$\varepsilon$};
			\node [below] at (0.75,0) {\green{$\varepsilon_0$}};
			\end{tikzpicture} & 
			\begin{tikzpicture}[scale=1]
			\draw [help lines,->] (-1.6,0) -- (1.6,0);
			\draw [help lines,->] (0,-0.5) -- (0,2.2);
			\draw [thick,red,domain=-1.5:0] plot (\x, {-1*\x});
			\draw [thick,red,domain=0:1.5] plot (\x, {1*\x});
			\draw [thick,blue,domain=-1.5:-0.75] plot (\x, {0.75});
			\draw [thick,blue,domain=0.75:1.5] plot (\x, {0.75});
			\draw [blue,domain=-0.75:0] plot (\x, {-1*\x});
			\draw [blue,domain=0:0.75] plot (\x, {1*\x});
			\draw [thick,dashed,green,domain=0:0.75] plot ({0.75},\x);
			\node [above] at (0,2.2) {$L$};
			\node [below] at (0.75,0) {\green{$\varepsilon_0$}};
			\node [right] at (1.8,0) {$\varepsilon$};
			\end{tikzpicture} \\[2cm]
			$\lambda_1>\lambda_2$ & 
			\begin{tikzpicture}[scale=1]
			\draw [help lines,->] (-1.6,0) -- (1.6,0);
			\draw [help lines,->] (0,-0.5) -- (0,2.2);
			\draw [thick,red,domain=-1.5:-0.4] plot (\x, {-1*\x-0.25});
			\draw [thick,red,domain=0.4:1.5] plot (\x, {1*\x-0.25});
			\draw [thick,red,domain=-0.4:0.4] plot (\x, {0.4+0*\x-0.25});
			\draw [thick,blue,domain=-1.5:-0.75] plot (\x, {0.75-0.25});
			\draw [blue,domain=-0.75:-0.4] plot (\x, {-\x-0.25});
			\draw [blue,domain=-0.4:0.4] plot (\x, {0.4-0.25});
			\draw [blue,domain=0.4:0.75] plot (\x, {\x-0.25});
			\draw [thick,blue,domain=0.75:1.5] plot (\x, {0.5});
			\draw [thick,dashed,green,domain=0:0.5] plot ({0.75},\x);
			\node [above] at (0,2.2) {$L$};
			\node [above] at (-2,0) {\hphantom{\blue{$\frac1{2\pi} L(\Phi,A_z,\varepsilon)$}}};
			\node [right] at (1.8,0) {$\varepsilon$};
			\node [below] at (0.75,0) {\green{$\varepsilon_0$}};
			\end{tikzpicture} & 
			\begin{tikzpicture}[scale=1]
			\draw [help lines,->] (-1.6,0) -- (1.6,0);
			\draw [help lines,->] (0,-0.5) -- (0,2.2);
			\draw [thick,red,domain=-1.5:-0.25] plot (\x, {-1*\x-0.25});
			\draw [thick,red,domain=0.25:1.5] plot (\x, {1*\x-0.25});
			\draw [thick,red,domain=-0.25:0.25] plot (\x, {0});
			\draw [thick,blue,domain=-1.5:-0.75] plot (\x, {0.75-0.25});
			\draw [blue,domain=-0.75:-0.25] plot (\x, {-\x-0.25});
			\draw [blue,domain=-0.25:0.25] plot (\x, {0});
			\draw [blue,domain=0.25:0.75] plot (\x, {\x-0.25});
			\draw [thick,blue,domain=0.75:1.5] plot (\x, {0.5});
			\draw [thick,dashed,green,domain=0:0.5] plot ({0.75},\x);
			\node [above] at (0,2.2) {$L$};
			\node [right] at (1.8,0) {$\varepsilon$};
			\node [below] at (0.75,0) {\green{$\varepsilon_0$}};
			\end{tikzpicture}
		\end{tabular}
	\end{center}
	\caption{Complexified Lyapunov exponents in the sub-critical, critical and super-critical regime.\label{f:complexifiedLE} }
\end{figure} 

As discussed in Section~\ref{sec:modelanddefs}, the remaining technical task is to show that $A_{\lambda_1,\lambda_2,z}/\sqrt{\det A_{\lambda_1,\lambda_2,z}}$ is conjugate to an $\SL(2,\bbR)$ cocycle, which is the content of Proposition~\ref{prop:conjtoSL2R}.

\begin{proof}[Proof of Proposition~\ref{prop:conjtoSL2R}]
Denoting
\[ X(\lambda) = \begin{bmatrix} \lambda & \lambda' \\ \lambda' & - \lambda  \end{bmatrix}, \]
one can easily check $X = X^* = X^{-1}$ and $\tr X = 0$. 
By a straightforward calculation, one may check that $\widetilde{A}_{\lambda_1,\lambda_2,z}(\theta) := A_{\lambda_1,\lambda_2,z}(\theta) / \sqrt{\det A_{\lambda_1,\lambda_2,z}(\theta)}$ belongs to  $\SU(1,1,X(\lambda_1))$, that is,
\begin{align*}
[\widetilde{A}_{\lambda_1,\lambda_2,z}(\theta)]^* X(\lambda_1)\widetilde{A}_{\lambda_1,\lambda_2,z}(\theta) = X(\lambda_1)
\end{align*}
for any $\theta \in \bbT$, $0<\lambda_1 \leq 1$, $0 \leq \lambda_2 < 1$, and $z \in \partial \bbD$.
Moreover, one can confirm that
\[ Y(\lambda)^* X(\lambda) Y(\lambda)
= \begin{bmatrix} 0 & i \\ -i & 0 \end{bmatrix} \] 
where $Y$ is as in \eqref{eq:selfdual:YlamDef}. In view of the discussion in \cite[Section~10.4]{Simon2005OPUC2}, this implies that $\widetilde{A}^\realified_{\lambda_1,\lambda_2,z}(\theta) \in \SL(2,\bbR)$.

The form of $A_{\lambda_1,\lambda_2,z}^\realified$ in \eqref{eq:realifiedcocycleexplicit} follows from direct calculations. As discussed previously, we may analytically choose the branch of the square root, so analyticity follows immediately. Alternatively, analyticity is obvious from \eqref{eq:realifiedcocycleexplicit}.
\end{proof}

At last, we may combine all of these tools and conclude with the proofs of the main statements regarding the cocycles.

\begin{proof}[Proof of Theorem~\ref{t:cocycleREG}]
Let $\lambda_2 <1$ be given. First, notice that the Lyapunov exponents of $A$, $B$, and $A^\realified$ all coincide in a strip, that is,
\begin{equation}
L(\Phi,B_z,\varepsilon) = L(\Phi,A_z,\varepsilon) = L(\Phi,A_z^\realified,\varepsilon), 
\quad |\varepsilon| \leq \varepsilon_0 = \frac{1}{2\pi} \mathrm{arcsinh}\sqrt{\lambda_2^{-2}-1}
\end{equation} which holds on account of \eqref{eq:AlambdaDet} and Lemma~\ref{lem:le:logIntegral}. Since $\lambda_2<1$, the strip has positive width. Thus, to conclude the desired statements about $(\Phi,A_{\lambda_1,\lambda_2,z}^\realified)$, it suffices to prove suitable statements about the shape of the graph of $L(\Phi,B_{\lambda_1,\lambda_2,z},\varepsilon)$ near $\varepsilon=0$.
From here, the proofs of the three statements are nearly identical, so we only write the details for the proof of part~\ref{t:cocycleSupercritical}, since it requires some additional work compared to the other two. 

In part ~\ref{t:cocycleSupercritical} $\lambda_1 < \lambda_2<1$, so that
\[ \log \left[ \frac{\lambda_2(1+\lambda_1')}{\lambda_1(1+\lambda_2')}\right] > 0. \]
Consider $z \in \partial \bbD$ and let $B = B_{\lambda_1,\lambda_2,z}$ be the associated cocycle map.
By Proposition~\ref{prop:cocycle:LBeAsymptotic}, $\omega(B,\varepsilon) = \mathrm{sgn}(\varepsilon)$ whenever $|\varepsilon|$ is very large.
By quantization of acceleration, convexity, and evenness of the function $\varepsilon \mapsto L(\Phi,B,\varepsilon)$, there are two possibilities.
\medskip

\noindent \textbf{Case 1: \boldmath $\omega(B,\varepsilon)=0$ for some $\varepsilon \in \bbR$.} In this case, $L(\Phi,B,\varepsilon)$ is positive and affine in a neighborhood of $\varepsilon = 0$, which implies that $(\Phi,B)$ enjoys a dominated splitting \cite{AJS2014JEMS}, and hence that $z \notin \Sigma$ by Theorem~\ref{thm:specviaDS}.
\medskip

\noindent \textbf{Case 2: \boldmath $\omega(B,\varepsilon) \neq 0$ for all $\varepsilon$.} By the previously discussed properties of $L$, we have
\begin{equation} \label{eq:cocycle:exactFormOfB}
L(\Phi,B,\varepsilon) = 2\pi|\varepsilon| + \log\left[ \frac{\lambda_2(1+\lambda_1')}{\lambda_1(1+\lambda_2')}\right], \quad \forall \varepsilon.\end{equation}
In view of \eqref{eq:cocycle:LAtoLB}, \eqref{eq:cocycle:exactFormOfB} yields
\[
L(\Phi,A,\varepsilon) = 2\pi|\varepsilon| + \log\left[ \frac{\lambda_2(1+\lambda_1')}{\lambda_1(1+\lambda_2')}\right], \quad  |\varepsilon| < \varepsilon_0.\]
In particular, this implies that $z$ belongs to $\Sigma$ and that $L(z) = L(\Phi,A,0) = \log\lambda_0$, as desired, where $\lambda_0$ is given by \eqref{eq:model:lambda0Def}.
\end{proof}

In thinking about the proof of Theorem~\ref{t:cocycleREG}, the reader may find it helpful to consult Figure~\ref{f:complexifiedLE}.

\begin{proof}[Proof of Theorem~\ref{t:cocycleLambda2=1}]
Part~\ref{t:cocyclelambda2=1lambda1=1} is proved in \cite[Section~5]{FOZ2017CMP}. Simply note that $-i B_{1,1,\Phi,\theta}(z) = N^z(\theta)$ in the notation of \cite{FOZ2017CMP}. Part~\ref{t:cocyclelambda2=1lambda1<1} follows directly from Proposition~\ref{prop:cocycle:LBeAsymptotic} using the same argument used to prove part~\ref{t:cocycleSupercritical} of Theorem~\ref{t:cocycleREG}.
\end{proof}

%% file: 05-aubryduality.tex

\section{Aubry--Andr\'{e} Duality} \label{sec:aubry} 

One crucial aspect of the model is that it enjoys a version of Aubry duality.
We will describe two formulations of duality, each of which expresses self-similarity of the operator families $\{W_{\lambda_1,\lambda_2,\Phi,\theta} : \theta \in \bbT\}$ and $\{W_{\lambda_1,\lambda_2,\Phi,\theta}^\aubrydual := W_{\lambda_2,\lambda_1,\Phi,\theta}^\top : \theta \in \bbT\}$, both of which are useful.
The first manifestation of duality shows how to relate solutions of the eigenvalue equation $W_{\lambda_1,\lambda_2,\Phi,\theta}\psi = z\psi$ to formal solutions of the dual equation $W_{\lambda_1,\lambda_2,\Phi,\theta}^\aubrydual \varphi = z\varphi$.
The second manifestation of duality expresses a unitary equivalence between the direct integrals of the families $W_{\lambda_1,\lambda_2,\Phi,\theta}$ and $W_{\lambda_1,\lambda_2,\Phi,\theta}^\aubrydual$ via the Fourier transform.

\subsection{Duality via Solutions}
For $f \in \ell^2(\bbZ)$, we denote its inverse Fourier transform by $\widecheck f$, which is defined for absolutely summable $f$ by
\[\widecheck{f}(x) = \sum_{n \in \bbZ}e^{2\pi i n x}f_n.\]
Given $\psi \in \ell^2(\bbZ) \otimes \bbC^2$, let us define (motivated by \eqref{eq:models:dualSolutionDef}):
\begin{align}\label{eq:phidef}
	\widecheck{\phi}^+ = \frac1{\sqrt2}(\widecheck\psi^+ + i\widecheck\psi^-), \quad \widecheck{\phi}^- = \frac1{\sqrt2}(i\widecheck\psi^+ + \widecheck\psi^-).
\end{align}

\begin{proof}[Proof of Theorem~\ref{t:aubryViaSolutions}]
From \eqref{eq:lemW+} and $W\psi = z\psi$, we have
	\begin{align*}
		z\psi_n^+ & = \lambda_1 \left(q_{n-1}^{11}\psi_{n-1}^+ + q_{n-1}^{12}\psi_{n-1}^-\right)
		-\lambda_1' \left(q_{n}^{21}\psi_{n}^+ + q_{n}^{22}\psi_{n}^-\right).
	\end{align*}
	Taking the inverse Fourier transform of both sides, shifting the indices of the terms containing $n-1$, and substituting the explicit form of $Q_n$ from \eqref{eq:moddefs:coinDef} yields
	\begin{align*}
		z\widecheck{\psi}^+(x) 
		& = \sum_{n\in\bbZ}e^{2\pi i (n+1) x}  \lambda_1 \left( (\lambda_2 \cos(2\pi (n \Phi+\theta)) + i\lambda_2') \psi_{n}^+ -\lambda_2 \sin(2\pi (n \Phi+\theta)) \psi_{n}^-\right) \\[1mm]
		& \qquad - \sum_{n \in \bbZ} e^{2\pi i n x} \lambda_1' \left(\lambda_2 \sin(2\pi (n \Phi+\theta))\psi_{n}^+ +  (\lambda_2 \cos(2\pi (n \Phi+\theta)) -i\lambda_2')\psi_{n}^-\right)
	\end{align*}
	for a.e.\ $x \in \bbT$.
	Using the exponential formulations of sine and cosine, and expressing in terms of the inverse Fourier transform this becomes 
	\begin{align*}
		z\widecheck{\psi}^+(x) 
		& = \frac{e^{2\pi i x}}{2}\left(e^{2\pi i\theta}\lambda_1\lambda_2 \widecheck\psi^+(x+\Phi) + e^{-2\pi i\theta}\lambda_1\lambda_2 \widecheck\psi^+(x-\Phi) + 2i \lambda_1\lambda_2' \widecheck{\psi}^+(x) \right) \\[1mm]
		& \qquad + \frac{e^{2\pi i x}}{2} \left( i e^{2\pi i\theta}\lambda_1\lambda_2 \widecheck\psi^-(x+\Phi) - i e^{-2\pi i\theta}\lambda_1\lambda_2 \widecheck\psi^-(x-\Phi) \right) \\[1mm]
		& \qquad + \frac{1}{2}\left( ie^{2\pi i\theta}\lambda_1'\lambda_2 \widecheck\psi^+(x+\Phi) - ie^{-2\pi i\theta}\lambda_1'\lambda_2 \widecheck\psi^+(x-\Phi) \right) \\[1mm]
		& \qquad + \frac{1}{2}\left( - e^{2\pi i\theta}\lambda_1'\lambda_2 \widecheck\psi^-(x+\Phi) -e^{-2\pi i\theta}\lambda_1'\lambda_2 \widecheck\psi^-(x-\Phi) + 2i\lambda_1'\lambda_2'\widecheck\psi^-(x)\right)
	\end{align*}
	for a.e.\ $x$.
Rearranging to collect like terms yields
	\begin{align*}
		z\widecheck{\psi}^+(x) 
		&= e^{2\pi i\theta}\left(\frac{e^{2\pi i x}}{2}\lambda_1\lambda_2 + \frac{1}{2}i\lambda_1'\lambda_2 \right)\left(\widecheck\psi^+(x+ \Phi)+i\widecheck\psi^-(x + \Phi)\right) \\[1mm]
		& \qquad + ie^{2\pi i x}\lambda_1\lambda_2' \widecheck\psi^+(x) + i\lambda_1'\lambda_2' \widecheck\psi^-(x)\\[1mm]
		& \qquad + e^{-2\pi i\theta}\left( \frac{e^{2\pi i x}}{2} \lambda_1\lambda_2 - \frac{1}{2} i\lambda_1'\lambda_2 \right)\left(\widecheck\psi^+(x-\Phi)-i\widecheck\psi^-(x-\Phi)\right)
	\end{align*}
	for a.e. $x$.

Applying the same steps to \eqref{eq:lemW-} yields
	\begin{align*}
		z\widecheck{\psi}^-(x) 
		& = e^{2\pi i\theta}\left(\frac{e^{-2\pi i x}}{2}\lambda_1\lambda_2+\frac12i\lambda_1' \lambda_2\right)\left(-i\widecheck{\psi}^+(x+\Phi)+\widecheck{\psi}^-(x+\Phi)\right)	\\[1mm]
		& \qquad + i\lambda_1' \lambda_2'\widecheck{\psi}^+(x)- ie^{-2\pi i x}\lambda_1\lambda_2'\widecheck{\psi}^-(x)	\\[1mm]
		& \qquad +e^{-2\pi i\theta}\left(\frac{e^{-2\pi i x}}{2}\lambda_1\lambda_2-\frac12i\lambda_1' \lambda_2\right)\left(i\widecheck{\psi}^+(x-\Phi)+\widecheck{\psi}^-(x-\Phi)\right)
	\end{align*}
for a.e.\ $x$. 
From these expressions for $\widecheck\psi^+$ and $\widecheck\psi^-$, we obtain for $a,b\in\bbC$
	\begin{align*}
		z(a\widecheck\psi^+(x)+b\widecheck\psi^-(x))	
		&=\frac{1}{2}e^{2\pi i\theta}i\lambda_1'\lambda_2(a-ib)\left(\widecheck\psi^+(x+ \Phi)+i\widecheck\psi^-(x+ \Phi)\right)		\\
		&\qquad+(-1)\frac{1}{2}e^{-2\pi i\theta}i\lambda_1'\lambda_2(-ia+b)\left(i\widecheck\psi^+(x- \Phi)+\widecheck\psi^-(x- \Phi)\right)		\\
		&\qquad+\frac12e^{2\pi i\theta}\lambda_1\lambda_2\left(ae^{2\pi i x}-ibe^{-2\pi i x}\right)\left(\widecheck\psi^+(x+ \Phi)+i\widecheck\psi^-(x+ \Phi)\right)	\\
		&\qquad+\frac12e^{-2\pi i\theta}\lambda_1\lambda_2\left(-iae^{2\pi i x}+be^{-2\pi i x}\right)\left(i\widecheck\psi^+(x- \Phi)+\widecheck\psi^-(x- \Phi)\right)	\\
		&\qquad+i\lambda_1\lambda_2'\left(ae^{2\pi ix}\widecheck\psi^+(x)-be^{-2\pi ix}\widecheck\psi^-(x)\right)+i\lambda_1'\lambda_2'\left(a\widecheck\psi^-(x)+b\widecheck\psi^+(x)\right)
	\end{align*}
	Writing $e^{\pm2\pi i x}=\cos(2\pi x)\pm i\sin(2\pi x)$ and sorting the terms, this amounts to
	\begin{align*}
		z(a\widecheck\psi^+(x)+b\widecheck\psi^-(x))
		&=\frac{1}{2}e^{2\pi i\theta}(\lambda_1\cos(2\pi x)+i\lambda_1')\lambda_2(a-ib)\left(\widecheck\psi^+(x+ \Phi)+i\widecheck\psi^-(x+ \Phi)\right)		\\
		&\qquad-\frac{1}{2}e^{-2\pi i\theta}(\lambda_1\cos(2\pi x)-i\lambda_1')\lambda_2(ia-b)\left(i\widecheck\psi^+(x- \Phi)+\widecheck\psi^-(x- \Phi)\right)		\\
		&\qquad+\frac12e^{2\pi i\theta}\lambda_1\lambda_2\sin(2\pi x)(ia-b)\left(\widecheck\psi^+(x+ \Phi)+i\widecheck\psi^-(x+ \Phi)\right)	\\
		&\qquad+\frac12e^{-2\pi i\theta}\lambda_1\lambda_2\sin(2\pi x)(a-ib)\left(i\widecheck\psi^+(x- \Phi)+\widecheck\psi^-(x- \Phi)\right)	\\
		&\qquad+i\lambda_1\lambda_2'\left(\cos(2\pi x)(a\widecheck\psi^+(x)-b\widecheck\psi^-(x))+i\sin(2\pi x)(a\widecheck\psi^+(x)+b\widecheck\psi^-(x))\right)\\
		&\qquad+i\lambda_1'\lambda_2'\left(a\widecheck\psi^-(x)+b\widecheck\psi^+(x)\right)
	\end{align*}
	
	Thus, choosing $(a,b)=(1,i)$ and $(a,b)=(i,1)$, respectively, and setting $\widecheck\phi^+$ and $\widecheck\phi^-$ as in \eqref{eq:phidef} we obtain
	\begin{align}\label{eq:phiplus}
			z\widecheck\phi^+(x)	&= e^{2\pi i\theta}(\lambda_1\cos(2\pi x)+i\lambda_1')\lambda_2\widecheck\phi^+(x+\Phi)+e^{-2\pi i\theta}\sin(2\pi x)\lambda_1\lambda_2\widecheck\phi^-(x-\Phi)	\\
			&\qquad +(\lambda_1\cos(2\pi x)+i\lambda_1')\lambda_2'\widecheck\phi^-(x)-\lambda_1\lambda_2'\sin(2\pi x)\widecheck\phi^+(x),	\nonumber\\
			\label{eq:phiminus}
			z\widecheck\phi^-(x)	&=-e^{2\pi i\theta}\sin(2\pi x)\lambda_1\lambda_2\widecheck\phi^+(x+\Phi)+e^{-2\pi i\theta}(\lambda_1\cos(2\pi x)-i\lambda_1')\lambda_2\widecheck\phi^-(x-\Phi)	\\
			&\qquad -(\lambda_1\cos(2\pi x)-i\lambda_1')\lambda_2'\widecheck\phi^+(x)-\lambda_1\lambda_2'\sin(2\pi x)\widecheck\phi^-(x).	\nonumber
	\end{align}

	To conclude the proof, we apply $W_{\lambda_2,\lambda_1,\Phi,\xi}^\top$ to $\varphi_n^s=\varphi_{n,s}^\xi=e^{2\pi in\theta}\widecheck\phi^s(n\Phi+\xi)$ and write $y=y(n,\xi)=n\Phi+\xi$. Thus from  Lemma \ref{lem:wT} we have in coordinates:
	\begin{align*}
		[W_{\lambda_2,\lambda_1,\Phi,\xi}^\top\varphi]_n^+	&=q_n^{11}(\lambda_2\varphi_{n+1}^++\lambda_2'\varphi_n^-)+q_n^{21}(-\lambda_2'\varphi_n^++\lambda_2\varphi_{n-1}^-)\\
		&=e^{2\pi i n\theta}\big[(\lambda_1\cos(2\pi y)+i\lambda_1')(e^{2\pi i \theta}\lambda_2\phi^+(y+\Phi)+\lambda_2'\widecheck\phi^-(y))\\
		&\qquad+\lambda_1\sin(2\pi y)(-\lambda_2'\widecheck\phi^+(y)+e^{-2\pi i \theta}\lambda_2\widecheck\phi^-(y-\Phi))\big]\\
		&=e^{2\pi i n\theta}z\widecheck\phi^+(y)	\\
		&=z\varphi_n^+,
	\end{align*}
	where we used \eqref{eq:phiplus}. An analogous calculation shows that $[W_{\lambda_2,\lambda_1,\Phi,\xi}^\top\varphi]_n^-=z\varphi_n^-$.
\end{proof}

\subsection{Duality via Direct Integrals}

Recalling $\scrH_1 = \ell^2(\bbZ)\otimes \bbC^2$, let us define
\[ \scrH_\oplus = \int_\bbT^\oplus \scrH_1 \, d\theta , \quad  \bm{W}_{\lambda_1,\lambda_2,\Phi} = \int_\bbT^\oplus W_{\lambda_1,\lambda_2,\Phi,\theta} \, d\theta. \]
We write the coordinates of $\Psi \in \scrH_\oplus$ as $\Psi_{n}^s(\theta)$ where $n \in \bbZ$, $s \in \{\pm\}$, and $\theta \in \bbT$.
For $\Psi \in \scrH_\oplus$, let $\widehat{\Psi}$ be its Fourier transform, given by
\[ \widehat{\Psi}_{n}^s(\theta) = \sum_{m \in \bbZ}\int_\bbT e^{-2\pi i m \theta}e^{-2\pi i nx} \Psi_{m}^s(x) \, dx, \quad n \in \bbZ, \ \theta \in \bbT, \ s \in \{\pm\}. \]
The Aubry dual operator is given by $\bm{\calA}  = \bm{\calA}_\Phi : \scrH_\oplus \to \scrH_\oplus$, where
\[ [\bm{\calA} \Psi]_{n}^s(\theta) = \widehat\Psi_{n}^s(n\Phi+\theta). \]
We also define $\bm{\calX} :\scrH_\oplus \to \scrH_\oplus$ by
\[ \bm{\calX} = \int_\bbT^\oplus \bigoplus_{n \in \bbZ} \frac{1}{\sqrt{2}} \begin{bmatrix}i & 1 \\ 1 & i  \end{bmatrix} \, dx. \]

\begin{theorem}[Aubry Duality for Operators] \label{t:aubry:abstract}
For all $\lambda_1,\lambda_2$, and irrational $\Phi$,
\begin{equation}
\bm{W}_{\lambda_1,\lambda_2,\Phi}\bm{\calA}_\Phi \bm{\calX} = \bm{\calA}_\Phi \bm{\calX} \bm{W}_{\lambda_2,\lambda_1,\Phi}^\top.
\end{equation}
\end{theorem}

\begin{proof}
Recall that every one-dimensional quantum walk can be related by an extended CMV matrix via the CGMV connection and that one can recover the $\mathcal{LM}$ factorization of the associated CMV matrix by inserting an infinite direct sum of $\sigma_1$'s, viz.: $W_\lambda = S_\lambda Q = (S_\lambda \bm{\sigma}_1 )(\bm{\sigma}_1 Q)$, where $\bm{\sigma_1}$ denotes an infinite direct sum of copies of $\sigma_1$. This motivates one to define operators $J=J_{\lambda}$, $U = U_{\lambda,\Phi,\theta}$, and $V = V_{\lambda,\Phi,\theta}$ on $\scrH_1 = \ell^2(\bbZ) \otimes \bbC^2$ by
\begin{align*}
[J_\lambda \varphi]_n^+ 
&= -\lambda'\varphi_n^+ + \lambda \varphi_{n-1}^- \\[1mm]
[J_\lambda \varphi]_n^- 
&= \lambda \varphi_{n+1}^+ + \lambda' \varphi_{n}^- \\[1mm]
[U \varphi]_n^+ 
&= (\lambda \sin(2\pi(n\Phi+\theta))) \varphi_n^+ + (\lambda \cos(2\pi(n\Phi+\theta)) - i\lambda') \varphi_{n}^- \\[1mm]
[U \varphi]_n^- 
&= (\lambda\cos(2\pi(n\Phi+\theta)) + i\lambda')\varphi_{n}^+ - \lambda\sin(2\pi(n\Phi+\theta))\varphi_{n}^- \\[1mm]
[V \varphi]_n^+ 
&= -\lambda' \varphi_n^+ + \lambda e^{-2\pi i (n\Phi+\theta)} \varphi_{n}^- \\[1mm]
[V \varphi]_n^- 
&=  \lambda e^{2\pi i (n\Phi+\theta)} \varphi_{n}^+ + \lambda' \varphi_{n}^-,
\end{align*}
and denote the corresponding direct integrals by
\[\bm{J}_{\lambda} = \int_\bbT^\oplus J_{\lambda} \,d\theta, \quad \bm{U}_{\lambda, \Phi} = \int_\bbT^\oplus U_{\lambda, \Phi,\theta} \,d\theta,
\quad \bm{V}_{\lambda, \Phi} = \int_\bbT^\oplus V_{\lambda, \Phi,\theta} \,d\theta.\]
One deduces 
\begin{equation} \label{eq:aubry:W=JU}
\bm{W}_{\lambda_1, \lambda_2, \Phi} = \bm{J}_{\lambda_1} \bm{U}_{\lambda_2,\Phi}
\end{equation} 
from the pointwise identity $W_{\lambda_1,\lambda_2,\Phi,\theta} = J_{\lambda_1} U_{\lambda_2,\Phi,\theta}$.
From the calculation
\begin{align*}
& \ \frac{1}{2} \begin{bmatrix} -i & 1 \\ 1 & -i \end{bmatrix}
\begin{bmatrix} \lambda \sin(2\pi(n\Phi+\theta)) 
&\lambda \cos(2\pi(n\Phi+\theta)) - i\lambda' \\
\lambda\cos(2\pi(n\Phi+\theta)) + i\lambda'
& - \lambda\sin(2\pi(n\Phi+\theta)) \end{bmatrix}
\begin{bmatrix} i & 1 \\ 1 & i \end{bmatrix} \\[1mm]
= & \begin{bmatrix} -\lambda'
& \lambda e^{-2\pi i (n\Phi+\theta)} \\
\lambda e^{2\pi i (n\Phi + \theta)} 
& \lambda' \end{bmatrix},
\end{align*}
we get 
\begin{equation} \label{eq:aubry:UtoV}
\bm{\calX}^*\bm{U}_{\lambda,\Phi} \bm{\calX} = \bm{V}_{\lambda_,\Phi}.
\end{equation}
Taking the transpose of both sides of \eqref{eq:aubry:UtoV} establishes 
\begin{equation} \label{eq:aubry:UTtoVT}
\bm{\calX}^*\bm{V}^\top_{\lambda,\Phi} \bm{\calX} = \bm{U}^\top_{\lambda,\Phi}.
\end{equation}
Next, we show 
\begin{equation} \label{eq:aubry:JtoVT} \bm{\calA}_\Phi^* \bm{J}_{\lambda}\bm{\calA}_\Phi = \bm{V}_{\lambda,\Phi}^\top.
\end{equation} 
Indeed, from the definitions,
\begin{align*}
[\bm{J}_\lambda\bm{\calA}_\Phi\Psi]_{n}^+(\theta)
& = -\lambda' [\bm\calA_\Phi\Psi]_{n}^+(\theta) + \lambda [\bm\calA_\Phi \Psi]_{n-1}^-(\theta) \\[1mm]
& = -\lambda' \sum_{m \in \bbZ} \int_\bbT     e^{-2\pi i m(n\Phi+\theta)}e^{-2\pi i nx} \Psi_{m}^+(x) \, dx \\[1mm]
& \qquad + \lambda \sum_{m \in \bbZ} \int_\bbT  e^{-2\pi i m((n-1)\Phi+\theta)} e^{-2\pi i (n-1)x} \Psi_{m}^-(x)  \, dx \\[1mm]
& = [\bm{\calA}_\Phi\bm{V}^\top_{\lambda,\Phi}]_{n}^{+}(\theta).
\end{align*}
A similar argument works for the spin-down component, proving \eqref{eq:aubry:JtoVT}
Finally, we show
\begin{equation} \label{eq:aubry:VtoJ} \bm{\calA}_\Phi^* \bm{V}_{\lambda,\Phi}\bm{\calA}_\Phi = \bm{J}_{\lambda}.
\end{equation}
Calculating $\bm{V}_{\lambda,\Phi}\bm{\calA}_\Phi$ gives
\begin{align*}
[\bm{V}_{\lambda,\Phi}\bm{\calA}_\Phi\Psi]_{n}^+
& = -\lambda' [\bm{\calA}_\Phi\Psi]_{n}^+ + \lambda e^{-2\pi i (n\Phi+\theta)} [\bm{\calA}_\Phi\Psi]_{n}^{-} \\[1mm]
& = \sum_{m\in\bbZ} \int_\bbT e^{-2\pi i m(n\Phi+\theta)}e^{-2\pi i n x}\left( -\lambda' \Psi_{m}^+(x) + \lambda e^{-2\pi i(n\Phi+\theta)}\Psi_{m}^-(x) \right) \, dx \\[1mm]
& = \sum_{m\in\bbZ} \int_\bbT e^{-2\pi i m(n\Phi+\theta)}e^{-2\pi i n x}\left( -\lambda' \Psi_{m}^+(x) + \lambda \Psi_{m-1}^-(x) \right) \, dx \\[1mm]
& = [\bm{\calA}_\Phi \bm{J}_\lambda \Psi]_{n}^{+}(\theta).
\end{align*}
The other component of \eqref{eq:aubry:VtoJ} is similar. 
Putting together \eqref{eq:aubry:W=JU}, \eqref{eq:aubry:UtoV}, \eqref{eq:aubry:UTtoVT}, and \eqref{eq:aubry:JtoVT}, \eqref{eq:aubry:VtoJ} (and using $\bm{\calA}_\Phi\bm{\calX} = \bm{\calX} \bm{\calA}_\Phi$), we get
\begin{align*}
 \bm{\calX}^* \bm{\calA}_\Phi^* \bm{W}_{\lambda_1,\lambda_2,\Phi}\bm{\calA}_\Phi \bm{\calX}
& =   \bm{\calX}^* \bm{\calA}_\Phi^*  \bm{J}_{\lambda_1} \bm{\calA}_\Phi \bm{\calX}  \bm{\calA}_\Phi^* \bm{\calX}^*  \bm{U}_{\lambda_2,\Phi} \bm{\calX} \bm{\calA}_\Phi \\[1mm]
& =   \bm{\calX}^*   \bm{V}_{\lambda_1,\Phi}^\top  \bm{\calX}  \bm{\calA}_\Phi^*  \bm{V}_{\lambda_2,\Phi}  \bm{\calA}_\Phi \\[1mm]
& = \bm{U}_{\lambda_1,\Phi}^\top \bm{J}_{\lambda_2} \\
& = \bm{W}_{\lambda_2,\lambda_1,\Phi}^\top,
\end{align*}
since $\bm{J}_\lambda^\top = \bm{J}_\lambda$. \end{proof}

\begin{proof}[Proof of Theorem~\ref{t:aubryOperator}]
Since $\bm{\calA}_\Phi$ and $\bm{\calX}$ are unitary, this is an immediate consequence of Theorem~\ref{t:aubry:abstract}.
\end{proof}

%% file: 06-contspec.tex

\section{Continuous Spectrum}
\label{sec:contspec}

Having studied the Lyapunov exponent and cocycle dynamics in detail, we now move towards the analysis of the spectrum and spectral type of $W_{\lambda_1,\lambda_2,\Phi,\theta}$.
We begin in the current section with results on continuous spectrum; equivalently, we prove several results that establish the absence of eigenvalues, which is equivalent to continuity of spectral measures.
We present three results, each of which covers a particular paramter region. We begin with the sharp Gordon argument showing that for suitable Liouville numbers (dependent on the coupling via the Lyapunov exponent), the spectrum is purely continuous.
Coupled with positivity of the Lyapunov exponent in the supercritical region, this immediately yields singular continuous spectrum.
Next, we show continuous spectrum in the self-dual region for all irrational $\Phi$ and all but countably many $\theta$.
Later on, we will show that the spectrum has zero Lebesgue measure in the self-dual region so this shall again yield purely singular continuous spectrum for each irrational frequency and all but countably many phases.
Finally, we show that this (i.e., purely continuous spectrum) holds uniformly in the frequency and phase in the subcritical region.

\subsection{Liouville Fields: Sharp Gordon Criterion}

We begin with Theorem~\ref{t:maintype}.\ref{t:maintype:gordon}. This is a combination of two facts: absence of absolutely continuous spectrum, which follows from positivity of the Lyapunov exponent and Kotani theory and absence of eigenvalues, which follows from Gordon-type arguments. By applying Gordon's lemma for CMV matrices \cite{F2017PAMS}, one can immediately see that the spectral type is purely continuous whenever $\Phi$ is Liouville.
In fact, by using the sharp Gordon criterion as in \cite{AJ2009Ann, AviYouZho2017DMJ, JitoKociIMRN, JitoLiu2017CPAM, JitoLiu2018Annals}, one can prove singular continuous spectrum for all $\Phi$ above a suitable arithmetic threshold dictated by the Lyapunov exponent.

\begin{definition}
Given $\Phi \in \bbT$ irrational, let $p_k/q_k$ denote the continued fraction convergents of $\Phi$ and
\begin{equation} \label{eq:betaPhidef}
\beta(\Phi)= \limsup_{k\to\infty} \frac{\log q_{k+1}}{q_k}.
\end{equation}
\end{definition}

\begin{theorem} \label{t:AYZ}
If $\lambda_1 < \lambda_2$ and
\begin{equation} \beta(\Phi) > \log\left[ \frac{\lambda_2(1+\lambda_1')}{\lambda_1(1+\lambda_2')} \right] \end{equation}
then $W_{\lambda_1,\lambda_2,\Phi,\theta}$ has purely singular continuous spectrum for every $\theta \in \bbT$.
\end{theorem}

\begin{proof}
Suppose $\lambda_1<\lambda_2$.
By Theorem~\ref{t:cocycleREG}.\ref{t:cocycleSupercritical}, the Lyapunov exponent is uniformly positive and given by
\begin{equation}
\label{eq:contspec:gordonLEexact} L_{\lambda_1,\lambda_2,\Phi}(z) = \log\left[\frac{\lambda_2(1+\lambda_1')}{\lambda_1(1+\lambda_2')}\right] \end{equation}
for $z \in \Sigma_{\lambda_1,\lambda_2,\Phi}$.
By Kotani theory, the a.c.\ spectrum of $W_{\lambda_1,\lambda_2,\Phi,\theta}$ is absent for a.e.\ $\theta$. More precisely, recalling the gauge equivalence of $W_{\lambda_1,\lambda_2,\Phi, \theta}$ and $\calE^\realified_{\lambda_1,\lambda_2,\Phi, \theta}$ as in Corollary~\ref{coro:UAMOtoCMV} and the equivalence between the Szeg\H{o} cocycle and the cocycle $A_z$ in Corollary~\ref{coro:UAMOtoSzego}, the absence of a.c.\ spectrum follows from \cite[Theorem~10.11.1]{Simon2005OPUC2}.
By minimality then, the absolutely continuous spectrum is empty for all $\theta \in \bbT$ by the Last--Simon theorem for CMV matrices; cf.\ \cite[Theorem~10.9.11]{Simon2005OPUC2}.

On the other hand, using the exact form of the Lyapunov exponent from \eqref{eq:contspec:gordonLEexact} and the sharp Gordon criterion, which was worked out for CMV matrices in Li--Damanik--Zhou (see \cite[Appendix~C]{LiDamZhou2021Preprint}), one sees that every spectral measure is continuous. More precisely, the desired absence of eigenvalues may be deduced from the argument of \cite{LiDamZhou2021Preprint} after taking two things into account. First, use the gauge equivalence of $W_{\lambda_1,\lambda_2,\Phi,\theta}$ and $\calE_{\lambda_1,\lambda_2,\Phi,\theta}^\realified$ as in Corollary~\ref{coro:UAMOtoCMV} to equate purely continuous spectrum for $W_{\lambda_1,\lambda_2,\Phi,\theta}$ with that of $\calE_{\lambda_1,\lambda_2,\Phi,\theta}^\realified$. Second, apply a nearly-verbatim repetition of the arguments of \cite{LiDamZhou2021Preprint} for $\calE_{\lambda_1,\lambda_2,\Phi,\theta}^\realified$ by passing to blocks of length two; that is, work with the two-step Szeg\H{o} cocycle as in \eqref{eq:szegococycledef}, which is genuinely quasiperiodic (and which belongs to $\SU(1,1)$ \cite[Section~10.4]{Simon2005OPUC2}), enabling the application of their techniques.
\end{proof}

\begin{proof}[Proof of Theorem~\ref{t:maintype}.\ref{t:maintype:gordon}]
Since the set of $\Phi$ with $\beta(\Phi) = \infty$ is known to be residual, this follows immediately from Theorem~\ref{t:AYZ}.\end{proof}

\subsection{Continuous Spectrum in the Critical Regime}

We now discuss the critical case. Later in the manuscript, we will show that when $\lambda_1 = \lambda_2 \in (0,1]$, the spectrum is a Cantor set of zero Lebesgue measure and hence cannot support any absolutely continuous spectrum. To classify the spectral type, it then remains to see whether there may be any eigenvalues. We will see that in the critical case $0<\lambda_1=\lambda_2 < 1$, the point spectrum is empty away from a countable set of phases, and the set of exceptional phases is an explicit subset defined by possible reflection symmetries of the coins:

\begin{definition}
If $\Phi \in \bbT$ is irrational, we say that $\phi \in \bbT$ is \emph{irrational with respect to} $\Phi$ if $2\phi+k\Phi \notin \bbZ$ every $k \in \bbZ$. 
Otherwise, we say $\phi$ is \emph{rational} with respect to $\Phi$.
\end{definition}

Our primary goal in the present section is to prove the following theorem. The proof follows the overall structure of \cite{avila2008point} (see also \cite{AviJitoMarx2017Invent}), but there are several complications arising from the more involved structure of the cocycle. Additionally, the reflection symmetry of the cocycle as in Proposition~\ref{prop:cocycRefl} is different from that in the self-adjoint case and required novel analysis to find.

\begin{theorem}\label{t:contspec:selfdual}
Consider $0 < \lambda  < 1$. 
For every irrational $\Phi$ and every $\theta$ for which $\theta-1/4$ is irrational with respect to $\Phi$, the operator $W_{\lambda,\lambda,\Phi,\theta}$ has empty point spectrum.
\end{theorem}

\begin{remark}
The theorem is clearly false if $\lambda =0$. When $\lambda = 1$, it holds in a more restricted form (namely, for all irrational $\Phi$ and a.e.\ $\theta$) \cite{FOZ2017CMP}.
Naturally, for a given irrational $\Phi$, there are only countably many $\theta$ that are rational with respect to $\Phi$, so this suffices to establish the continuity half of Theorem~\ref{t:maintype}.\ref{t:maintype:critical}.
\end{remark}

For the work ahead, we need to note a reflection symmetry of the cocycles.

\begin{prop} \label{prop:cocycRefl}
Let $z$, $\lambda_1$, and $\lambda_2$ be given, let $A_z = A_{\lambda_1,\lambda_2,z}$ denote the cocycle defined in \eqref{eq:AlambdaCocycleDef}, and let $A_z^\aubrydual$ denote the dual cocycle from \eqref{eq:models:Asharpdef}. For all $\theta \in \bbT$, one has
\begin{align}\label{eq:pi/2inverse}
R^{-1}A_z(\theta)^{-1}R 
& = -A_z\left(\frac{1}{2} - \theta \right) \\
\label{eq:pi/2inverseDual}
R^{-1}A_z^\aubrydual(\theta)^{-1}R 
& = -A_z^\aubrydual\left(\frac{1}{2} - \theta \right), 
\end{align}
where
\begin{equation}
R=
\begin{bmatrix} 0 & 1 \\ -1 & 0 \end{bmatrix}.
\end{equation}
\end{prop}

\begin{proof}
Note that the inverse of $A_{z}(\theta)$ is given by
\begin{equation}
A_z(\theta)^{-1}
= \frac{1}{\lambda_2 \costwopi(\theta) + i\lambda_2'}
\begin{bmatrix} 	\lambda_1z	&	\lambda_2 \sintwopi(\theta) +\lambda_1'z	\\	 \lambda_2 \sintwopi(\theta) + \lambda_1'z	
& \lambda_1^{-1}z^{-1}+2\lambda_1'\lambda_1^{-1}\lambda_2 \sintwopi(\theta) + z{\lambda_1'}^2\lambda_1^{-1}	\end{bmatrix}
\end{equation}
One has
	\begin{align}
	\nonumber
	R^{-1}A_z(\theta)^{-1}R
	& =  \begin{bmatrix} 0 & 1 \\ -1 & 0 \end{bmatrix} A_z(\theta)^{-1} \begin{bmatrix} 0 & -1 \\ 1 & 0 \end{bmatrix}\\
	\nonumber
	& = \frac{1}{\lambda_2 \costwopi(\theta) + i\lambda_2'}
	\begin{bmatrix} 	\lambda_1^{-1}z^{-1}+2\lambda_1'\lambda_1^{-1}\lambda_2 \sintwopi(\theta) + z{\lambda_1'}^2\lambda_1^{-1}	
	&	-(\lambda_2 \sintwopi(\theta) +\lambda_1'z)	\\	 
	-(\lambda_2 \sintwopi(\theta) + \lambda_1'z)	
	& \lambda_1z 	\end{bmatrix}
	\end{align}
	so, exploiting $\sintwopi(1/2-\theta)= \sintwopi(\theta)$ and $\costwopi(1/2-\theta) = -\costwopi(\theta)$ give us
	\begin{align*}
	& \ R^{-1}A_z(\theta)^{-1}R \\
	 = &\ \frac{-1}{\lambda_2 \costwopi(1/2-\theta) - i\lambda_2'}
	\begin{bmatrix} 	\lambda_1^{-1}z^{-1}+2\lambda_1'\lambda_1^{-1}\lambda_2 \sintwopi(1/2-\theta) + z{\lambda_1'}^2\lambda_1^{-1}	
	&	-(\lambda_2 \sintwopi(1/2-\theta) +\lambda_1'z)	\\	 
	-(\lambda_2 \sintwopi(1/2-\theta) + \lambda_1'z)	
	& \lambda_1z 	\end{bmatrix}	\\
	= & \ -A_z\left(\frac{1}{2}-\theta\right),
	\end{align*}
	proving \eqref{eq:pi/2inverse}. The identity \eqref{eq:pi/2inverseDual} follows immediately from \eqref{eq:pi/2inverse} and \eqref{eq:models:Asharpdef}. \end{proof}

Crucial in the proof of Theorem~\ref{t:contspec:selfdual} is a dynamical reformulation of duality, which we presently make precise. If $\psi$ is an $\ell^2$-eigenfunction of $W_{\lambda_1,\lambda_2,\Phi,\theta}$, then Theorem~\ref{t:aubryViaSolutions} implies that $\exp(2\pi i \theta n) \widecheck{\phi}(n\Phi+\xi)$ is a solution of $W^\aubrydual_{\lambda_1,\lambda_2,\Phi,\xi} \phi=z\phi$ for a.e.\ $\xi$ (with $\widecheck{\phi}$ defined by \eqref{eq:phidef}). Consequently, we have
\begin{align}
\begin{bmatrix}e^{2\pi i\theta}\widecheck\phi^+(x+\Phi)\\\widecheck\phi^-(x)\end{bmatrix}
\label{eq:contspec:dualTM1}
	= & \ A_{\lambda_1,\lambda_2,z}^\aubrydual(x)\begin{bmatrix}\widecheck\phi^+(x)\\e^{-2\pi i\theta}\widecheck\phi^-(x-\Phi)\end{bmatrix}
\end{align}
for a.e.\ $x$, with $A^\aubrydual$ defined in \eqref{eq:models:Asharpdef}. 
Now, substitute $\hat{x} = 1/2-x$ for $x$ and apply \eqref{eq:pi/2inverse} to get
\[ \begin{bmatrix}
e^{2\pi i\theta}\widecheck\phi^+(1/2-x+\Phi)\\
\widecheck\phi^-(1/2-x)\end{bmatrix}
= -R^{-1}A_{\lambda_1,\lambda_2,z}^\aubrydual(x)^{-1} R\begin{bmatrix}
	\widecheck\phi^+(1/2-x)\\
	e^{-2\pi i\theta}\widecheck\phi^-(1/2-x-\Phi)\end{bmatrix}. 
\]
Use \eqref{eq:AlambdaDet} and the definition of $R$ to see
\begin{equation} \label{eq:contspec:dualTM2b}
-A_{\lambda_1,\lambda_2,z}^\aubrydual(x)
\begin{bmatrix}
-\widecheck\phi^-(1/2-x) \\
 e^{2\pi i\theta}\widecheck\phi^+(1/2-(x-\Phi))
\end{bmatrix}
=  e^{-2\pi i\theta}\begin{bmatrix}
	-\widecheck\phi^-(1/2-(x+\Phi))\\
	e^{2\pi i\theta}\widecheck\phi^+(1/2-x)\end{bmatrix}
\end{equation}
Putting together \eqref{eq:contspec:dualTM1} and \eqref{eq:contspec:dualTM2b}, we get
\begin{equation}\label{eq:reducibility}
A_{\lambda_1,\lambda_2,z}^\aubrydual(x) M(x) = M(x+\Phi) \begin{bmatrix}e^{2\pi i \theta} \\ & -e^{-2\pi i \theta} \end{bmatrix},\end{equation}
where
\begin{equation}\label{eq:Mreducible}
M(x)= \begin{bmatrix}\widecheck\phi^+(x) & -\widecheck\phi^-(1/2-x)\\
e^{-2\pi i\theta}\widecheck\phi^-(x-\Phi)
& e^{2\pi i\theta}\widecheck\phi^+(1/2-(x-\Phi))\end{bmatrix}.
\end{equation}

\begin{proof}[Proof of Theorem~\ref{t:contspec:selfdual}]
Let $\lambda_1=\lambda_2 = \lambda \in (0,1)$, $\Phi$ irrational, and $\theta$ irrational with respect to $\Phi$ be given, and assume on the contrary that $W_{\lambda_1,\lambda_2,\Phi,\theta}$ has nonempty point spectrum. Let $\psi$ denote a normalized eigenfunction with corresponding eigenvalue $z$.
Define $\widecheck \phi$ by \eqref{eq:phidef} and consider $M$ as in \eqref{eq:Mreducible}. 
Since $|\det(A_{\lambda_1,\lambda_2,z}^\aubrydual(x))|=1$, \eqref{eq:reducibility} implies that $|\det(M(x))|=|\det(M(x+\Phi))|$. By ergodicity of $x\mapsto x+\Phi$ on $\bbT$, $|\det(M(x))|$ is a.e.\ constant with respect to $x$. 

Also $\|M(x)\|>0$ for a.e.\ $x$, since if not, $\|M(x)\|=0$ for a.e. $x$ by \eqref{eq:reducibility} and ergodicity of $x\mapsto x+\Phi$. However, this would imply $\widecheck \phi_\pm(x)$ vanishes for a.e.\ $x$, which contradicts nontriviality of the eigenfunction $\psi$ that generates $\widecheck{\phi}$.

Next we prove that $\det M(x)\neq 0$ for a.e. $x$. 
Assume for the sake of contradiction that $\det M = 0$ on a positive-measure set. Since $|\det M|$ is a.e.\ constant, one has $\det M = 0$ a.e, which means that the columns of $M(x)$ are linearly dependent for a.e.\ $x$. 
Equivalently, there is a function $\ell:\bbT \to \bbC$ with $\ell(x)\neq 0$ such that
\begin{equation}\label{eq:ell}
\begin{bmatrix}\widecheck\phi^+(x) \\
e^{-2\pi i\theta}\widecheck\phi^-(x-\Phi)
\end{bmatrix}=\ell(x)
\begin{bmatrix}-\widecheck\phi^-(1/2-x)\\
e^{2\pi i\theta}\widecheck\phi^+(1/2-(x-\Phi))\end{bmatrix}.
\end{equation}

Using \eqref{eq:ell}, \eqref{eq:reducibility}, \eqref{eq:ell}, and again \eqref{eq:reducibility} we have
\begin{align}
\ell(x+\Phi)
\begin{bmatrix}-\widecheck\phi^-(1/2-(x+\Phi))\\
e^{2\pi i\theta}\widecheck\phi^+(1/2-x)\end{bmatrix}
=&  \  \begin{bmatrix}\widecheck\phi^+(x+\Phi) \\
e^{-2\pi i\theta}\widecheck\phi^-(x)
\end{bmatrix} \nonumber\\
=&  \  e^{-2\pi i\theta}A_{\lambda_1,\lambda_2,z}^\aubrydual(x) \begin{bmatrix}\widecheck\phi^+(x) \\
e^{-2\pi i\theta}\widecheck\phi^-(x-\Phi)
\end{bmatrix} \nonumber\\
= &  \  e^{-2\pi i\theta}A_{\lambda_1,\lambda_2,z}^\aubrydual(x)\ell(x)
\begin{bmatrix}-\widecheck\phi^-(1/2-x)\\
e^{2\pi i\theta}\widecheck\phi^+(1/2-(x-\Phi))\end{bmatrix} \nonumber\\
=& \ e^{-2\pi i\theta}\ell(x)
(-e^{-2\pi i\theta})
\begin{bmatrix}-\widecheck\phi^-(1/2-(x+\Phi))\\
e^{2\pi i\theta}\widecheck\phi^+(1/2-x)\end{bmatrix} \nonumber\\
=& \ -e^{-4\pi i\theta}\ell(x)
\begin{bmatrix}-\widecheck\phi^-(1/2-(x+\Phi))\\
e^{2\pi i\theta}\widecheck\phi^+(1/2-x)\end{bmatrix},
\end{align}
which implies $-e^{-4\pi i\theta}\ell(x)=\ell(x+\Phi)$ for a.e.\ $x$. 
Fourier expanding $\ell(x)=\sum \hat\ell_k e^{2\pi ikx}$ this gives
\begin{equation} \label{eq:contspec:ellFourierResult}
-e^{-4\pi i\theta}\sum \hat\ell_k e^{2\pi ikx}
=\sum \hat\ell_k e^{2\pi ik \Phi}e^{2\pi ikx}.
\end{equation}
Since $\theta-1/4$ is irrational with respect to $\Phi$, we have $-e^{-4\pi i \theta} \neq e^{2\pi i k \Phi}$ for all $k \in \bbZ$. Combined with \eqref{eq:contspec:ellFourierResult}, this implies that all Fourier coefficients of $\ell$ must vanish, forcing $\ell \equiv 0$, a contradiction.

Let us now consider for $n\geq 1$ the $n$-step transfer matrix given by 
\begin{equation}
A^{\aubrydual n}(x)=A_{\lambda_1,\lambda_2,z}^\aubrydual(x+(n-1)\Phi)A_{\lambda_1,\lambda_2,z}^\aubrydual(x+(n-2)\Phi)\ldots A_{\lambda_1,\lambda_2,z}^\aubrydual(x+\Phi) A_{\lambda_1,\lambda_2,z}^\aubrydual(x), 
\end{equation}
where we choose notation consistent with \eqref{eq:cociterates}.
Then, by \eqref{eq:reducibility}
\begin{equation}\label{eq:AAubryDualn}
A^{\aubrydual n}(x)=M(x+n\Phi)
		\begin{bmatrix}
			e^{2\pi i\theta} &0\\
			0& -e^{-2\pi i\theta} 
		\end{bmatrix}^n  M(x)^{-1}.
\end{equation}

We define 
\begin{align}\Psi^{(n)}(x)
& = \tr A^{\aubrydual n}(x) - \tr \left( 
\begin{bmatrix}
e^{2\pi i\theta} &0\\
0& -e^{-2\pi i\theta} 
\end{bmatrix}^n \right)\label{eq:Psi^n} \\
\label{eq:Psi^nb}
& = \tr A^{\aubrydual n}(x) -(e^{2\pi i n \theta} + (-1)^n e^{-2\pi i n \theta}),
\end{align}
and note that
\begin{align}
\Psi^{(n)}(x) 
= \tr\left(\Big( M(x+n\Phi)-M(x) \Big)
\begin{bmatrix}
e^{2\pi i\theta} &0\\
0& -e^{-2\pi i\theta} 
\end{bmatrix}^n  M(x)^{-1}\right).
\end{align}
Since we have the inequality $|\tr(\mathfrak A)| \leq 2\|\mathfrak A\|$ and the identity $\|\mathfrak A^{-1}\| = \|\mathfrak A\|/|\det(\mathfrak A)|$, we have
\begin{equation}
|\Psi^{(n)}(x)|\leq \frac{2}{|\det M|}\Vert M(x+n\Phi)-M(x)\Vert \Vert M(x)\Vert.
\end{equation}
By the Cauchy-Schwarz inequality,
\begin{equation}
\Vert\Psi^{(n)}(x)\Vert_{L^1}\leq \frac{2}{|\det M|}\Vert M(x+n\Phi)-M(x)\Vert_{L^2} \Vert M(x)\Vert_{L^2}.
\end{equation}
Note that $\Vert M(x)\Vert _{L^2}$ is finite by unitarity of the inverse Fourier transform. This also implies that $\Vert\Psi^{(n)}\Vert_{L^1}$  can be made arbitrarily small by making $\mathrm{dist}(n\Phi,\bbZ)$ sufficiently small. We then have
\begin{equation}\label{eq:L1=0}
\liminf_{n\to\infty}\Vert\Psi^{(n)}(x)\Vert_{L^1}=0
\end{equation}
On the other hand, by \eqref{eq:AlambdaCocycleDef},  \eqref{eq:models:Asharpdef}, and \eqref{eq:Psi^n} we notice that $\Psi^{(n)}$ is a trigonometric polynomial divided by  
\begin{equation}\label{eq:Qn}
\mathcal Q^{(n)}_1
= \mathcal Q^{(n)}_1(x)
:=\prod_{k=0}^{n-1}q_{11}(x+k\Phi)
=\prod_{k=0}^{n-1} (\lambda\cos(2\pi(x+k\Phi) + i\lambda')  .
\end{equation} 
By \cite[Theorem~2.3]{AviJitoMarx2017Invent} we can control this quantity with the help of Lemma~\ref{lem:le:logIntegral} with $\varepsilon = 0$. Thus we can modify \eqref{eq:L1=0} to be
\begin{equation}\label{eq:L1=0.Q}
\liminf_{n\to\infty}\left\Vert  \left(\frac{1+\lambda'}{2}\right)^{-n}\mathcal Q^{(n)}_1(x) \Psi^{(n)}(x)\right \Vert_{L^1}=0
\end{equation}

 We observe from \eqref{eq:Psi^nb} that 
\begin{equation}\label{eq:QPsi}
\mathcal Q^{(n)}_1 \Psi^{(n)}= \mathcal Q^{(n)}_1(x) \tr A^{\aubrydual n}(x)-\mathcal Q^{(n)}_1(x) (e^{2\pi i n \theta} + (-1)^n e^{-2\pi i n \theta}).
\end{equation}
It is clear that the second term on the right hand side of the above equation is a constant multiple of the trigonometric polynomial, $\mathcal{Q}_1^{(n)}(x)$, (recall that $\theta$ is currently fixed).
By \eqref{eq:AlambdaCocycleDef} and \eqref{eq:models:Asharpdef}, it is also true that $\mathcal Q^{(n)}_1 \tr A^{\aubrydual n}(x)$ is a trigonometric polynomial. We will estimate the leading coefficients of these  trigonometric polynomials, that is, the coefficients of the $e^{2\pi i nx}$ and $e^{-2\pi i nx}$ terms. One sees immediately that the leading coefficients of $\mathcal{Q}_1^{(n)}$ have absolute value $(\lambda/2)^n$.

Next, we examine the leading coefficient of $\mathcal Q^{(n)}_1(x) \tr A^{\aubrydual n}(x)$ term of \eqref{eq:QPsi}. Recall by \eqref{eq:Qn} that $\mathcal Q^{(n)}_1(x)$ is a product of $q_j^{11}$ terms, and by \eqref{eq:AAubryDualn} $A^{\aubrydual n}(x)$ is a product of $A^\aubrydual(x+j\Phi)$ terms. Thus to understand $\mathcal Q^{(n)}_1(x) \tr A^{\aubrydual n}(x)$ we multiply out $q_{j}^{11}A^\aubrydual(x+j\Phi)$. Since we are trying to express $\mathcal Q^{(n)}_1(x) \tr A^{\aubrydual n}(x)$ as a trigonometric polynomial, we are only concerned with terms that are multiples of $q^{21}_{j}$, since only these terms produce trigonometric expressions. So separating out the terms that are multiples of $q^{21}_{j}$, we have
\begin{equation}\label{eq:lambdaMatrix}
q_{j}^{11}A^\aubrydual(x+j\Phi)=
\begin{bmatrix}
2\lambda' \lambda^{-1}  &-1\\
-1& 0 
\end{bmatrix}q^{21}_{j} + C
= \underbrace{\begin{bmatrix}
\lambda'   &-\lambda/2\\
-\lambda/2 & 0 
\end{bmatrix}}_{=:\mathfrak{S}(\lambda)}2 \sin(2\pi(j\Phi+x)) + C,
\end{equation}
where $C$ is a matrix that does not depend on $x$. We want to compute the leading coefficients of  $\mathcal Q^{(n)}_1(x) \tr A^{\aubrydual n}(x)$ expressed as a trigonometric polynomial. Note that this is a product of the $n$ expressions of the form \eqref{eq:lambdaMatrix} with $j=0,\ldots, n-1$. The leading terms of the trigonometric polynomial will be multiples of $e^{2n\pi ix }$ and $e^{-2n\pi ix }$. The only way to obtain terms of the form $e^{2n\pi ix }$ and $e^{-2n\pi ix }$ from $\mathcal Q^{(n)}_1(x) \tr A^{\aubrydual n}(x)$ is by multiplying $n$ sines in \eqref{eq:lambdaMatrix}, from $j=0,1,\ldots, n-1$. Thus we can see that these coefficients of these terms are obtained from calculating $\mathfrak{S}(\lambda)^n$.
One can check that $\mathfrak{S}(\lambda)$ has eigenvalues $\mu_\pm = \mu_\pm(\lambda) = (\lambda'\pm 1)/2$ (compare \eqref{eq:sprBlambda1}). 
Diagonalizing $\mathfrak{S}$ explicitly, we have the following for $n \geq 0$:
\[ \mathfrak{S}^n = \frac{1}{2} \begin{bmatrix} (1+\lambda')\mu_+^n +(1-\lambda')\mu_-^n 
& -\lambda \mu_+^n + \lambda \mu_-^n \\
-\lambda \mu_+^n + \lambda \mu_-^n 
& (1-\lambda')\mu_+^n + (1+\lambda')\mu_-^n \end{bmatrix} \]
Thus, the leading coefficient of $\mathcal{Q}_1^{(n)}(x) \tr A^{\aubrydual n}(x)$ has absolute value $\mu_+^n + \mu_-^n$.
But then, writing $c_n$ for the leading coefficient of $\mu_+^{-n}\mathcal{Q}_1^{(n)} \tr A^{\aubrydual n}$, this all means that
\begin{equation} \label{eq:selfdualcontspec:contra}
|c_n| \geq \frac{\mu_+^n + \mu_-^n}{\mu_+^n} - \left(\frac{\lambda}{2\mu_+} \right)^n
= 1 + \left(\frac{\lambda'-1}{\lambda' + 1} \right)^n - \left(\frac{\lambda}{\lambda'+1} \right)^n
\geq 1-o(1)
\end{equation}
as $n \to \infty$. However,
\[|c_n| = \left| \int_0^1 e^{-2\pi i n x} \frac{\mathcal Q^{(n)}_1}{\left(\frac{1+\lambda'}{2}\right)^n}\Psi^{(n)}(x) \, dx \right| \leq \left\|\frac{\mathcal Q^{(n)}_1}{\left(\frac{1+\lambda'}{2}\right)^n}\Psi^{(n)}(x) \right\|_{L^1},\]
so \eqref{eq:selfdualcontspec:contra} contradicts \eqref{eq:L1=0.Q}, which concludes the proof.
\end{proof}

\subsection{The Subcritical Regime}

The duality in Theorem \ref{t:aubryViaSolutions} together with the positivity of the Lyapunov exponent of the dual model in Theorem \ref{thm:posle} allows us to exclude point spectrum in the subcritical region for \emph{all} phases $\theta$ and \emph{all} irrational fields $\Phi$ by Delyon's argument \cite{Delyon1987JPA}.
\begin{prop}\label{prop:no_ac}
	Let $\Phi$ be irrational and $\lambda_1>\lambda_2$. 
	For every $\theta\in\mathbb R$, the point spectrum of $W_{\lambda_1,\lambda_2,\Phi,\theta}$ is empty.
\end{prop}

\begin{proof}
This follows from an argument of Delyon \cite{Delyon1987JPA}.
Arguing by contradiction, assume that for some $\lambda_1>\lambda_2$ the operator $W_{\lambda_1,\lambda_2,\Phi,\theta}$ has an eigenvalue $z$ with corresponding eigenvector $\psi \in \scrH_1$ normalized by $\|\psi\|=1$. 
Applying Theorem~\ref{t:aubryViaSolutions}, it follows that  $\varphi = \varphi^\xi$ given by
	\begin{equation*}
		\begin{bmatrix}\varphi_{n}^+\\\varphi_{n}^-\end{bmatrix}=e^{2\pi in\theta}\begin{bmatrix}\widecheck{\phi}^+(n\Phi+\xi)\\ \widecheck{\phi}^-(n\Phi+\xi)\end{bmatrix},
	\end{equation*}
	with $\widecheck{\phi}^\pm$ as in \eqref{eq:phidef} is a formal eigenfunction of $W_{\lambda_1,\lambda_2, \Phi, \xi}^\aubrydual$ for a.e.\ $\xi$. 
	Note that the inverse Fourier transform of $\psi$ is only defined almost everywhere on $\bbT$, so $\varphi^\xi$ itself is only defined for a.e.\ $\xi \in \bbT$.
	
Moreover, by unitarity of the inverse Fourier transform, one also has $\|\widecheck{\phi}\|_{L^2(\bbT)\otimes\bbC^2}=1$. 
Consequently, for every $\epsilon>0$ we have
	\begin{equation*}
		\int_\bbT d\xi \sum_{n\in\bbZ,s=\pm1} \left| \widecheck{\phi}^s(n\Phi+\xi) \right|^2/|n|^{1+\epsilon}=\sum_n|n|^{-(1+\epsilon)}<\infty.
	\end{equation*}
	Consequently, for a.e.\ $\xi$ and any $\epsilon>0$, there exists $C_\epsilon(\xi) > 0$ such that
	\begin{equation} \label{eq:contspec:spuriousLinBd}
	\left|\widecheck{\phi}^s(n\Phi+\xi) \right| < C_\epsilon(\xi)|n|^{(1+\epsilon)/2}.
	\end{equation}
	However, since $\lambda_1 > \lambda_2$, this contradicts positivity of the Lyapunov exponent associated with the dual model (cf.\ Theorem~\ref{thm:posle}).
This is a standard argument, but let us provide the details for the reader's convenience.
First, for any fixed $z \in \partial \bbD$, the set of $\theta \in \bbT$ for which $z$ is an eigenvalue of $W_{\lambda_1,\lambda_2, \Phi, \theta}$ is a Lebesgue null set (by Proposition~\ref{prop:cocycle:pastur}).
Consequently, by the multiplicative ergodic theorem, one must have that $\widecheck{\phi}^s(n\Phi+\xi)$ grows exponentially on at least one half-line for a positive-measure set of $\xi$, in contradiction with \eqref{eq:contspec:spuriousLinBd}.\end{proof}

%% file: 07-zeromeas.tex

\section{Zero-Measure Cantor Spectrum} \label{sec:zeromeas}

We now discuss the proof of zero-measure Cantor spectrum in the critical case $\lambda_1 = \lambda_2 = \lambda \in (0,1]$. 
As noted before, $\Sigma_{1,1,\Phi}$ was shown to be a zero-measure Cantor set in \cite{FOZ2017CMP}. On the other hand, it is trivial to see that $\Sigma_{0,0,\Phi} = \{\pm i\}$, so the assumption $\lambda>0$ is necessary for the Cantor spectrum result (though obviously not for the zero-measure result). In view of these comments, we focus on $0<\lambda < 1$ in the present section.

The proof of Theorem~\ref{t:maintype}.\ref{t:maintype:criticalCantor} will follow from the general theory of analytic $\SL(2,\bbR)$ cocycles as developed in \cite{AviFayKri2011GAFA, AviKri2015}.

\begin{lemma}\label{l.homotopicconstant}
For every $\lambda_1 \in (0,1]$, $\lambda_2 \in [0,1]$ and $z \in \partial \bbD$, $A_{\lambda_1,\lambda_2,z}^\realified$ is homotopic to a constant.
\end{lemma}

\begin{proof}
This is a straightforward calculation.
\end{proof}

In order to apply the well-developed machinery of real-analytic cocycles, it is helpful to note that the cocycle is monotonic in the argument of the spectral parameter. This is a straightforward (albeit cumbersome) calculation. For the reader's convenience, we supply the details.

\begin{lemma} \label{lem:AmonotonicInZ}
For $\lambda_1>0$ and $\lambda_2<1$, $t \mapsto A_{\lambda_1,\lambda_2,e^{it}}^\realified$ is monotonic in $t$ in the following sense: for any $\theta \in \bbT$ and any $v \in \bbR^2\setminus \{0\}$ {\rm(}any determination of{\rm)} the argument of $A_{\lambda_1,\lambda_2,e^{it}}^\realified(\theta) v$ has positive derivative with respect to $t$.
\end{lemma}

We will need the following elementary calculation.
\begin{lemma} \label{lem:cos^2sin^2trigbound}
Suppose $a,b,c \in \bbR$. If $a, b > 0$ and $ab-c^2>0$, then 
\[a \cos^2(\vartheta)+b\sin^2(\vartheta)+2c\sin\vartheta\cos\vartheta > 0\]
for all $\vartheta \in [0,2\pi]$. 
\end{lemma}

\begin{proof}
First, observe that
\begin{align}\nonumber
a \cos^2(\vartheta)+b\sin^2(\vartheta)+2c\sin\vartheta\cos\vartheta 
&\geq a\cos^2(\vartheta)+b\sin^2(\vartheta)-|2c\sin\vartheta\cos\vartheta|\nonumber\\
\label{eq:0meas:triglemma1}
& \geq a\cos^2(\vartheta)+b\sin^2(\vartheta)-|2\sqrt{ab}\sin\vartheta\cos\vartheta|\\
\nonumber
&\geq |\sqrt a\cos(\vartheta))-\sqrt b\sin(\vartheta)|^2 \\
\label{eq:0meas:triglemma2}
& \geq 0.
\end{align}
Since \eqref{eq:0meas:triglemma1} is strict whenever $\cos\vartheta\sin\vartheta\neq 0$ and \eqref{eq:0meas:triglemma2} is strict whenever $\cos\vartheta\sin\vartheta = 0$, the lemma is proved.
\end{proof}

\begin{proof}[Proof of Lemma~\ref{lem:AmonotonicInZ}]
Let $\lambda_1$, $\lambda_2$, and $\theta$ be given. Of course, positive scalar multiples do not affect the argument of a vector, so we may consider
\begin{align*}
\mathfrak{N}_{\lambda_1,\lambda_2,z}^\realified(\theta)
& =\lambda_1 \sqrt{{\lambda_2'}^2 + \lambda_2^2\costwopi^2(\theta)}  A_{\lambda_1,\lambda_2,z}^\realified(\theta) \\
&= 
\begin{bmatrix} 
\Re \, z + \lambda_1'\lambda_2 \sintwopi(\theta) 
& \lambda_1 \Im \, z - \lambda_1'\Re \, z - \lambda_2 \sintwopi(\theta)  \\
 -\lambda_1 \Im \, z - \lambda_1'\Re \, z - \lambda_2 \sintwopi(\theta) 
 & \Re \, z + \lambda_1'\lambda_2 \sintwopi(\theta) 
\end{bmatrix}
\end{align*}
 Write $z = e^{it}$, consider $v_u = [\cos u,\sin u]^\top$, and define $y_1,y_2$ by
\[
\begin{bmatrix}
y_2(t) \\ y_1(t)
\end{bmatrix}
= \mathfrak{N}^\realified_{\lambda_1,\lambda_2,e^{it}}(\theta)v_u.
\]
Since $\tan(\arg(y_2,y_1))= y_2/y_1$, computing the derivative of the argument of $A^\realified_{\lambda_1,\lambda_2,z}(\theta)v_u$ with respect to $t$ (and denoting the derivative with respect to $t$ by a dot) gives
\begin{align}
\label{eq:0meas:monotonicargdot1}
\frac{d}{dt}\left( \arg \begin{bmatrix}
y_2 \\ y_1
\end{bmatrix} \right)
= \frac{1}{1+(y_2/y_1)^2} \frac{\dot{y}_2y_1-\dot{y}_1 y_2}{y_1^2}.
\end{align}
Thus, we are left with considering $\dot{y}_2y_1-\dot{y}_1 y_2$. We have
\begin{align}
\nonumber
\dot{y}_2y_1-\dot{y}_1 y_2
 & = [-\Im \, z \cos u + ( \lambda_1 \Re \, z + \lambda_1'\Im \, z)\sin u] \\
 \nonumber
 & \qquad\qquad \times [( -\lambda_1 \Im \, z - \lambda_1'\Re \, z - \lambda_2 \sintwopi(\theta) ) \cos u + ( \Re \, z + \lambda_1'\lambda_2 \sintwopi(\theta))\sin u ] \\
 \nonumber
 & \qquad\qquad - [(-\lambda_1\Re z + \lambda_1' \Im z)\cos u -\Im z \sin u] \\ 
 \nonumber
 &\qquad\qquad \times [(\Re \, z + \lambda_1'\lambda_2 \sintwopi(\theta) )\cos u + ( \lambda_1 \Im \, z - \lambda_1'\Re \, z - \lambda_2 \sintwopi(\theta))\sin u] \\ 
 \nonumber
 & = \lambda_1(1 + \lambda_2 \sintwopi(\theta)(\lambda_1' \Re z + \lambda_1 \Im z)) \cos^2u \\
 \nonumber
 & \qquad\qquad + \lambda_1(1 - \lambda_2 \sintwopi(\theta)(-\lambda_1' \Re z + \lambda_1 \Im z)) \sin^2 u \\
 \nonumber
 & \qquad\qquad  - 2\lambda_1(\lambda_1' + \lambda_2\sintwopi(\theta)\Re z) \sin u \cos u \\
 \label{eq:0meas:monotonicargdot2}
& =a \cos^2u + b\sin^2u + 2c \cos u \sin u,
\end{align}
where
\begin{align}
\label{eq:0meas:monotonica}
a & = \lambda_1(1 + \lambda_2 \sintwopi(\theta)(\lambda_1' \Re z + \lambda_1 \Im z)) \\
\label{eq:0meas:monotonicb}
b & = \lambda_1(1 - \lambda_2 \sintwopi(\theta)(-\lambda_1' \Re z + \lambda_1 \Im z)) \\
\label{eq:0meas:monotonicc}
c & = -\lambda_1(\lambda_1' + \lambda_2\sintwopi(\theta)\Re z).
\end{align}
Note that the assumptions $\lambda_1>0$ and $\lambda_2<1$ imply $a,b>0$. After some algebra, one gets
\begin{align}
\nonumber
ab-c^2 
& = \lambda_1^2(1+2\lambda_2 \sintwopi(\theta)\lambda_1'\Re\, z + \lambda_2^2\sintwopi(\theta)^2({\lambda_1'}^2[\Re \, z]^2 - \lambda_1^2 [\Im \, z]^2)) \\
\nonumber
& \qquad \qquad - \lambda_1^2(\lambda_1' + \lambda_2\sintwopi(\theta)\Re\, z)^2 \\
\nonumber
& = \lambda_1^2 \left( (1+2\lambda_2 \sintwopi(\theta)\lambda_1'\Re\, z + \lambda_2^2\sintwopi(\theta)^2([\Re \, z]^2 - \lambda_1^2)) - (\lambda_1' + \lambda_2\sintwopi(\theta)\Re\, z)^2 \right) \\
\nonumber
& = \lambda_1^2 \left(  -\lambda_1^2 \lambda_2^2\sintwopi(\theta)^2 +\lambda_1^2 \right) \\
\label{eq:0meas:monotonicdisc}
& = \lambda_1^4(1-\lambda_2^2\sintwopi(\theta)^2) > 0,\end{align}
which is strictly positive by the assumption $\lambda_2<1$. Combining \eqref{eq:0meas:monotonicargdot1}, \eqref{eq:0meas:monotonicargdot2}, \eqref{eq:0meas:monotonica}, \eqref{eq:0meas:monotonicb}, \eqref{eq:0meas:monotonicc}, and \eqref{eq:0meas:monotonicdisc} with Lemma~\ref{lem:cos^2sin^2trigbound} we get
\begin{equation}
\frac{d}{dt}\left( \arg \begin{bmatrix}
y_2 \\ y_1
\end{bmatrix} \right) > 0,
\end{equation}
which concludes the proof.
\end{proof}

Given $\lambda_1$, $\lambda_2$, and $\Phi$, let
\begin{equation}
\scrZ = \scrZ_{\lambda_1,\lambda_2,\Phi} = \set{z \in \partial \bbD: L_{\lambda_1, \lambda_2, \Phi}(z) = 0}.
\end{equation}

For $\mathcal{Y} = C^\omega,L^2$, we say that an $\SL(2,\bbR)$ cocycle $(\Phi,A)$ is $\mathcal{Y}$\emph{-reducible} to rotations if there exists $B \in \mathcal{Y}(\bbT,\SL(2,\bbR))$ such that
\[ [B(x+\Phi)]^{-1} A(x)B(x) \in \SO(2,\bbR). \]

\begin{lemma} \label{lem:selfdual:L2reducible}
For Lebesgue almost every $z \in \scrZ_{\lambda_1, \lambda_2, \Phi}$, $(\Phi,A_{\lambda_1, \lambda_2, z}^\realified)$ is $L^2$-reducible to rotations.
\end{lemma}

\begin{proof}
This follows from Lemma \ref{lem:AmonotonicInZ} and \cite[Theorem 1.7]{AviKri2015}.
\end{proof}

\begin{lemma} \label{lem:selfdual:Comegareducible}
For Lebesgue almost every $z \in \scrZ_{\lambda_1, \lambda_2, \Phi}$, $(\Phi,A_{\lambda_1, \lambda_2, z}^\realified)$ is $C^\omega$-reducible to rotations.
\end{lemma}

\begin{proof}
This follows from Lemma~\ref{lem:selfdual:L2reducible}, real-analyticity of $A_{\lambda_1,\lambda_2,z}^\realified$, and Avila--Fayad--Krikorian \cite{AviFayKri2011GAFA}; see \cite[Lemma~1.4]{AviFayKri2011GAFA} and its proof.
See also the discussion on \cite[Page~324]{AviJitoMarx2017Invent}. More precisely, \cite{AviFayKri2011GAFA} yields the desired reducibility for a.e.\ \emph{rotation number}. To pass from a.e.\ rotation number to a.e.\ $z \in \scrZ_{\lambda_1, \lambda_2, \Phi}$, one uses the relationship between the rotation number for the Szeg\H{o} cocycle and the density of states for the CMV matrix \cite{GeronimoJohnson1996JDE}, Kotani theory for CMV matrices \cite[Section~10.11]{Simon2005OPUC2}, and the Kotani formula for the derivative of the CMV density of states on the vanishing set of the Lyapunov exponent \cite{FilOng2018PAMS}.

Let us supply some additional details for the reader's convenience. For $\theta \in \bbT$, let $\eta_\theta$ denote the spectral measure  given by
\begin{equation}
\int f \, d\eta_\theta = \langle \delta_0^+, f(W_{\lambda_1,\lambda_2,\Phi,\theta}) \delta_0^+ \rangle, \quad f \in C(\partial \bbD),
\end{equation}
and recall that the density of states measure $\nu = \nu_{\lambda_1, \lambda_2, \Phi}$ associated with the family $\{W_{\lambda_1, \lambda_2, \Phi,\theta}\}_{\theta \in \bbT}$ is given by the average of the spectral measures:
\begin{equation} \label{eq:DOSDEF}
\int f \, d\nu
= \int_\bbT \int f \, d\eta_\theta \, d\theta
= \int_\bbT \langle \delta_0^+, f(W_{\lambda_1,\lambda_2,\Phi,\theta}) \delta_0^+ \rangle \,  d\theta.
\end{equation}
Keeping in mind the equivalence between $W$ and an extended CMV matrix as in Corollary~\ref{coro:UAMOtoCMV}, the result of \cite{FilOng2018PAMS} together with Kotani theory for ergodic CMV matrices \cite[Section~10.11]{Simon2005OPUC2} implies
\begin{equation} \label{eq:zeromeas:dnodphi}
\frac{d\nu}{d\phi}(e^{i\phi})>0
\end{equation}
for Lebesgue a.e.\ $e^{i\phi} \in \scrZ_{\lambda_1, \lambda_2, \Phi}$. Let $\rho$ denote the rotation number of the Szeg\H{o} cocycle. By \cite{GeronimoJohnson1996JDE}, $d\nu/d\phi = 0 \iff d\rho/d\phi =0$, so \eqref{eq:zeromeas:dnodphi} holds for a.e.\ $e^{i\phi} \in \scrZ_{\lambda_1, \lambda_2, \Phi}$ with $\nu$ replaced by $\rho$. This suffices to conclude the argument.
\end{proof}

\begin{lemma} \label{lem:selfdual:subcritical}
For Lebesgue almost every $z \in \scrZ_{\lambda_1, \lambda_2, \Phi}$, the cocycle $(\Phi,A^\realified_{\lambda_1, \lambda_2, z})$ is subcritical.
\end{lemma}

\begin{proof}
This follows from Lemma~\ref{lem:selfdual:Comegareducible} and \eqref{eq:selfdual:realifiedAdef}.
\end{proof}

Of course, we have already shown that subcriticality is \emph{absent} for $\lambda_1 = \lambda_2 \in (0,1)$, so the crucial consequence of Lemma~\ref{lem:selfdual:subcritical} is that $\scrZ_{\lambda,\lambda,\Phi}$ (and hence $\Sigma_{\lambda,\lambda,\Phi})$ has zero Lebesgue measure.

Putting all of this together, we may prove the desired zero-measure Cantor result in the critical regime.

\begin{proof}[Proof of Theorem~\ref{t:maintype}.\ref{t:maintype:criticalCantor}]
The case $\lambda = 1$ was proved in \cite{FOZ2017CMP}, so let us consider $0 < \lambda < 1$. By Proposition~\ref{prop:cocycle:pastur}, the spectrum lacks isolated points. By Theorem~\ref{t:cocycleREG}, we have $\scrZ_{\lambda,\lambda,\Phi} =\Sigma_{\lambda,\lambda,\Phi}$. If the measure of the spectrum is zero, it follows that the spectrum has empty interior and cannot support any absolutely continuous measures. Thus, it suffices to show that the measure of $\Sigma_{\lambda,\lambda,\Phi}$ is zero.

However, this is immediate from Lemma~\ref{lem:selfdual:subcritical} and Theorem~\ref{t:cocycleREG}.\ref{t:cocycleCritical}.
Namely, if the measure of $\Sigma_{\lambda,\lambda,\Phi}$ were positive, Lemma~\ref{lem:selfdual:subcritical} would imply subcriticality of $(\Phi,A_{\lambda,\lambda,z})$ for a positive-measure set of $z \in \partial \bbD$, which contradicts Theorem~\ref{t:cocycleREG}.\ref{t:cocycleCritical}.
\end{proof}

\begin{proof}[Proof of Theorem~\ref{t:maintype}.\ref{t:maintype:critical}]
Since a zero-measure set cannot support absolutely continuous measures, this is an immediate consequence of Theorems~\ref{t:maintype}.\ref{t:maintype:criticalCantor} and \ref{t:contspec:selfdual}.
\end{proof}

%% file: 08-localization.tex

\section{Localization in the Supercritical Region and Consequences} \label{sec:localization}

At last, we conclude by discussing localization. 
Beginning from positivity of the Lyapunov exponent, which follows from the bound
\[ L(z) \geq \log\left[ \frac{\lambda_2(1+\lambda_1')}{\lambda_1(1+\lambda_2')} \right] \]
when $\lambda_2>\lambda_1$, there is a cornucopia of techniques that one could apply in order to prove Anderson localization for $W_{\lambda_1, \lambda_2, \Phi, \theta}$ (for a.e.\ $\theta$).

We note that the non-perturbative localization proof of Bourgain--Goldstein \cite{BourgainGoldstein2000Annals}, generalized to CMV matrices by Wang--Damanik \cite{WangDamanik2019JFA}, suffices for our purposes.
We expect that most of the other localization techniques that have been employed in the study of self-adjoint quasiperiodic operators will find fruitful application here.
For instance, we expect Jitomirskaya's proof of localization for the self-adjoint supercritical almost Mathieu operator with Diohphantine frequency and nonresonant phase~\cite{Jitomirskaya1999Annals} and Avila--Jitomirskaya's proof of almost-localization~\cite{AviJito2010JEMS} can be generalized to the present setting. 
We plan to address this and other finer localization statements in forthcoming work.\footnote{Note added in revision: The first goal was accomplished in \cite{YangFPreprint}.}

\begin{proof}
[Proof of Theorem~\ref{t:maintype}.\ref{t:maintype:super}]
Consider $0<\lambda_1 < \lambda_2 \leq 1$,
and recall from Corollaries~\ref{coro:UAMOtoCMV} and \ref{coro:UAMOtoSzego} that $W_{\lambda_1,\lambda_2,\Phi, \theta}$ is unitarily equivalent to the CMV matrix $\calE^\realified_{\lambda_1,\lambda_2,\Phi, \theta}$ and that the cocycle $A_{\lambda_1,\lambda_2,z}$ is conjugate to the (two-step) Szeg\H{o} cocycle $S_{\lambda_1,\lambda_2,z}$. In particular, Theorem~\ref{thm:posle} gives 
\[L(S_{\lambda_1,\lambda_2,z}) = L(A_{\lambda_1,\lambda_2,z}) \geq \log\left[ \frac{\lambda_2(1+\lambda_1')}{\lambda_1(1+\lambda_2')} \right] >0
\] for all $z \in \partial \bbD$, so the assumptions of \cite[Theorem~1.1]{WangDamanik2019JFA} are met with $\mathcal{I}$ the set of all irrationals and $\mathcal{K} = \partial \bbD$.\footnote{The reader will notice that our Verblunsky coefficients have an alternating quasi-periodic structure, so formally, \cite{WangDamanik2019JFA} does not directly apply. However, passing to blocks of length two, one can re-run the arguments of \cite{WangDamanik2019JFA} with cosmetic changes to deduce the desired localization statement.}
As such, we get for each $\theta$ a full-measure set of $\Phi$ for which $W_{\lambda_1, \lambda_2, \Phi, \theta}$ enjoys Anderson localization.
The conclusion then follows by Fubini's theorem.
\end{proof}

We may thus conclude the proofs of Theorem~\ref{t:maintype}.\ref{t:maintype:sub} by leveraging Theorem~\ref{t:maintype}.\ref{t:maintype:super} and duality.

\begin{proof}[Proof of Theorem~\ref{t:maintype}.\ref{t:maintype:sub}]
This follows from Theorem~\ref{t:maintype}.\ref{t:maintype:super} and Aubry duality via standard arguments, which we describe for the reader's convenience.
Let $\lambda_1>\lambda_2$ be given, and let $\Phi$ be taken from the full-measure set of Diophantine irrationals for which $W_{\lambda_1, \lambda_2, \Phi, \theta}^\aubrydual$ is almost-surely localized.
Define 
\begin{equation} \label{frakEDEF}
\mathfrak{E} = \mathfrak{E}_{\lambda_1,\lambda_2,\Phi}
 = \set{z \in \partial \bbD : \substack{z \text{ is an eigenvalue of } W_{\lambda_1, \lambda_2,\Phi,\theta}^\aubrydual \text{ for some } \theta \\ \text{ whose eigenfunction is exponentially localized} }}.
\end{equation}
By assumption $\mathfrak{E}$ supports the spectral measures of $W_{\lambda_1,\lambda_2, \Phi, \theta}^\aubrydual$ for a.e.\ $\theta$, and hence also supports a.e.\ spectral measure of $W_{\lambda_1,\lambda_2,\Phi,\theta}$ by duality (cf.\ Theorem~\ref{t:aubryOperator}) and general facts about direct integrals which may be found in most texts on operator theory; see, e.g., \cite[Section~XIII.16]{ReedSimon4}. Let us describe this in more detail. Denoting $T = \partial \bbD \setminus \mathfrak{E}$, one has 
$$\chi_T(W_{\lambda_1,\lambda_2, \Phi, \theta}^\aubrydual)=0$$ 
for a.e.\ $\theta$ since $\mathfrak{E}$ supports the spectral measures of $W_{\lambda_1,\lambda_2, \Phi, \theta}^\aubrydual$ for a.e.\ $\theta$. By Theorem~\ref{t:aubryOperator} and the functional calculus for direct integrals, this in turn implies 
$$\chi_T(W_{\lambda_1,\lambda_2, \Phi, \theta})=0$$ for a.e.\ $\theta$ and hence that $\mathfrak{E}$ supports the spectral measures of $W_{\lambda_1,\lambda_2, \Phi, \theta}$ for a.e.\ $\theta$.
On the other hand, for each $z \in \mathfrak{E}$, duality at the level of solutions (i.e., Theorem~\ref{t:aubryViaSolutions}) shows that all solutions to $W_{\lambda_1, \lambda_2, \Phi, \theta}\psi = z\psi$ are bounded, and hence (for a.e.\ phase) the dual model has no subordinate solutions on $\mathfrak{E}$. 
Recalling the gauge equivalence of $W_{\lambda_1,\lambda_2,\Phi, \theta}$ and $\calE^\realified_{\lambda_1,\lambda_2,\Phi, \theta}$ as in Corollary~\ref{coro:UAMOtoCMV}, we note that $W_{\lambda_1,\lambda_2,\Phi, \theta}$ has a subordinate solution at spectral parameter $z \in \partial \bbD$ if and only if $\calE^\realified_{\lambda_1,\lambda_2,\Phi, \theta}$ has a subordinate solution at spectral parameter $z$. Consequently, we may apply subordinacy theory for standard CMV matrices as in \cite{DamGuoOng}. We see that $W_{\lambda_1,\lambda_2,\Phi, \theta}$ has purely absolutely continuous spectrum on $\mathfrak{E}$, hence purely absolutely continuous spectrum, concluding the argument.
\end{proof}

There is another proof of  Theorem~\ref{t:maintype}.\ref{t:maintype:sub} which we now sketch for the interested reader:
\begin{proof}[Sketch of alternate proof of Theorem~\ref{t:maintype}.\ref{t:maintype:sub}]
Let $\nu$ denote the density of states of $\{W_{\lambda_1,\lambda_2,\Phi,\theta}\}$ as in \eqref{eq:DOSDEF}, and let $\nu^\sharp$ denote the density of states associated to the dual model. Since $\mathfrak{E}$ (defined in \eqref{frakEDEF}) supports the spectral measure of the dual model for a.e.\ phase, it supports $\nu^\sharp$. By Aubry duality, $\nu^\sharp=\nu$, so $\mathfrak{E}$ supports $\nu$ and hence the spectral measure of $W_{\lambda_1,\lambda_2,\Phi,\theta}$ for a.e.\ $\theta$. One concludes using subordinacy theory as above.
\end{proof}

%% file: CFO.bbl
\begin{thebibliography}{10}

\bibitem{AhaDavZaq1993PRA}
Y.~Aharonov, L.~Davidovich, and N.~Zagury.
\newblock Quantum random walks.
\newblock {\em Phys. Rev. A}, 48:1687--1690, 1993.

\bibitem{molecules}
A.~Ahlbrecht, A.~Alberti, D.~Meschede, V.~B. Scholz, A.~H. Werner, and R.~F.
  Werner.
\newblock Molecular binding in interacting quantum walks.
\newblock {\em New J. Phys.}, 14:073050, 2012.
\newblock  \href{https://arxiv.org/abs/1105.1051}{{\ttfamily arXiv:1105.1051}}.

\bibitem{SpacetimeRandom}
A.~Ahlbrecht, C.~Cedzich, R.~Matjeschk, V.~B. Scholz, A.~H. Werner, and R.~F.
  Werner.
\newblock Asymptotic behavior of quantum walks with spatio-temporal coin
  fluctuations.
\newblock {\em Quantum Inf. Process.}, 11(5):1219--1249, 2012.
\newblock  \href{https://arxiv.org/abs/1201.4839}{{\ttfamily arXiv:1201.4839}}.

\bibitem{ASW2011JMP}
A.~Ahlbrecht, V.~B. Scholz, and A.~H. Werner.
\newblock Disordered quantum walks in one lattice dimension.
\newblock {\em J. Math. Phys.}, 52(10):102201, 48, 2011.
\newblock  \href{https://arxiv.org/abs/1101.2298}{{\ttfamily arXiv:1101.2298}}.

\bibitem{AVWW2011JMP}
A.~Ahlbrecht, H.~Vogts, A.~H. Werner, and R.~F. Werner.
\newblock Asymptotic evolution of quantum walks with random coin.
\newblock {\em J. Math. Phys.}, 52(4):042201, 36, 2011.
\newblock  \href{https://arxiv.org/abs/1009.2019}{{\ttfamily arXiv:1009.2019}}.

\bibitem{Ambainis}
A.~Ambainis.
\newblock Quantum walk algorithm for element distinctness.
\newblock {\em SIAM J. Comput.}, 37(1):210--239, 2007.
\newblock  \href{https://arxiv.org/abs/quant-ph/0311001}{{\ttfamily
  arXiv:quant-ph/0311001}}.

\bibitem{ambainis2001one}
A.~Ambainis, E.~Bach, A.~Nayak, A.~Vishwanath, and J.~Watrous.
\newblock One-dimensional quantum walks.
\newblock In {\em Proceedings of the thirty-third annual ACM symposium on
  Theory of computing}, pages 37--49. ACM, 2001.

\bibitem{AubryAndre1980}
S.~Aubry and G.~Andr\'{e}.
\newblock Analyticity breaking and {A}nderson localization in incommensurate
  lattices.
\newblock In {\em Group theoretical methods in physics ({P}roc. {E}ighth
  {I}nternat. {C}olloq., {K}iryat {A}navim, 1979)}, volume~3 of {\em Ann.
  Israel Phys. Soc.}, pages 133--164. Hilger, Bristol, 1980.

\bibitem{AvilaARAC2}
A.~Avila.
\newblock Lyapunov exponents, {KAM} and the spectral dichotomy for
  one-frequency {S}chr\"odinger operators.

\bibitem{avila2008point}
A.~Avila.
\newblock On point spectrum with critical coupling.
\newblock {\em preprint}, 2008.

\bibitem{AvilaARAC1}
A.~Avila.
\newblock Almost reducibility and absolute continuity {I}.
\newblock 2010.
\newblock  \href{https://arxiv.org/abs/1006.0704}{{\ttfamily arXiv:1006.0704}}.

\bibitem{Avila2015Acta}
A.~Avila.
\newblock Global theory of one-frequency {S}chr\"{o}dinger operators.
\newblock {\em Acta Math.}, 215(1):1--54, 2015.
\newblock  \href{https://arxiv.org/abs/0905.3902}{{\ttfamily arXiv:0905.3902}}.

\bibitem{AviFayKri2011GAFA}
A.~Avila, B.~Fayad, and R.~Krikorian.
\newblock A {KAM} scheme for {${\rm SL}(2,\mathbb R)$} cocycles with
  {L}iouvillean frequencies.
\newblock {\em Geom. Funct. Anal.}, 21(5):1001--1019, 2011.
\newblock  \href{https://arxiv.org/abs/1001.2878}{{\ttfamily arXiv:1001.2878}}.

\bibitem{AJ2009Ann}
A.~Avila and S.~Jitomirskaya.
\newblock The ten {M}artini problem.
\newblock {\em Ann. Math.}, 170(1):303--342, 2009.
\newblock  \href{https://arxiv.org/abs/math/0503363}{{\ttfamily
  arXiv:math/0503363}}.

\bibitem{AviJito2010JEMS}
A.~Avila and S.~Jitomirskaya.
\newblock Almost localization and almost reducibility.
\newblock {\em J. Eur. Math. Soc.}, 12(1):93--131, 2010.
\newblock  \href{https://arxiv.org/abs/0805.1761}{{\ttfamily arXiv:0805.1761}}.

\bibitem{AviJitoMarx2017Invent}
A.~Avila, S.~Jitomirskaya, and C.~A. Marx.
\newblock Spectral theory of extended {H}arper's model and a question by
  {E}rd{\H{o}}s and {S}zekeres.
\newblock {\em Invent. Math.}, 210(1):283--339, 2017.
\newblock  \href{https://arxiv.org/abs/1602.05111}{{\ttfamily
  arXiv:1602.05111}}.

\bibitem{AJS2014JEMS}
A.~Avila, S.~Jitomirskaya, and C.~Sadel.
\newblock Complex one-frequency cocycles.
\newblock {\em J. Eur. Math. Soc.}, 16(9):1915--1935, 2014.
\newblock  \href{https://arxiv.org/abs/1306.1605}{{\ttfamily arXiv:1306.1605}}.

\bibitem{AviKri2015}
A.~Avila and R.~Krikorian.
\newblock Monotonic cocycles.
\newblock {\em Invent. Math.}, 202:271--331, 2015.
\newblock  \href{https://arxiv.org/abs/1310.0703}{{\ttfamily arXiv:1310.0703}}.

\bibitem{AviYouZho2017DMJ}
A.~Avila, J.~You, and Q.~Zhou.
\newblock Sharp phase transitions for the almost {M}athieu operator.
\newblock {\em Duke Math. J.}, 166(14):2697--2718, 2017.
\newblock  \href{https://arxiv.org/abs/1512.03124}{{\ttfamily
  arXiv:1512.03124}}.

\bibitem{papillon}
J.~Bellissard.
\newblock Le papillon de {H}ofstadter.
\newblock {\em Ast{\'e}risque}, 206:7--39, 1992.

\bibitem{BelSim1982JFA}
J.~Bellissard and B.~Simon.
\newblock Cantor spectrum for the almost {M}athieu equation.
\newblock {\em J. Funct. Anal.}, 48(3):408--419, 1982.

\bibitem{BourgainGoldstein2000Annals}
J.~Bourgain and M.~Goldstein.
\newblock On nonperturbative localization with quasi-periodic potential.
\newblock {\em Ann. Math.}, 152(3):835--879, 2000.
\newblock  \href{https://arxiv.org/abs/math-ph/0011053}{{\ttfamily
  arXiv:math-ph/0011053}}.

\bibitem{BGVW}
J.~Bourgain, F.~A. Gr\"{u}nbaum, L.~Vel\'{a}zquez, and J.~Wilkening.
\newblock Quantum recurrence of a subspace and operator-valued {S}chur
  functions.
\newblock {\em Commun. Math. Phys.}, 329(3):1031--1067, 2014.
\newblock  \href{https://arxiv.org/abs/1302.7286}{{\ttfamily arXiv:1302.7286}}.

\bibitem{AlainUnitaryBandMats}
O.~Bourget, J.~S. Howland, and A.~Joye.
\newblock Spectral analysis of unitary band matrices.
\newblock {\em Commun. Math. Phys.}, 243:191--227, 2003.
\newblock  \href{https://arxiv.org/abs/math-ph/0204016}{{\ttfamily
  arXiv:math-ph/0204016}}.

\bibitem{CGMV2010CPAM}
M.-J. Cantero, F.~A. Gr\"{u}nbaum, L.~Moral, and L.~Vel\'{a}zquez.
\newblock Matrix-valued {S}zeg{\H{o}} polynomials and quantum random walks.
\newblock {\em Commun. Pure Appl. Math.}, 63(4):464--507, 2010.
\newblock  \href{https://arxiv.org/abs/0901.2244}{{\ttfamily arXiv:0901.2244}}.

\bibitem{CGMV2012QIP}
M.-J. Cantero, F.~A. Gr\"{u}nbaum, L.~Moral, and L.~Vel\'{a}zquez.
\newblock The {CGMV} method for quantum walks.
\newblock {\em Quantum Inf. Process.}, 11(5):1149--1192, 2012.

\bibitem{Ti}
C.~Cedzich, , T.~Geib, C.~Stahl, L.~Vel\'azquez, A.~H. Werner, and R.~F.
  Werner.
\newblock Complete homotopy invariants for translation invariant symmetric
  quantum walks on a chain.
\newblock {\em Quantum}, 2:95, 2018.
\newblock  \href{https://arxiv.org/abs/1804.04520}{{\ttfamily
  arXiv:1804.04520}}.

\bibitem{CFGW2020LMP}
C.~Cedzich, J.~Fillman, T.~Geib, and A.~H. Werner.
\newblock Singular continuous {C}antor spectrum for magnetic quantum walks.
\newblock {\em Lett. Math. Phys.}, 110:1141--1158, 2020.
\newblock  \href{https://arxiv.org/abs/1908.09924}{{\ttfamily
  arXiv:1908.09924}}.

\bibitem{CFLOZ}
C.~Cedzich, J.~Fillman, L.~Li, D.~Ong, and Q.~Zhou.
\newblock A unitary mosaic model with exact mobility edges.
\newblock {\em In preparation}, 2023.

\bibitem{WeAreSchur}
C.~Cedzich, T.~Geib, F.~Gr\"unbaum, L.~Vel{\'a}zquez, A.~Werner, and R.~Werner.
\newblock Quantum walks: {S}chur functions meet symmetry protected topological
  phases.
\newblock {\em Commun. Math. Phys.}, 389:31--74, 2022.
\newblock  \href{https://arxiv.org/abs/1903.07494}{{\ttfamily
  arXiv:1903.07494}}.

\bibitem{TopClass}
C.~Cedzich, T.~Geib, F.~A. Gr\"unbaum, C.~Stahl, L.~Vel\'azquez, A.~H. Werner,
  and R.~F. Werner.
\newblock The topological classification of one-dimensional symmetric quantum
  walks.
\newblock {\em Ann. Inst. H. Poincar{\'e}}, 19(2):325--383, 2018.
\newblock  \href{https://arxiv.org/abs/1611.04439}{{\ttfamily
  arXiv:1611.04439}}.

\bibitem{CGWW2019JMP}
C.~Cedzich, T.~Geib, A.~H. Werner, and R.~F. Werner.
\newblock Quantum walks in external gauge fields.
\newblock {\em J. Math. Phys.}, 60(1):012107, 2019.
\newblock  \href{https://arxiv.org/abs/1808.10850}{{\ttfamily
  arXiv:1808.10850}}.

\bibitem{F2W}
C.~Cedzich, T.~Geib, A.~H. Werner, and R.~F. Werner.
\newblock Chiral {F}loquet systems and quantum walks at half period.
\newblock {\em Ann. Inst. H. Poincar{\'e}}, 22(2):375--413, 2021.
\newblock  \href{https://arxiv.org/abs/2006.04634}{{\ttfamily
  arXiv:2006.04634}}.

\bibitem{ewalks}
C.~Cedzich, T.~Ryb\'ar, A.~H. Werner, A.~Alberti, M.~Genske, and R.~F. Werner.
\newblock Propagation of quantum walks in electric fields.
\newblock {\em Phys. Rev. Lett.}, 111:160601, 2013.
\newblock  \href{https://arxiv.org/abs/1302.2081}{{\ttfamily arXiv:1302.2081}}.

\bibitem{locQuasiPer}
C.~Cedzich and A.~H. Werner.
\newblock Anderson localization for electric quantum walks and skew-shift {CMV}
  matrices.
\newblock {\em Commun. Math. Phys.}, 387:1257--1279, 2021.
\newblock  \href{https://arxiv.org/abs/1906.11931}{{\ttfamily
  arXiv:1906.11931}}.

\bibitem{CFKS}
H.~L. Cycon, R.~G. Froese, W.~Kirsch, and B.~Simon.
\newblock {\em Schr\"{o}dinger operators with application to quantum mechanics
  and global geometry}.
\newblock Texts and Monographs in Physics. Springer-Verlag, Berlin, 1987.

\bibitem{Damanik2009surv}
D.~Damanik.
\newblock Almost everything about the {F}ibonacci operator.
\newblock In {\em {XV}th {I}nternational {C}ongress of {M}athematical
  {P}hysics}, pages 149--159. 2009.

\bibitem{DEFHV}
D.~Damanik, J.~Erickson, J.~Fillman, G.~Hinkle, and A.~Vu.
\newblock Quantum intermittency for sparse {CMV} matrices with an application
  to quantum walks on the half-line.
\newblock {\em J. Approx. Theory}, 208:59--84, 2016.
\newblock  \href{https://arxiv.org/abs/1507.02041}{{\ttfamily
  arXiv:1507.02041}}.

\bibitem{DFO2016JMPA}
D.~Damanik, J.~Fillman, and D.~C. Ong.
\newblock Spreading estimates for quantum walks on the integer lattice via
  power-law bounds on transfer matrices.
\newblock {\em J. Math. Pures Appl.}, 105(3):293--341, 2016.
\newblock  \href{https://arxiv.org/abs/1505.07292}{{\ttfamily
  arXiv:1505.07292}}.

\bibitem{DamGuoOng}
D.~Damanik, S.~Guo, and D.~C. Ong.
\newblock Subordinacy theory for extended {CMV} matrices.
\newblock {\em Sci. China Math.}, 64, 2021.
\newblock  \href{https://arxiv.org/abs/2005.04696}{{\ttfamily
  arXiv:2005.04696}}.

\bibitem{Delyon1987JPA}
F.~Delyon.
\newblock Absence of localisation in the almost {M}athieu equation.
\newblock {\em J. Phys. A}, 20(1):L21--L23, 1987.

\bibitem{F2017PAMS}
J.~Fillman.
\newblock Purely singular continuous spectrum for {S}turmian {CMV} matrices via
  strengthened {G}ordon lemmas.
\newblock {\em Proc. Amer. Math. Soc.}, 145(1):225--239, 2017.
\newblock  \href{https://arxiv.org/abs/1507.02044}{{\ttfamily
  arXiv:1507.02044}}.

\bibitem{FilOng2018PAMS}
J.~Fillman and D.~C. Ong.
\newblock A condition for purely absolutely continuous spectrum for {CMV}
  operators using the density of states.
\newblock {\em Proc. Amer. Math. Soc.}, 146(2):571--580, 2018.
\newblock  \href{https://arxiv.org/abs/1612.03208}{{\ttfamily
  arXiv:1612.03208}}.

\bibitem{FOZ2017CMP}
J.~Fillman, D.~C. Ong, and Z.~Zhang.
\newblock Spectral characteristics of the unitary critical almost-{M}athieu
  operator.
\newblock {\em Commun. Math. Phys.}, 351:525--561, 2017.
\newblock  \href{https://arxiv.org/abs/1512.07641}{{\ttfamily
  arXiv:1512.07641}}.

\bibitem{GeronimoJohnson1996JDE}
J.~S. Geronimo and R.~A. Johnson.
\newblock Rotation number associated with difference equations satisfied by
  polynomials orthogonal on the unit circle.
\newblock {\em J. Differ. Equations}, 132(1):140--178, 1996.

\bibitem{GZ2006JAT}
F.~Gesztesy and M.~Zinchenko.
\newblock Weyl-{T}itchmarsh theory for {CMV} operators associated with
  orthogonal polynomials on the unit circle.
\newblock {\em J. Approx. Theory}, 139(1-2):172--213, 2006.
\newblock  \href{https://arxiv.org/abs/math/0501210}{{\ttfamily
  arXiv:math/0501210}}.

\bibitem{grimmett2004weak}
G.~Grimmett, S.~Janson, and P.~F. Scudo.
\newblock Weak limits for quantum random walks.
\newblock {\em Phys. Rev. E}, 69(2):026119, 2004.
\newblock  \href{https://arxiv.org/abs/quant-ph/0309135}{{\ttfamily
  arXiv:quant-ph/0309135}}.

\bibitem{Han2018IMRN}
R.~Han.
\newblock Absence of point spectrum for the self-dual extended {H}arper's
  model.
\newblock {\em Int. Math. Res. Not. IMRN}, 2018(9):2801--2809, 2018.
\newblock  \href{https://arxiv.org/abs/1909.03995}{{\ttfamily
  arXiv:1909.03995}}.

\bibitem{herman}
M.~Herman.
\newblock Une m\'ethode pour minorer les exposants de {L}yapounov et quelques
  exemples montrant le caract\`ere local d'un th\'eor\`eme d' {A}rnold et de
  {M}oser sur le tore de dimension 2.
\newblock {\em Comment. Math. Helv.}, 58:453--502, 1983.

\bibitem{hof76}
D.~R. Hofstadter.
\newblock Energy levels and wave functions of {B}loch electrons in rational and
  irrational magnetic fields.
\newblock {\em Phys. Rev. B}, 14:2239--2249, 1976.

\bibitem{Jitomirskaya1999Annals}
S.~Jitomirskaya.
\newblock Metal-insulator transition for the almost {M}athieu operator.
\newblock {\em Ann. Math.}, 150(3):1159--1175, 1999.
\newblock  \href{https://arxiv.org/abs/math/9911265}{{\ttfamily
  arXiv:math/9911265}}.

\bibitem{Jitomirskaya2021Advances}
S.~Jitomirskaya.
\newblock On point spectrum of critical almost {M}athieu operators.
\newblock {\em Adv. Math.}, 392:107997, 2021.

\bibitem{JitoKociIMRN}
S.~Jitomirskaya and S.~Koci\'{c}.
\newblock Spectral theory of {S}chr\"{o}dinger operators over circle
  diffeomorphisms.
\newblock {\em Int. Math. Res. Not.}, 2022(13):9810--9829, 2021.

\bibitem{JitoLiu2017CPAM}
S.~Jitomirskaya and W.~Liu.
\newblock Arithmetic spectral transitions for the {M}aryland model.
\newblock {\em Commun. Pure Appl. Math.}, 70(6):1025--1051, 2017.
\newblock  \href{https://arxiv.org/abs/1611.10027}{{\ttfamily
  arXiv:1611.10027}}.

\bibitem{JitoLiu2018Annals}
S.~Jitomirskaya and W.~Liu.
\newblock Universal hierarchical structure of quasiperiodic eigenfunctions.
\newblock {\em Ann. Math.}, 187(3):721--776, 2018.
\newblock  \href{https://arxiv.org/abs/1609.08664}{{\ttfamily
  arXiv:1609.08664}}.

\bibitem{JitoMarx2012CMP}
S.~Jitomirskaya and C.~A. Marx.
\newblock Analytic quasi-perodic cocycles with singularities and the {L}yapunov
  exponent of extended {H}arper's model.
\newblock {\em Commun. Math. Phys.}, 316(1):237--267, 2012.
\newblock  \href{https://arxiv.org/abs/1010.0751}{{\ttfamily arXiv:1010.0751}}.

\bibitem{JitoMarx2012CMPErr}
S.~Jitomirskaya and C.~A. Marx.
\newblock Erratum to: {A}nalytic quasi-perodic cocycles with singularities and
  the {L}yapunov exponent of extended {H}arper's model.
\newblock {\em Commun. Math. Phys.}, 317(1):269--271, 2013.

\bibitem{Jitomirskaya1995ICMP}
S.~Y. Jitomirskaya.
\newblock Almost everything about the almost {M}athieu operator. {II}.
\newblock In {\em X{I}th {I}nternational {C}ongress of {M}athematical {P}hysics
  ({P}aris, 1994)}, pages 373--382. Int. Press, Cambridge, MA, 1995.

\bibitem{joye_d_dim_loc}
A.~Joye.
\newblock Dynamical localization for d-dimensional random quantum walks.
\newblock {\em Quantum Inf. Process.}, 11(5):1251--1269, Oct 2012.
\newblock  \href{https://arxiv.org/abs/1201.4759}{{\ttfamily arXiv:1201.4759}}.

\bibitem{Joye_Merkli}
A.~Joye and M.~Merkli.
\newblock Dynamical localization of quantum walks in random environments.
\newblock {\em J. Stat. Phys.}, 140(6):1--29, 2010.
\newblock  \href{https://arxiv.org/abs/1004.4130}{{\ttfamily arXiv:1004.4130}}.

\bibitem{KitaExploring}
T.~Kitagawa, M.~S. Rudner, E.~Berg, and E.~Demler.
\newblock Exploring topological phases with quantum walks.
\newblock {\em Phys. Rev. A}, 82(3):033429, 2010.
\newblock  \href{https://arxiv.org/abs/1003.1729}{{\ttfamily arXiv:1003.1729}}.

\bibitem{KKMS2021}
T.~Komatsu, N.~Konno, H.~Morioka, and E.~Segawa.
\newblock Generalized eigenfunctions for quantum walks via path counting
  approach.
\newblock {\em Rev. Math. Phys.}, 33(06):2150019, 2021.
\newblock  \href{https://arxiv.org/abs/2009.03498}{{\ttfamily
  arXiv:2009.03498}}.

\bibitem{Last1995ICMP}
Y.~Last.
\newblock Almost everything about the almost {M}athieu operator. {I}.
\newblock In {\em X{I}th {I}nternational {C}ongress of {M}athematical {P}hysics
  ({P}aris, 1994)}, pages 366--372. Int. Press, Cambridge, MA, 1995.

\bibitem{Laughlin}
R.~B. Laughlin.
\newblock Quantized {H}all conductivity in two dimensions.
\newblock {\em Phys. Rev. B}, 23(10):5632, 1981.

\bibitem{LiDamZhou2021Preprint}
L.~Li, D.~Damanik, and Q.~Zhou.
\newblock Absolutely continuous spectrum for {CMV} matrices with small
  quasi-periodic {V}erblunsky coefficients.
\newblock {\em Trans. Amer. Math. Soc.}, 375(9):6093--6125, 2022.
\newblock  \href{https://arxiv.org/abs/2102.00586}{{\ttfamily
  arXiv:2102.00586}}.

\bibitem{Linden2009}
N.~Linden and J.~Sharam.
\newblock Inhomogeneous quantum walks.
\newblock {\em Phys. Rev. A}, 80(5):052327, 2009.
\newblock  \href{https://arxiv.org/abs/0906.3692}{{\ttfamily arXiv:0906.3692}}.

\bibitem{Lovett:2010ff}
N.~B. Lovett, S.~Cooper, M.~Everitt, M.~Trevers, and V.~Kendon.
\newblock Universal quantum computation using the discrete-time quantum walk.
\newblock {\em Phys. Rev. A}, 81(4):042330, 2010.
\newblock  \href{https://arxiv.org/abs/0910.1024}{{\ttfamily arXiv:0910.1024}}.

\bibitem{MandelZhito1991CMP}
V.~A. Mandel'shtam and S.~Y. Zhitomirskaya.
\newblock {$1$}{D}-quasiperiodic operators. {L}atent symmetries.
\newblock {\em Commun. Math. Phys.}, 139(3):589--604, 1991.

\bibitem{MarxJito2017ETDS}
C.~A. Marx and S.~Jitomirskaya.
\newblock Dynamics and spectral theory of quasi-periodic {S}chr\"{o}dinger-type
  operators.
\newblock {\em Ergodic Theory Dynam. Systems}, 37(8):2353--2393, 2017.
\newblock  \href{https://arxiv.org/abs/1503.05740}{{\ttfamily
  arXiv:1503.05740}}.

\bibitem{morioka2019detection}
H.~Morioka and E.~Segawa.
\newblock Detection of edge defects by embedded eigenvalues of quantum walks.
\newblock {\em Quantum Inf. Process.}, 18(9):1--18, 2019.
\newblock  \href{https://arxiv.org/abs/1805.11742}{{\ttfamily
  arXiv:1805.11742}}.

\bibitem{Pastur1980CMP}
L.~A. Pastur.
\newblock Spectral properties of disordered systems in the one-body
  approximation.
\newblock {\em Commun. Math. Phys.}, 75(2):179--196, 1980.

\bibitem{portugal2013quantum}
R.~Portugal.
\newblock {\em Quantum walks and search algorithms}.
\newblock Springer, 2013.

\bibitem{Puig2004CMP}
J.~Puig.
\newblock Cantor spectrum for the almost {M}athieu operator.
\newblock {\em Commun. Math. Phys.}, 244(2):297--309, 2004.
\newblock  \href{https://arxiv.org/abs/math-ph/0309004}{{\ttfamily
  math-ph/0309004}}.

\bibitem{ReedSimon4}
M.~Reed and B.~Simon.
\newblock {\em Methods of modern mathematical physics. {IV}. {A}nalysis of
  operators}.
\newblock Academic Press, New York-London, 1978.

\bibitem{SKW}
N.~Shenvi, J.~Kempe, and R.~B. Whaley.
\newblock A quantum walk search algorithm.
\newblock {\em Phys. Rev. A}, 67, 2003.
\newblock  \href{https://arxiv.org/abs/quant-ph/0210064}{{\ttfamily
  arXiv:quant-ph/0210064}}.

\bibitem{Shikano:2010id}
Y.~Shikano and H.~Katsura.
\newblock Localization and fractality in inhomogeneous quantum walks with
  self-duality.
\newblock {\em Phys. Rev. E.}, 82(3):031122, 2010.
\newblock  \href{https://arxiv.org/abs/1004.5394}{{\ttfamily arXiv:1004.5394}}.

\bibitem{Shubin}
M.~A. Shubin.
\newblock Discrete magnetic {L}aplacian.
\newblock {\em Commun. {M}ath. {P}hys.}, 164(2):259--275, 1994.

\bibitem{Simon2005OPUC1}
B.~Simon.
\newblock {\em Orthogonal polynomials on the unit circle. {P}art 1}, volume~54
  of {\em American Mathematical Society Colloquium Publications}.
\newblock American Mathematical Society, Providence, RI, 2005.
\newblock Classical theory.

\bibitem{Simon2005OPUC2}
B.~Simon.
\newblock {\em Orthogonal polynomials on the unit circle. {P}art 2}, volume~54
  of {\em American Mathematical Society Colloquium Publications}.
\newblock American Mathematical Society, Providence, RI, 2005.
\newblock Spectral theory.

\bibitem{TKNN}
D.~J. Thouless, M.~Kohmoto, M.~P. Nightingale, and M.~Den~Nijs.
\newblock Quantized {H}all conductance in a two-dimensional periodic potential.
\newblock {\em Phys. Rev. Lett.}, 49(6):405, 1982.

\bibitem{WangDamanik2019JFA}
F.~Wang and D.~Damanik.
\newblock Anderson localization for quasi-periodic {CMV} matrices and quantum
  walks.
\newblock {\em J. Funct. Anal.}, 276(6):1978--2006, 2019.
\newblock  \href{https://arxiv.org/abs/1804.00301}{{\ttfamily
  arXiv:1804.00301}}.

\bibitem{WeidemannNature}
S.~Weidemann, M.~Kremer, S.~Longhi, and A.~Szameit.
\newblock Topological triple phase transition in non-{Hermitian} {Floquet}
  quasicrystals.
\newblock {\em Nature}, 601(7893):354--359, 2022.

\bibitem{YangFPreprint}
F.~Yang.
\newblock Localization for magnetic quantum walks.
\newblock 2022.
\newblock  \href{https://arxiv.org/abs/2201.05779}{{\ttfamily
  arXiv:2201.05779}}.

\bibitem{Zhang2012Nonlin}
Z.~Zhang.
\newblock Positive {L}yapunov exponents for quasiperiodic {S}zeg{\H{o}}
  cocycles.
\newblock {\em Nonlinearity}, 25(6):1771--1797, 2012.
\newblock  \href{https://arxiv.org/abs/1204.2234}{{\ttfamily arXiv:1204.2234}}.

\end{thebibliography}
